\documentclass[a4paper,10pt, reqno]{amsart}
\usepackage{amsmath,amsthm,latexsym,amscd,amssymb}
\usepackage[all]{xy}
\usepackage{colortbl}
\usepackage{hyperref}

\numberwithin{equation}{section}
\DeclareMathOperator{\spfl}{sf}

\begin{document}

\newtheorem{prop}{Proposition}[section]
\newtheorem{lem}[prop]{Lemma}
\newtheorem{rem}[prop]{Remark}
\newtheorem{theorem}[prop]{Theorem}

\theoremstyle{definition}
\newtheorem{definition}[prop]{Definition}
\newtheorem{cor}[prop]{Corollary}
\newtheorem{que}[prop]{Question}

%\newenvironment{name}[num]{before}{after}

%%%%%    e n v i r o m e n t   f o r   C o m m e n t s :
%%
\newenvironment{com}%
{\vspace{0cm}\  \noindent \%\color{blue}\% }{\%\color{black}\%\vspace{0cm}}
\newenvironment{comch}%
{\vspace{0cm}\  \noindent \%\color{red}\% }{\%\color{black}\%\vspace{0cm}}

\def\no{\noindent}
% quando il footnote e' della stessa dimensione:
\newcommand{\foot}[1]{\footnote{\begin{normalsize}#1\end{normalsize}}}
\newcommand{\spc}{\mathbf A}
\newcommand{\di}{{\rm d}}
\newcommand{\der}{\mathfrak d}

% o p e r a t o r i     m a t e m a t i c i    

%\def\di{\mathop{\rm di}}
\def\dim{\mathop{\rm dim}}
\def\ov{\overline}
\def\Re{\mathop{\rm Re}}
\def\Im{\mathop{\rm Im}}
\def\I{\mathop{\rm I}}
\def\Id{\mathop{\rm Id}}
\def\grad{\mathop{\rm z}\nolimits}
\def\vol{\mathop{\rm vol}}
\def\SU{\mathop{\rm SU}}
\def\SO{\mathop{\rm SO}}
\def\Aut{\mathop{\rm Aut}}
\def\End{\mathop{\rm End}}
\def\GL{\mathop{\rm GL}}
\def\Cinf{\mathop{\mathcal C^{\infty}}}
\def\Ker{\mathop{\rm Ker}}
\def\Coker{\mathop{\rm Coker}}
\def\dom{\mathop{\rm Dom}}
\def\Hom{\mathop{\rm Hom}}
\def\Ch{\mathop{\rm Ch}}
\def\sign{\mathop{\rm sign}\nolimits}
\def\SF{\mathop{\rm SF}}
\def\loc{\mathop{\rm loc}}
\def\AS{\mathop{\rm AS}}
\def\spec{\mathop{\rm spec}}
\def\Ric{\mathop{\rm Ric}}
\def\ch{\mathop{\rm ch}\nolimits}
\def\cs{\mathop{\rm cs}\nolimits}
\def\Ch{\mathop{\rm Ch}}

\def\ev{\mathop{\rm ev}\nolimits}
\def\id{\mathop{\rm id}}
\def\ra{\partial}
\def\ten{\otimes}
\def\fc{{\mathfrak c}}

\def\Cli{\mathbb{C}l(1)}

%                  l e t t e r e     g r e c h e 
\def\Fi{\Phi}
\def\phi{\varphi}
\def\de{\delta}
\def\e{\eta}
\def\ve{\varepsilon}
\def\ep{\epsilon}
\def\ro{\rho}
\def\a{\alpha}
\def\o{\omega}
\def\olott{{\rm w}}
\def\O{\Omega}
\def\b{\beta}
\def\la{\lambda}
\def\th{\theta}
\def\s{\sigma}
\def\t{\tau}
\def\g{\gamma}
\def\D{\Delta}
\def\G{\Gamma}

\def\Q{\mathbin{\mathbb Q}}
\def\R{\mathbin{\mathbb R}}
\def\Z{\mathbin{\mathbb Z}}
\def\bn{\mathbin{\mathbb N}}
\def\Rn{\R^{n}}
\def\C{\mathbb{C}}
\def\Cm{\mathbb{C}^{m}}
\def\Cn{\mathbb{C}^{n}}

%D I R A C :
\newcommand\Di{D\kern-6.5pt/\kern+1pt}
\newcommand\cDi{\mathcal{D}\kern-6.5pt/\kern+1pt}

\def\w{{\mathchoice{\,{\scriptstyle\wedge}\,}{{\scriptstyle\wedge}}
{{\scriptscriptstyle\wedge}}{{\scriptscriptstyle\wedge}}}}
% Calligraphic and bold abbreviations
\def\cA{{\mathcal A}}
\def\Ai{{\mathcal A}_{\infty}}
\def\cO{{\mathcal O}}\def\cT{{\mathcal T}}\def\cU{{\mathcal U}}
\def\cD{{\mathcal D}}\def\cF{{\mathcal F}}\def\cP{{\cal P}}\def\cL{{\mathcal L}}
\def\cB{{\mathcal B}}
\def\cC{{\mathcal C}}
\def\cH{\mathcal H}
\def\cI{{\mathcal I}}
\def\cK{{\mathcal K}}
\def\cN{\mathcal N}
\def\cM{\mathcal M}

%differential algebra
\newcommand{\Ok}{\hat \Omega_k}
\newcommand{\Oi}{\hat \Omega_*}

% N O R M A   D I    U N    V E T T O R E 

\newcommand{\n}[1]{\left\| #1\right\|}% grande palla B di raggio R

%%%%%%%%%%%%%%%%%%%%%%%%%%%%%%%%%%%%%%%%%%%%%%%%%%%%%%%%%%%%%%%%%%%%%%%%%%%%%%%%
%%%%%%%%%%%%%%%%%%%%%%%%%%%%%%%%%%%%%%%%%%%%%%%%%%%%%%%%%%%%%%%%%%%%%%%%%%%%%%%5

\def\Mt{\widetilde{M}}
\def\Et{\tilde{E}}
\def\Vt{\tilde{V}}
\def\Xt{\tilde{X}}
\def\Cd{\mathbb{C}^{d}}

%%%%%%        T R A C C E   %%%%%%%%%%%%%%%
\def\tr{\mathop{\rm tr}\nolimits}\def\tralg{\tr{}^{\text{alg}}}       
\def\TR{\mathop{\rm TR}}\def\trace{\mathop{\rm trace}}
\def\STR{\mathop{\rm STR}}
\def\trG{\mathop{\rm tr_\Gamma}}
\def\TRG{\mathop{\rm TR_\Gamma}}
\def\Tr{\mathop{\rm Tr}\nolimits}
\def\Str{\mathop{\rm Str}}
\def\Cl{\mathop{\rm Cl}}
\def\Op{\mathop{\rm Op}}
\def\supp{\mathop{\rm supp}}
\def\scal{\mathop{\rm scal}}
\def\ind{\mathop{\rm ind}\nolimits}
\def\Ind{\mathop{\mathcal I\rm nd}\,}
\def\Diff{\mathop{\rm Diff}}
\def\T{\mathcal{T}}

\def\cgs{C^{*}(\Gamma,\sigma)}
\def\bcgs{C^{*}(\Gamma,\bar{\sigma})}
\def\cgsr{C^{*}_{r}(\Gamma,\sigma)}
\def\Et{\widetilde{E}}
\def\Wt{\widetilde{W}}
\def\Xt{\tilde{X}}
\def\N{\mathcal{V}}
\def\Nbs{\N^{-c}}
\def\Ns{\N^c}

\def\etast{\eta_{\tau_s}}
\def\rhoct{\rho^{\mathbf c}_{\tau_s}}
\def\rhost{\rho^{\s_s}_{\tau_s}}
\newcommand{\question}[1]{\color{blue}\vspace{5 mm}\par \noindent
\marginpar{\color{blue}\textsc{Q }} \framebox{\begin{minipage}[c]{0.95
\textwidth} \raggedright \tt #1 \end{minipage}}\vspace{5 mm}\color{black}
\par}
\newcommand{\Komm}[1]{\color{blue}\vspace{5 mm}\par \noindent 
\marginpar{\color{blue} Comment} \framebox{\begin{minipage}[c]{0.95
\textwidth} \raggedright \tt #1 \end{minipage}}\vspace{5 mm}\color{black}\par}

%% AMBIENTI  PER  commenti

\setlength{\marginparwidth}{1.2in}
\let\oldmarginpar\marginpar
\renewcommand\marginpar[1]{\-\oldmarginpar[\raggedleft\footnotesize #1]%
{\raggedright\footnotesize #1}}

%This is some other text.\marginpar{This is the new marginpar command.
%This is the new marginpar command.}

\newenvironment{Sara}
{ \marginpar{\color{blue} Sara $\top$} 
\color{blue} }{
\marginpar{\color{blue}\hspace{.7cm} $\bot$}
\color{black}}

 \title{Two-cocycle twists and Atiyah--Patodi--Singer index theory}
 \author{Sara Azzali}
\address{Sara Azzali\\   Institut f\"ur Mathematik\\
Universit\"at Potsdam\\\newline Am Neuen Palais, 10 \\14465 Potsdam
              \\ Germany}
\email{azzali@uni-potsdam.de}
\thanks{}

\author{Charlotte Wahl}
\address{Charlotte Wahl\\ Leibniz-Archiv\\Waterloostr. 8\\ 30169 Hannover\\ Germany}
\email{wahlcharlotte@googlemail.com}
\thanks{}

\begin{abstract}
We construct $\eta$- and $\rho$-invariants for Dirac operators, on the universal covering of a closed manifold, that are invariant under the projective action associated to a $2$-cocycle of the fundamental group. 
 We prove an Atiyah--Patodi--Singer index theorem in this setting, as well as its higher generalization. Applications concern the classification of positive scalar curvature metrics on closed spin manifolds. We also investigate the properties of these twisted invariants for the signature operator and the relation to the higher invariants.
\end{abstract}

 \maketitle
 \tableofcontents
\section{Introduction}

The purpose of this paper is to construct and investigate $\eta$- and $\rho$-invariants for Dirac operators, on the universal covering of a closed manifold, that are invariant under the projective action associated to a $2$-cocycle of the fundamental group, and prove an Atiyah--Patodi--Singer index theorem in this setting, which we quickly recall. 

Let  $\Mt$ be the universal covering of a closed manifold $M$ with $\pi_1(M)=\G$, and $\Et\to \Mt$ be the lift of a vector bundle $E\to M$. 
A \emph{projective action} $T$ of $\G$ on $L^2(\Mt, \Et)$ is by definition a map $\G \to \cU(L^2(\Mt,\Et)),~ \g\mapsto T_\g$ satisfying  $$T_\g\circ T_{\g'}=\s(\g,\g')T_{\g\g'}$$ where  $\s(\g,\g')\in  U(1)$ and is actually a group 2-cocycle. $T$ is therefore called a \emph{$(\G, \s)$-action}.

The bounded operators on $L^2(\Mt, \Et)$ that commute with $T$ form a von Neumann algebra $B(L^2(\Mt, \Et))^T$, and can be expressed as a von Neumann algebraic tensor product
$$
B(L^2(\Mt, \Et))^{T}\cong \mathcal N(\G,\s)\otimes B(L^2(\mathcal F, \Et_{|\mathcal F}))
$$
where $\mathcal F$ is a fundamental domain, and $ \mathcal N(\G,\s)$ is the \emph{twisted von Neumann algebra of $\G$}, \emph{i.\,e.} defined through the twisted convolution product.

There is a semifinite trace on this algebra generalizing Atiyah's $L^2$-trace, which enables one to do twisted $L^2$-index theory.

The study of projectively invariant operators was initiated by Gromov, who considered a hermitian line bundle with unitary connection $(E,\nabla)$ on the universal covering of a closed manifold such that the fundamental group does not act on it but its curvature is invariant. He generalized Atiyah's $L^2$-index theorem to this setting \cite[\S 2.3.B]{Gr}.
%proof in \cite[Thm. 3.6]{Ma1} and \cite[Ch.10]{pan}). 
 
Projectively invariant elliptic operators appear in the mathematical description of magnetic fields, in particular in some models of the fractional quantum Hall effect (see for example \cite{mm1, mm} and references therein).
In this context, the spectral theory of such operators is often referred to as \emph{noncommutative Bloch theory}, alluding to the fact that the classical Bloch theory studies operators which are invariant under the action of an abelian group, while here the group is arbitrary \cite{Gu}.

From the geometric point of view, these operators give interesting invariants analogous to those studied in $L^2$-index theory, or more generally higher index theory. For instance, twisted $L^2$-signatures and $\hat A$-genera are linear combinations of their higher analogues \cite{Ma1,Ma2}. 

Exploiting this and following an idea of Gromov \cite[\S 9$\frac{1}{7}$]{Gr1}, Mathai obtained a twisted analogue of the Mishchenko--Kasparov approach to the Novikov conjecture: in \cite{Ma2} he proves the conjecture for all \emph{low degree cohomology classes}, \emph{i.\,e.} those in the subring of $H^*(B\G,\Q)$ generated by cohomology classes of degree less or equal than $2$. 
The same technique yields a vanishing result for the higher $\hat A$-genus associated with a low degree cohomology class on a closed spin manifold admitting a metric of scalar curvature \cite{Ma1}.

The  strategy is the following: Starting (for simplicity) with $\mathbf{c}=[c] \in H^2(\G,\Q)\cong H^2(B\G,\Q)$, one considers twists $\s^{s}=e^{isc}$ parametrized by $s \in \Q$. The projectively invariant operators one deals with can be interpreted as Dirac operators twisted by a Mishchenko type bundle whose fiber is the twisted group $C^*$-algebra $C^*(\G,\s^s)$. The curvature of the bundle depends linearly on $s$. In particular, the twisted spin Dirac operator of a manifold with positive scalar curvature is invertible for $|s|$ small enough. Its index is a polynomial in $s$ whose coefficients are higher $\hat A$-genera. This leads to the vanishing result mentioned above.

Marcolli and Mathai studied higher twisted index theory on covering spaces of good orbifolds,  generalizing the higher index theorem of Connes and Moscovici \cite{mm}. This is used in \cite{mm1} to compute the range of the trace on the $K$-theory of certain twisted group algebras. As a further application, in \cite{mm} certain ranges of higher cyclic traces are computed, which are relevant in the model of the quantum Hall effect.

\medskip 

So far only primary twisted invariants have been considered. Our goal is to define and study secondary invariants, in particular twisted $\eta$- and $\rho$-invariants. (Note that the sense of ``twisted'' here is different from the one in \cite{bm1,bm2,bm3}, also dealing with secondary invariants, and in general from papers concerned with twisted $K$-theory.)

To do so, we sharpen some of the tools developed by Mathai, with particular attention to the dependence on the choices involved, and to the extension to manifolds with boundary.

First we prove a general von Neumann algebraic Atiyah--Patodi--Singer index theorem of which the twisted and the classical $L^2$-Atiyah--Patodi--Singer index theorems are special cases (\S \ref{neumann_Cstar}). Here we use a different strategy than Ramachandran \cite{r}, namely we derive our theorem from the $C^*$-algebraic Atiyah--Patodi--Singer index theorem \cite{ps1}.

Using this formalism we can associate well-defined twisted $\eta$-invariants to any family of positive traces $\tau_s$ on $C^*(\G,\s^s)$ (\S \ref{etarho}). However, we need to assume that the traces are invariant under automorphisms induced by characters on $\G$, i.\,e. automorphisms of the form $\delta_\g \mapsto \chi(\g)\delta_\g$. 

Our main results concern spin manifolds with positive scalar curvature. Here twisted $\rho$-invariants $\rhoct$ associated to the spin Dirac operator can be defined for any invariant delocalized trace, for $s$ small enough. Therefore we consider  $\rhoct$ as a germ at $s=0$. 

We define a new equivalence relation of metrics of positive scalar curvature on a fixed closed spin manifold $M$ with a map $f\colon M \to B\G$ classifying the universal covering (\S \ref{bordism_inv}). This relation, called strong $\G$-bordism, is stronger than $\G$-bordism but weaker than concordance. In high dimensions it is equivalent to $\G$-bordism. We show, by applying the $C^*$-algebraic index theorem, that $\rhoct$ is invariant under strong $\G$-bordism.

Following ideas of Botvinnik--Gilkey and Leichtnam--Piazza \cite{BG,lpetapos} and using a product formula for $\rhoct$, we prove that manifolds whose fundamental groups have torsion and a particular product structure have infinitely many $\G$-bordism classes of positive scalar curvature metrics, if they admit at least one (\S \ref{app}). As in \cite{lpetapos} our results apply to even and odd dimensions; on the other hand, our techniques are complementary, in the following sense:
the techniques of \cite{lpetapos} work for Gromov hyperbolic groups (by using the most general version of the higher Atiyah--Patodi--Singer index theorem proved in \cite{lptwist}), while we do not need any growth condition, but we have an additional algebraic assumption, namely we ask that the abelization is finite.

Furthermore we show that $\rhoct$ is trivial for torsion free groups fulfilling the maximal Baum--Connes conjecture by combining methods of Piazza--Schick \cite{ps1} and Hanke--Schick \cite{HaS}. 

These results show a strong analogy between our invariants and the higher invariants introduced by Lott \cite{lohigheta}; indeed we show that under some assumptions on $\Gamma$, the twisted eta invariant can be expressed by higher eta invariants in the same way as it holds for the index class (\S \ref{sec10}).
For general groups this remains an open question.

We also establish some elementary facts about twisted spectral invariants of the signature operator  (\S \ref{sign}). In particular we show that the spectral flow of a path of twisted signature operators with varying metric may be nonzero, in contrast to the $L^2$-case \cite{aw}. Thus, the standard method for showing that $\rho$-invariants are metric independent fails here. Furthermore, a straightforward definition of twisted signatures for manifolds with boundary would not have the required invariance properties. 

Last but not least, we consider the higher situation (\S \ref{high_twist_APS}), where $\tau_s$ is not a trace but a densely defined cyclic cocycle. From the general noncommutative Atiyah--Patodi--index theorem proven in \cite{waAPS}, we derive a higher twisted Atiyah--Patodi--Singer theorem generalizing the higher Atiyah--Patodi--Singer index theorem of Leichtnam and Piazza on the one hand \cite{lphighAPS} and the higher twisted Atiyah--Singer index theorem by Marcolli and Mathai \cite{mm} on the other hand. Applications in the higher case are analogous to those of the twisted $\rho$-invariants and are not given in detail. In the appendix we study the extendability of cyclic cocycles to holomorphically closed subalgebras in the twisted case.

There remain many open questions: the relation of the twisted $\eta$-invariant to higher $\eta$-invariants for general groups or, for solvmanifolds, to the $\eta$-invariants studied by Marcolli in \cite{marc}; furthermore the properties of $\rhoct$ and other twisted invariants for the signature operator. It would also be interesting to have a twisted version of torsion invariants. 

\medskip 

\paragraph*{\textbf{Acknowledgements}}
The problem of defining $\eta$- and $\rho$-invariants for $2$-cocycle twists arose in the thesis of the first named author, and was the starting point of our research. We thank Paolo Piazza for drawing our attention to the subject and for helpful conversations. We also thank Thomas Schick for interesting discussions, in particular for comments clarifying the relationship between $\G$-bordism and strong $\G$-bordism, and the \emph{\'Equipe Alg\`ebres d'Operateurs} of Paris Diderot for their kind hospitality.

\section{Multipliers, twisted group algebras}
\label{multipliers}

We start by reviewing the constructions of twisted group algebras.

Let $\G$ be a discrete group. A \emph{(unitary)
projective representation} of $\G$ on a Hilbert space $H$ is a map $T\colon\G
\rightarrow \cU(H)$, $\g\mapsto T_\g$ such that there is $\s \colon\Gamma\times
\Gamma\rightarrow U(1)$ with
\begin{equation*}
\left\{\label{proj.repr}
\begin{array}{ll}
    T_e=1& \hbox{} \\
    T_\g\circ T_\mu=\s(\g,\mu)T_{\g\cdot \mu}   & \hbox{} \\
\end{array}%
\right. \end{equation*}
The map $\s$ is called a \emph{multiplier} of $\G$, it automatically satisfies
$$
\left\{%
\begin{array}{ll}
    \sigma(e,\gamma)=\sigma(\gamma,e)=1,& \hbox{} \\
    \sigma(\g_{1}\g_{2},\g_{3})\sigma(\g_{1},\g_{2})=\sigma(\g_{1},\g_{2}\g_{3})\sigma(\g_{2},\g_{3}) & \hbox{} \\
\end{array}
\right. $$ 
and is actually a group $2$-cocycle of $\G$ with coefficients in $U(1)$. Hence it defines an element $[\s]\in H^2(\G,U(1))$. 

\medskip

Recall that two cocycles $\s, \s'$ are cohomologous if there is $z\colon \G \to U(1)$ such that
$$\s'(\g,\mu)=\s(\g,\mu)(\partial z)(\g,\mu)$$
with $$\partial z(\g,\mu):=z(\g)z(\mu)z(\g\mu)^{-1} \ .$$
Automatically, $z(e)=1$.

Observe that
$\bar{\s}(\g,\mu):=\overline{\sigma(\g,\mu)}$ is still a
multiplier, and $\s\bar{\s}=1$.

\medskip
The projective representation $T$ is also called a \emph{$(\Gamma,\s)$-representation}. 
Two projective representations $T$ and
$T'$ are \emph{projectively equivalent} if there exists
$z \colon\G\rightarrow U(1)$ with $z(e)=1$ and an unitary isomorphism $f \colon\mathcal H_{1}\rightarrow
\mathcal H_{2}$ such that $T_\g=z(\g)f^{-1}T'_\g f$; two projective representations $T$
and $T'$ are \emph{linearly equivalent} if $z\equiv 1$. If
$T$ and $T'$ are projectively equivalent, then 
\begin{eqnarray*}
T_{\g}\circ T_{\mu}&=&z(\g)z(\mu)f^{-1}T'_{\g}\circ T'_\mu f=z(\g)z(\mu)\sigma'(\g,\mu)f^{-1}T'_{\g \mu}f\\
&=&z(\g)z(\mu)z(\g \mu)^{-1} \sigma'(\g,\mu)T_{\g \mu} \ .
\end{eqnarray*}
Thus the corresponding multipliers $\sigma, \sigma'$ are cohomologous via $\s=\s'\partial z$.

\medskip
On the vector space $\C\Gamma$ one can consider the twisted convolution
product
%$$
%\C(\Gamma,\sigma)= \{f:\Gamma\rightarrow \mathbb{C}\: ,  f \text{
%with finite support}\}
%$$
$$
(f\ast g)(\gamma):=\sum_{\g_{1}\g_{2}=\g}
f(\g_{1})g(\g_{2})\sigma(\g_{1},\g_{2})
$$
The operation $*$ is associative since $\sigma$ is a
$2$-cocycle.   Let $\delta_\g \in \C\Gamma$ be given as usual by 
$\de_{\g}(\mu)=1$, if $\mu =\g$, and $\de_{\g}(\mu)=  0$ if $ \mu\neq \g$. 
The twisted involution is defined by $(\delta_{\g})^*=\bar\s(\g,\g^{-1})\delta_{\g^{-1}}$.

We denote this algebra by $\C(\G,\s)$, and
call it as usual the \emph{twisted group algebra}. We denote the
convolution in $\mathbb{C}(\Gamma,\bar{\sigma})$ by the symbol
$\bar{*}$.

The involutive algebra $\mathbb{C}(\Gamma,\sigma)$ has unit $\de_
{e}$.

In general it holds that $\de_{\g}*\de_{\mu}=\sigma(\g,\mu)\de_{\g\mu}$.

\begin{rem} As a vector space,
$\mathbb{C}(\Gamma,\sigma)$ is generated by the functions
$\{\de_{\g}\;,\;\g \in \Gamma\}$ and it is dense in the Banach
$*$-algebra $L^{1}(\Gamma, \sigma)=\{f\colon \G\rightarrow \mathbb{C} \:
|\sum_{\G}|f(\g)| <\infty\}$ with twisted convolution defined
above.
\end{rem}
\begin{lem}
\label{extLemma}
Every $(\G,\s)$-representation on
a Hilbert space
$$
T\colon (\G,\s)\rightarrow \mathcal{U}(H)
$$
extends to a classical representation $\tilde{T}$ of the involutive algebra
$\C(\Gamma, \sigma)$.
\end{lem}
\begin{proof}
For $f \in \C(\Gamma, \sigma)$ define $\tilde{T}(f)=\sum_{\G}f(\g)T_\g\:\: \in \cB(H) $. The map 
$\tilde{T}\colon \C(\Gamma,
\sigma)\rightarrow\cB(H)$ is an involutive algebra homomorphism. In fact,
\begin{equation*}
\tilde{T}(f* g)=\sum_{\G}(f*
g)(\g)T_\g
=\sum_{\g\in\G}\left(\sum_{\mu\in\G}
f(\mu)g(\mu^{-1}\g)\s(\mu,\mu^{-1}\g)\right) T_\g
\end{equation*}
On the other hand, for $v\in H$
\begin{multline*}
\tilde{T}(f)\tilde{T}(g)v=
\tilde{T}(f)\left(\sum_{\g\in\G}g(\g)T_\g v \right)
=\sum_{\mu\in\G}\left(\sum_{\g\in\G}g(\g)T_\g v
\right)f(\mu)T_\mu=\\
=\sum_{\mu\in\G}\sum_{\g\in\G} f(\mu) g(\g)T_\mu T_\g v 
=\sum_{\mu\in\G}\sum_{\g\in\G}\s(\mu,\g)f(\mu)g(\g)T_{\mu\g} v
=\\
=\sum_{\mu\in\G}\sum_{\tau\in\G}\s(\mu,\mu^{-1}\tau)f(\mu)g(\mu^{-1}\tau)T_\tau v.
\end{multline*}
Hence $\tilde{T}(f*
g)(v)=\tilde{T}(f)\tilde{T}(g)v$. Moreover, the map $\tilde T $ is compatible with the involution:  $\tilde T ((\delta_\g)^*)=\bar \s(\g, \g^{-1})T_{\g^{-1}}$, and on the other hand $T_\g^*=T_{\g^{-1}}=\s(\g,\g^{-1})^{-1}T_{\g^{-1}}= \bar \s(\g, \g^{-1})T_{\g^{-1}}$. 
\end{proof}

\subsection{Maximal and reduced twisted $C^{*}$-algebras} 
\begin{definition}
On $\C(\Gamma, \sigma) $ define the seminorm
$$\n{f}':= \sup\{\n{\pi(f)}\: | ~ \pi \colon\C(\Gamma,
\sigma)\rightarrow \cB(H_{\pi}) \: \text{representation}\} \ ,$$ which
is actually a norm. The completion $$ \overline{\C(\Gamma,
\sigma)}^{\n{\;\:}'}=:\cgs
$$is called the \emph{maximal twisted $C^{*}$-algebra}.
\end{definition}

\begin{prop}\label{esten.repr}
Every projective representation $T\colon (\G,\s)\rightarrow
\mathcal{U}(H)$ extends to a representation of the maximal
twisted $C^{*}$-algebra
$$
\tilde{T} \colon C^{*}(\G,\s)\rightarrow \cB(H)
$$
\end{prop}
\begin{proof}
By Lemma \ref{extLemma}, $T$ induces a representation $\tilde T$ of $\C(\G,\s)$. Moreover
$\cgs$ is the enveloping $C^{*}$-algebra of the Banach $*$-algebra
$L^{1}(\Gamma, \sigma)$, hence $\tilde T$ extends to $\cgs$ by
construction.
\end{proof}

Now we consider the dependence of $\cgs$ on $\s$.

Let $\sigma', \sigma$ be cohomologous, \emph{i.e.} $\s'=\s\partial z$. Note that $z\colon \G \to U(1)$ is uniquely determined up to multiplication by a group homomorphism $\G \to U(1)$. If $\s'=\s$, then $z$ is a group homomorphism.

We can define an isomorphism $$b_z \colon C^*(\G,\s')\rightarrow\cgs, ~ 
\delta_\g\mapsto z(\g)\delta_\g \ .$$
It is compatible with the product since 
$$b_z(\delta_{\g}* \delta_{\mu})=b_z(\s'(\g,\mu)\delta_{\g \mu})=\s'(\g,\mu)z(\g \mu)\delta_{\g \mu}$$
and, on the other hand,
$$
b_z(\delta_{\g})* b_z(\delta_{\mu})=z(\g)z(\mu)\s(\g,\mu)\delta_{\g \mu}=\s(\g,\mu)\partial z(\g,\mu) z(\g,\mu)\delta_{\g \mu} \ .
$$

This leads to the following definition:

\begin{definition}
\label{projiso}
Let $\s, \s'$ be multipliers on $\G$.
\begin{enumerate}
\item
Let $z\colon \G \to U(1)$ be such that $b_z\colon C^*(\G,\s')\rightarrow\cgs,~\delta_\g\mapsto z(\g)\delta_\g$ is an isomorphism. We call $b_z$ a \emph{projective isomorphism}.

\item
Let $\chi\colon \G \to U(1)$ be a group homomorphism. We call the induced automorphism  $b_{\chi}:\cgs \to \cgs, ~ \delta_g \mapsto \chi(g)\delta_g$ a \emph{projective automorphism}.
\end{enumerate}
\end{definition}

The above calculation implies that $z, \s, \s'$ as in (1) automatically fulfill $\s'(\g,\mu)=\s(\g,\mu)\partial z(\g,\mu)$.

Thus, the multipliers $\s, \s'$ are cohomologous if and only if there is a projective isomorphism between $C^*(\G,\s')$ and $\cgs$. It is uniquely determined up to composition by a projective automorphism of $\cgs$.

\begin{rem}
Note that (contrary to what is claimed in the proof of \cite[Lemma 1.4]{marange}) a projective automorphism does not necessarily induce the identity in $K$-theory: 

Consider $\G=\Z/2\Z$ and the nontrivial group homomorphism $\chi\colon \G \to U(1)$ with $\chi(\bar 1)=-1$. The complex representation ring $R(\G)$ is generated by the trivial representation and the representation induced by $\chi$. Composition with $b_{\chi}$ interchanges the associated representations of the group algebra. Hence it induces a nontrivial automorphism on $K_0(C^*\G)\cong R(\G)$. 

This also leads to an example for $K_1$: Consider the group $\G=\Z \times \Z/2\Z$. Since $$K_1(C^*\G) \cong K_1(C(S^1) \ten C^*(\Z/2\Z)) \cong K_1(C^*(\Z/2\Z)) \oplus K_0(C^*(\Z/2\Z)) \ ,$$ the above projective automorphism induces a nontrivial automorphism on $K_1(C^*\G)$.
\end{rem}

\subsubsection{Reduced twisted $C^*$-algebra} Consider the
projective representation $T \colon(\G,\s) \rightarrow \cB (l^{2}(\G))$ given
by
$$
(T_{\g}f)\,(x)=f(\g^{-1}x)\s(\g,\g^{-1}x)=(\de_{\g}*
f)(x) \ .
$$
It fulfills  $T_{\g}T_{\mu}=\s(\g,\mu)T_{\g\mu}$. Thus $T$ extends
to the left regular representation $\widetilde{T} \colon\C(\Gamma, \sigma)\rightarrow \cB (l^{2}(\G))$, which is given for $\phi \in \C(\Gamma,
\sigma)$ by 
$$\widetilde{T}_{\phi}(f)(x)=\sum_{\g\in\G}\phi(\g)(T_{\g}f)(x)
=
$$
$$
=\sum_{\g\in\G} \phi(\g) f(\g^{-1}x)\s(\g,\g^{-1}x)
=(\phi* f)(x).
$$
$\widetilde{T}$ is injective since the algebra contains $\de_{e}$.

The
reduced twisted $C^{*}$-algebra is now defined as the completion of
the image of $\C(\G,\s)$ in $\cB (l^{2}(\G))$, \emph{i.\,e.}
$$
\cgsr= \overline{\widetilde T(\C(\G,\s))}^{\n{\;\:}_{\cB}}\ .
$$

\section{Mathai's constructions for manifolds with boundary}
\label{mathai}

From now on we assume that $\G$ is finitely generated. Let $c\in H^{2}(B\G,\Q)$ be a fixed cohomology class.

Recall that Mathai's constructions in
\cite{Ma1} and \cite[\S\S 1-2]{Ma2} produce for any closed connected manifold $M$ with fundamental group $\G$
\begin{enumerate}
  \item  a \emph{geometrically defined multiplier $\s$}, and, more generally, a family $\s^s,~s \in \R,$ with $\s^0=1$ and $\s^1=\s$;
    \item a $(\G,\bar\s)$-action on every $\G$-invariant bundle $\Et\rightarrow \Mt$ on the universal covering $\Mt$ of $M$;
  \item a canonical $C^*(\G,\s)$-vector bundle $\mathcal V^\s\rightarrow M$
   whose isomorphism class only depends on $c$;
  \item by passing to the cocycle $\s^s$, a connection on $\mathcal V^{\s^s}$ with small curvature for $s$ small. 
\end{enumerate}

In the following we adapt these constructions to manifolds with boundary. We give a detailed exposition since for our purposes it is more critical than in Mathai's situation to understand precisely how $\s^s$ and the bundle with connection $\mathcal V^{\s^s}$ depend on the choices made on the way.

\medskip

Let $p \colon\widetilde{M}\rightarrow M$ be the
universal covering of a compact connected Riemannian manifold $M$ (possibly with boundary).  We assume that the restriction $p|_{\ra \widetilde M}\colon \ra \widetilde M \to \ra M$ is also the universal covering. Let $f \colon M\rightarrow B\G$ be a classifying map realizing
$$
\xymatrix{
\widetilde{M} \ar[r]^{\tilde{f}} \ar[d]^p & E\G\ar[d]^{\pi} \\
M\ar[r]_f  & B\G}
$$
If $\partial M\neq\emptyset$, we may assume that $f$ does not depend on the normal variable in a neighborhood of $\partial M$. 

\begin{lem}\label{eta}
There is $\eta \in \O^1(\widetilde{M},\R)$ such that $d\eta$ is $\G$-invariant and such that $\o \in \O^2(M,\R)$ given by $d\eta=p^*\o$ fulfills $-[\frac{\omega}{2\pi}]=f^*c$ in $H^2(M,\R)$. If $M$ is a manifold with boundary, then $\eta$ can be chosen such that the restriction to a collar neighborhood of the boundary is the pull back of some $\eta_1\in \O^1(\partial \Mt)$.
\end{lem}

\begin{proof}
We can argue as in \cite[p. 3]{HaS}. For $N \in \mathbb N$ large enough, $Nc$ lifts to a class in $H^2(B\G,\Z)$. Thus, there is a line bundle $\mathcal L\rightarrow B\G$ whose first Chern class is $Nc$. Endow the bundle $\mathcal{L}':=f^*\mathcal{L}$ with a hermitian metric and a unitary connection $\nabla$. Let $i\o$ be its curvature form. It holds that $c_1(\mathcal{L}')=[\frac{-\omega}{2\pi}]=f^*(Nc)$ in $H^2(M,\R)$. Since $E\G$ is contractible, the line bundle $\pi^*\mathcal{L}\rightarrow E\G$ is trivial and thus the bundle $p^*\mathcal{L}'$ is trivial as well. After fixing a trivialization, the lift of $\nabla$ to $\widetilde M$ is given by $d+i\eta$, for some $\eta\in \O^1(\widetilde{M},\R)$. Its curvature equals $id\eta=ip^*\o$. Thus the forms $\frac 1N \eta$ and $\frac 1N \o$ fulfill the claim.

If $\partial M\neq \emptyset$, we choose a connection on $\mathcal{L}'$ and a trivialization of $p^*\mathcal{L}'$ which are of product type near the boundary. Then the form $\eta$ has the desired property.
\end{proof}

Of course, the factor $-2\pi$ in the statement of the Lemma is not relevant. It comes from the interpretation of $i\o$ as the curvature of a line bundle, which will be important throughout the paper. 
\medskip

For consistency with \cite{Ma1,Ma2} we consider the left action of $\G$ on $\widetilde M$ by setting, as usual, $\g x:=x \g^{-1}$. In the following, the notation refers to the left action if it is not clear otherwise.

\medskip

Let the situation be as in the previous Lemma. We set $\tilde{\o}=p^{*}\o$. Then for all $\g\in \G$
$$
0=\g^{*}\tilde{\o}-\tilde{\o}=d(\g^{*}\e-\e) \ .$$ 
Since
$\widetilde{M}$ is simply connected, the closed $1$-form $\g^{*}\e-\e$
is exact. Thus we may choose $\psi_{\g}\colon\widetilde{M}\rightarrow \R$ such
that $d\psi_{\g}=\g^{*}\e-\e$. We assume that $\psi_{e}=0$.

Hence for fixed $\g, \mu \in \G$
\begin{equation*}
   \psi_{\mu}(x)+\psi_{\g}(\mu
   x)-\psi_{\g\mu}(x)=\text{constant}.
\end{equation*}

We choose $x_0 \in \widetilde M$ and define $\s\colon \G\times \G\to U(1)$ by
\begin{equation}\label{c-sigma} 
\s(\g,\mu):=\exp\bigl(i (\psi_{\mu}(x_0)+\psi_{\g}(\mu
   x_0)-\psi_{\g\mu}(x_0))\bigr) \ .
\end{equation}
From the previous equation it follows that $\s$ is a multiplier independent of $x_0$. 

Note that in \cite{Ma1,Ma2} the functions $\psi_{\g}$ were normalized by $\psi_{\g}(x_0)=0$. Then
\begin{equation*}
  \sigma(\g,\mu)=\exp(i (\psi_{\g}(\mu x_0)) \ .
\end{equation*}
In the following, we do not assume this normalization since the additional flexibility simplifies some calculations. 

\begin{lem}
\textrm{\cite[\S\S 1.3-1.4]{Ma2}} Different choices of $\psi_{\g}, \e$ and $\o$ in the above construction give
cohomologous multipliers in \eqref{c-sigma}. 
\label{lemma-scelte}
\end{lem}
\begin{proof}
We proceed in three steps.
\begin{enumerate}
\item  First we consider $\eta$ fixed. Thus let $\psi_\g'$ be as above with $d\psi_{\g}'=\g^{*}\e-\e$. Note that $a_{\g}=\psi_{\g}'-\psi_{\g}$ is constant. It holds that 
\begin{align*}
\lefteqn{\exp(i (\psi'_{\mu}(x_0)+\psi'_{\g}(\mu
   x_0)-\psi'_{\g\mu}(x_0)))}\\
   &=\exp(i(a_{\mu}+a_{\g}-a_{\g\mu}))\exp(i (\psi_{\mu}(x_0)+\psi_{\g}(\mu
   x_0)-\psi_{\g\mu}(x_0)))\ .
\end{align*}
The equation shows that the two induced multipliers are cohomologous.

    \item Now we only assume that $\omega$ is fixed. Let $\e$ and $\e'$ be such that
    $d\e=d\e'=\tilde{\o}$. Then there is $h\colon \Mt\rightarrow \C$ such
    that  $dh= \e'-\e$. We may set $\psi_{\g}'=\psi_{\g}+\g^{*}h-h$ since by (1) the cohomology class does not depend on the choice of $\psi_{\g}'$. It follows that
    \begin{eqnarray*}
\lefteqn{\s'(\g,\mu)= \exp\bigl( i (\psi'_{\mu}(x_0)+\psi'_{\g}(\mu
   x_0)-\psi'_{\g\mu}(x_0))\bigr)}\\
   &=& \exp\bigl( i(\psi_{\mu}(x_0)+\mu^{*}h(x_0)-h(x_0)+\psi_{\g}(\mu
   x_0)+\g^{*}h(\mu x_0) \\
   && \qquad -h(\mu x_0)-\psi_{\g\mu}(x_0)-(\g\mu)^*h(x_0)+h(x_0))\bigr)   =\s(\g,\mu) \ .
\end{eqnarray*}

    \item Let $\o'$ be such that
    $[\o]=[\o']\in H^2(M,\R)$.
    Then $\o'-\o=d\phi$ for some $\phi\in
    \O^{1}(M,\R)$.
    Let $\e $ be such that $d\e=\tilde{\o}$. By (2) the cohomology class of $\s'$ does not depend on the choice of $\e'$
    with $d\e'=\tilde{\o}'$. So we can choose
    $\e'=\e+p^{*}\phi$. Then $\g^{*}\e'-\e'=\g^{*}\e-\e$, because the contribution from the $\G$-invariant term $p^{*}d\phi$ disappears. By (1) we are also free in our choice of $\psi_\g'$. Since
    $d\psi_{\g}'=d\psi_{\g}$, we may choose $\psi_{\g}'=\psi_{\g}$ and get $\s'=\s$.
\end{enumerate}
\end{proof}

Thus, starting from a fixed $c\in H^{2}(B\G,\Q)$ the above construcion
gives a well defined $[\s]\in H^{2}(\G,U(1))$, geometrically described via a universal $\G$-covering.

\subsection{Projective %$(\G,\bar\s)$-
action on a $\G$-invariant bundle on a universal covering}

Let $q\colon \widetilde{E}\rightarrow \widetilde{M}$ be a $\G$-invariant vector
bundle on the universal $\G$-covering $\widetilde{M}$ (\emph{i.\,e.} $q$ is the pull-back
of a bundle $E\rightarrow M$). As usual, the group $\G$ acts from the right on $\widetilde E$, but in general, if not clear from the notation, we will consider the associated left action $\g u :=u\g^{-1}$.

Let $\s(\g,\mu)=\exp(i (\psi_{\mu}(x_0)+\psi_{\g}(\mu
   x_0)-\psi_{\g\mu}(x_0)))$ be constructed as above and consider
$A\colon \G\rightarrow \Aut(\widetilde{E})$ given by
$$
A_{\g}(u):=\g e^{-i \psi_{\g}(q(u))}u$$
Then
\begin{multline*}
(A_{\mu}\circ A_{\g})u= \mu\left(e^{-i
\psi_{\mu}(q(\g u))} \g e^{-i \psi_{\g}(q(u))} u
\right)=\\
=\mu\g e^{-i\left[\psi_{\mu}(\g\cdot
q(u))+\psi_{\g}(q(u)) \right]}u=\\
= \bar\s(\mu,\g) \mu\g e^{-i \psi_{\mu\g}(q(u))}u= \: \bar\s(\mu,\g)A_{\mu\g}(u).
\end{multline*}
So we have seen that $A$ is a $(\G,\bar\s)$-action
on the total space $\widetilde{E}$. (It will become clear later why we consider $(\G,\bar\s)$-actions and not $(\G,\s)$-actions here.)
\medskip

On sections $\Cinf(\widetilde{E})$ of $q$ the induced projective action is given by
$$(A_{\g}\t) (x)=
\g e^{-i\psi_{\g}(x\g)}\t(x\g) \ ,
$$
or, equivalently, $A_{\g}\t=(\g^{-1})^*(e^{-i\psi_{\g}}\t)$.

\begin{lem}\label{scelte-repr} 
Different choices of $\psi_\g, \e$ and $\o$ induce projectively equivalent actions.
More precisely: If $A_\g, A'_\g$ are two actions constructed as above but with different choices, then there are a map $z\colon \Gamma \to U(1)$ and a smooth function $f\colon \widetilde M \to U(1)$ such that for all $u \in \widetilde E$
$$A'_\g(u)=\bar z(\g) f(\g \cdot q(u))^{-1}A_{\g} \bigl( f(q(u))u \bigr) \ .$$

If $A_{\g}$ is a $(\Gamma,\bar\s)$-action, then $A'_\g$ is a $(\Gamma, \bar \s')$-action with $\s'= \s \partial z$.
\end{lem}

\begin{proof} The method of proof is analogous to the one of Lemma \ref{lemma-scelte}. Again, we proceed in three steps. 

\begin{enumerate}
\item  With $\eta'=\eta$ and $\psi'_{\g}-\psi_{\g}
=a_{\g}\in \R$ it holds that
$$A_{\g}'(u)=\g e^{-i \psi_{\g}'(q(u))}u=e^{-ia_\g}\g e^{-i \psi_{\g}(q(u))}u=e^{-ia_\g}A_{\g}(u) \ .$$

\item If $d\e=d\e'$, then $\e'-\e=dh$. We choose 
$\psi_{\g}'=\psi_{\g}+\g^{*}h-h$ and set $f:=e^{i h}$. Then
\begin{align*}
A_{\g}'(u)&=\g e^{-i \psi_{\g}'(q(u))}u= \g
e^{-i \psi_{\g}(q(u)) }e^{-i\g^{*}h(q(u)) }e^{i h(q(u))}u \\
&=e^{-i h(\g \cdot q(u))}A_{\g}(e^{i h(q(u))}u)=f(\g \cdot q(u))^{-1}A_{\g}f(q(u))u \ .
\end{align*}
For general $\psi'_{\g}$ the assertion follows from (1).

\item Part (3) of the proof of Lemma \ref{lemma-scelte} shows that for two different choices of $\omega$ one may choose the same functions $\psi_{\g}$. Since $A_{\g}$ only depends through these functions on $\omega$, the claim follows here with $f=1$, $z=1$.
\end{enumerate}
The last assertion follows from the general properties of projective representations given in \S \ref{multipliers}.
\end{proof}

In particular we get the following: If $A_{\g}'$ and $A_{\g}$ are projective actions associated to  the same multiplier $\sigma$, but coming from different choices, then the function $z\colon \G \to U(1)$ is a group homomorphism.

\subsection{$\cgs$-bundle on $M$ with $\pi_1(M)=\G$} \textrm{\cite[\S 2]{Ma2}} 
\label{bcgs-bundle}

Consider on $\Mt$ the trivial line bundle $L=\Mt\times \C\rightarrow \Mt$ and the trivial $C^*(\G,\s)$-bundle   $L\otimes
\cgs\rightarrow \Mt$. Define on the total space $L\otimes
\cgs$ the following action, denoted  by $A\otimes T$,
\begin{equation}
\label{AtensT}
(A\otimes T)_{\g}(u\otimes v)=A_{\g}(u)\otimes T_{\g}v
\end{equation}
where $T_{\g}v=\de_{\g}*\,v$. $A\otimes T$
turns out to be a classical (left) action of the group $\G$, because it
is defined as the tensor product of the projective
$(\G,\bar\s)$-action $A_{\g}$ and the projective
$(\G,\s)$-action $T$ (recall that $\s\bar{\s}=1$). Moreover $q((A\otimes
T)_{\g}w)=\g\cdot q(w)$, thus there exist a quotient bundle
on $M$, with fibre $\cgs$, which is denoted by
$$\N^{\sigma}:=\frac{L\otimes \cgs}{A\otimes T} \ .$$

The bundle $\N^{\sigma}\to M$ plays the role of the Mishchenko--Fomenko bundle
$\N=\Mt\times_{\G}C^*\G$ in the untwisted case.

The standard $\cgs$-valued scalar product on $L \ten \cgs$  induces a $\cgs$-valued scalar product on $\N^{\sigma}$.

\begin{lem}\label{choices}
Assume that $\N^\s$ and $\N^{\s'}$ arise from different choices of $\psi_\g, \eta$ and $\omega$ in the constructions above. Then there is a projective isomorphism $b_z \colon C^{*}(\G,\s') \to \cgs$ as in Def. \ref{projiso} such that the bundles $\N^{\s}$ and $\N^{\s'}\ten_{b_z} \cgs$ are unitarily isomorphic as $\cgs$-vector bundles.
\end{lem}

\begin{proof}
Let $A_\g, A'_{\g}$ be two projective actions associated to different choices in the constructions above. By Lemma \ref{scelte-repr} there are $z\colon \Gamma \to U(1)$ with $z(e)=1$ and a smooth function $f\colon \widetilde M \to U(1)$ such that $A_\g'=\bar z(\g) f^{-1} A_\g f$ and $\s'=\s\partial z$. We may identify $\N^{\s'}\ten_{b_z} \cgs$ with the quotient of $L\otimes  C^*(\G,\s')\ten_{b_z} \cgs$ under the action $A \ten T' \ten 1$, where $1$ denotes the trivial action of $\G$ on $\cgs$. 

The isomorphism between the two bundles is constructed this way: consider the map
\begin{eqnarray*}
\Sigma\colon L\otimes  C^*(\G,\s')\ten_{b_z} \cgs \rightarrow L\otimes \cgs \\
\Sigma (u\otimes \delta_\g \ten \delta_\mu)=f u\otimes z(\g)\delta_\g*\delta_\mu \ .
\end{eqnarray*}
It is well defined since the last term equals
$\Sigma (u \ten \delta_e \ten z(\g)\delta_\g*\delta_\mu)$.

It satisfies $\Sigma (A'\otimes T' \ten 1)_{\g}=(A\otimes T)_\g \Sigma$, in fact
\begin{multline*}\begin{aligned}
\Sigma (A'\otimes T' \ten 1)_{\g_1})(u\otimes \delta_{\g_2} \ten \delta_{\g_3})=&
\Sigma (A'_\g u\otimes \s'(\g_1,\g_2)\delta_{\g_1 \g_2} \ten \delta_{\g_3})\\
=& fA'_\g u\otimes \s'(\g_1,\g_2)z(\g_1 \g_2)\delta_{\g_1 \g_2}*\delta_{\g_3}
\end{aligned}
\end{multline*}
and 
\begin{multline*}
(A_{\g_1}\otimes  T_{\g_1})(\Sigma(u\otimes \delta_{\g_2} \ten \delta_{\g_3}))=(A_{\g_1}\otimes  T_{\g_1})(f u\otimes z(\g_2)\delta_{\g_2}*\delta_{\g_3})
\\
=A_{\g_1} f u\otimes z(\g_2)\s(\g_1,\g_2)\delta_{\g_1 \g_2}*\delta_{\g_3}=z(\g_1)f A'_{\g_1} u\otimes z(\g_2)\s(\g_1,\g_2)\delta_{\g_1 \g_2}*\delta_{\g_3}
\\
=f A'_{\g_1} u\otimes z(\g_1 \g_2)\partial z (\g_1,\g_2)\s(\g_1,\g_2)\delta_{\g_1 \g_2}*\delta_{\g_3}
=fA'_{\g_1} u\otimes z(\g_1 \g_2)\s'(\g_1,\g_2)\delta_{\g_1 \g_2}*\delta_{\g_3}.
\end{multline*}
Thus, the map $\Sigma$ descends to the quotients. One easily checks that it induces an isomorphism.
\end{proof}

In particular, for fixed $\sigma$, the bundle $\N^{\sigma}$ is canonically defined up to tensoring by $\ten_{b_{\chi}} \cgs$ for group homomorphisms $\chi \colon \Gamma \to U(1)$.

\subsection{Connection on $\N^\s$} \textrm{\cite[\S 3]{Ma2}}
Choose $\psi_\g$, $\eta$, $\o$ as above. On the induced bundle $\N^\s$ a connection is defined as follows. Consider first on the
trivial bundle $L\otimes \cgs$ the connection $\nabla$ given by
\begin{equation}\label{connessione}
    \nabla \t =d\t+i\e\t \ .
\end{equation}
Note that if $\partial M\neq\emptyset$, then $\nabla$ is of product type near the boundary by Lemma \ref{eta}.
 
\begin{lem}
The connection $\nabla$ commutes with the action $A\otimes T$; hence it defines a connection $\nabla^{\s}$ on the quotient bundle $\N^\s$. 
\end{lem}

\begin{proof}

The connection $\nabla$ commutes with the action $T$, which is ``vertical"
along the fibre. Moreover, for $\tau \in \Cinf(L)$,
\begin{multline*}
\nabla A_{\g}\tau =\nabla\left(e^{-i
(\g^{-1})^{*}\psi_{\g}}(\g^{-1})^{*}\tau\right)=
\\
=e^{-i (\g^{-1})^{*}\psi_{\g}} \nabla \left((\g^{-1})^{*}\tau\right)
+ e^{-i (\g^{-1})^{*}\psi_{\g}} \left( -i(\g^{-1})^{*}(d\psi_{\g})
\right)((\g^{-1})^{*}\tau)=
\\
=e^{-i (\g^{-1})^{*}\psi_{\g}} \left[(\g^{-1})^{*} d\tau \;+ \;i \e
(\g^{-1})^{*} \tau\;-\; i (\g^{-1})^{*}(\g^{*}\e -\e
)((\g^{-1})^{*}\tau) \right]=
\\
=e^{-i (\g^{-1})^{*}\psi_{\g}}\left[(\g^{-1})^{*} (d\tau)\;+\;
\left(i\:(\g^{-1})^{*}\e \right)(\g^{-1})^{*} \tau\right]\,.
\end{multline*}
On the other side
\begin{multline*}
A_{\g}(\nabla \tau)= A_{\g}(d\tau+i\e \tau)\\=e^{-i(\g^{-1})^{*}\psi_{\g}}  (\g^{-1})^{*}(d\tau)+ ie^{-i
(\g^{-1})^{*}\psi_{\g}}(\g^{-1})^{*}\e\left((\g^{-1})^{*} \tau\right)
\end{multline*}
 hence $A_{\g}(\nabla \tau)=\nabla A_{\g}\tau$. Therefore $\nabla$ commutes with $A\otimes
T$.
\end{proof}

The curvature of $\nabla^{\s}$ is
$i\o \in i\O^{2}(M,\R)$. In fact, for $\tau\in \Cinf(L)$ and $v\in \cgs $  
\begin{multline*}\nabla^2 (\tau\otimes v)= \nabla (d\tau+ i\tau\eta )\otimes v\\
=\bigl(i d(\tau \eta)
+i\eta \wedge d\tau-\tau(\eta\wedge \eta)\bigr) \ten v\\
=\bigr(i d\tau\wedge \eta
+i\tau d\eta +i\eta \wedge d\tau\bigr) \otimes v=i\tau \tilde{\o}\otimes v\ .\end{multline*}

\begin{lem}
\label{choices-conn}
Let $\N^{\s},\N^{\s'}$ be two bundles arising from choices $\psi_\g$, $\eta$, $\o$  and $\psi_\g'$, $\eta'$, $\o'$, respectively, in the constructions above.

For any  $\phi\in \Omega^1(M,\R)$ with $d\phi=\omega'-\omega$, there is $z\colon \G \to U(1)$   and an unitary isomorphism $U \colon \N^{\s'}\ten_{b_z} \cgs \to \N^{\s}$ such that $$U\nabla^{\s'}U^{-1}=\nabla^{\s}+i \phi\ .$$  
\end{lem}

\begin{proof}
We write $\nabla', \nabla$ for the connections on $L\otimes  C^*(\G,\s')$ and $L\otimes  \cgs$ inducing $\nabla^{\s'}, \nabla^{\s}$, respectively.

By Lemma \ref{scelte-repr} there are $z\colon \G \to U(1)$ and $f\colon\widetilde M \to U(1)$  such that $A_\g'=\bar z(\g)f^{-1} A_\g f$. Recall from the proof of Lemma \ref{choices} that then  
\begin{eqnarray*}
\Sigma\colon L\otimes  C^*(\G,\s')\ten_{b_z} \cgs \rightarrow L\otimes \cgs \\
\Sigma (u\otimes \delta_\g \ten \delta_\mu)=f u\otimes z(\g)\delta_\g*\delta_\mu
\end{eqnarray*}
is an unitary isomorphism.

For $\tau \ten \delta_\g \in \Cinf(L) \ten \cgs$ we have that
\begin{eqnarray*}
\Sigma \nabla' \Sigma^{-1}(\tau \ten \delta_\g)&=&\Sigma \nabla' z(\g)^{-1}(f^{-1}\tau \ten \delta_\g \ten \delta_e)\\
&=&z(\g)^{-1}\Sigma \bigl((\tau df^{-1}+ f^{-1}d\tau + i f^{-1}\tau \eta') \ten \delta_\g \ten \delta_e\bigr)\\
&=&(\tau fdf^{-1}+ d\tau + i\tau \eta') \ten \delta_\g \ .
\end{eqnarray*}

Now we consider the three cases from the proof of Lemma \ref{scelte-repr}. 
\begin{enumerate}
\item If $\eta=\eta'$, $\omega=\omega'$, then by case (1) of Lemma \ref{scelte-repr} we may choose $f=1$. Thus in this case the assertion holds.

\item If $\eta'=\eta+dh$ and the functions $\psi_\g'$ are related to $\psi_\g$ as in case (2) of Lemma \ref{scelte-repr}, then we may set $f=e^{ih}$ and $z=1$. Thus the last line equals
$$(-i\tau dh+ d\tau + i\tau \eta+i\tau dh) \ten \delta_\g =\nabla \tau \ten \delta_\g \ .$$ 

\item In these cases we had $\o=\o'$. Now if $\o'=\o+d\phi$ for $\phi\in \Omega^1(M,\R)$, we set $\eta'=\eta+p^*\phi$ and $\psi'_{\g}=\psi_{\g}$. Then $A'_{\g}=A_{\g}$; thus we may set $f=1$, $z=1$. Now the last line equals
$$(d\tau + i\tau \eta + i\tau p^*\phi) \ten \delta_\g=(\nabla \tau  +i\tau p^*\phi)\ten \delta_{\g} \ .$$
Thus, in this case $\Sigma\nabla'\Sigma^{-1}=\nabla + ip^*\phi$. 
\end{enumerate}
\end{proof}

\begin{rem}
As usual, we can twist a Dirac operator $D$ associated to a Dirac bundle $E \to M$ with the bundle $\N^{\s}$. The resulting Dirac operator $D_{\mathcal V^{\s}}$ acts on the sections of $E\otimes \mathcal V^{\s}$.

For an even-dimensional closed manifold $M$, the index $\ind (D^+_{\mathcal V^{\s}}) \in K_0(C^*(\G,\s))$ depends on the choices of $\psi, \eta, \o$ only through the action of projective automorphisms on $K_0(C^*(\G,\s))$.

\end{rem}

\subsection{Parametrizations}
\label{para-constr}

We also need parametrized generalizations of the constructions above.

Let $(c_1, c_2 \dots c_k) \in H^2(B\G,\Q)^k$ and for $s=(s_1,s_2 \dots s_k) \in \Q^k$ define $c_s:=s_1 c_1 + s_2c_2 + \dots s_k c_k \in H^2(B\G,\Q)$. 

For each $i=1, \dots k$ we can construct $\o^i \in \O^2(M,\R)$ and $\eta_i \in \Omega^1(\widetilde M,\R)$ from $c_i$ as in Lemma \ref{eta} and put $\eta_s:=s_1\eta_1 + s_2 \eta_2 + \dots s_k\eta_k$ and  $\omega_s=d\eta_s=s_1\o_1+s_2\o_2+ \dots s_k\o_k$.

Similarly we choose $\psi_\g^i,~i=1,\dots, k$ and set $\psi^s_{\g}=s_1\psi_{\g}^1 + s_2\psi_{\g}^2 + \dots +s_k\psi_\g^k$. 

Using $\eta_s, \omega_s, \psi_\g^s$, we get parametrized versions of the above constructions, in particular a family of multipliers $\sigma^s$.

We adapt some lemmata from above to this situation. The proofs are easy generalizations of the previous ones and are left to the reader.

The following generalizes Lemma \ref{scelte-repr}.

\begin{lem}\label{scelte-repr-s}
If $A(s)_\g, A'(s)_\g$ are two parametrized actions constructed in such a way, then there are maps $a_1,a_2, \dots a_k\colon \Gamma \to \R$ and smooth functions $h_1,h_2, \dots h_k \colon \widetilde M \to \R$ such that with 
$$z_s=\exp(-i(s_1a_1+s_2a_2+\dots s_ka_k)),~ f_s=\exp(i(s_1h_1+s_2h_2+\dots s_kh_k))$$
we have 
$$A'(s)_\g=\bar z_s(\g) f_s^{-1} A(s)_{\g} f_s \ .$$
Furthermore, if $A_{\g}$ is a $(\Gamma,\bar\s_s)$-action and $A'_\g$ is a $(\Gamma, \bar \s_s')$-action, then $\s_s'=\s_s \partial z_s$.
\end{lem}

Note that $\s_s'=\s_s$ for all $s$ in a neighbourhood of $\,0 \in \Q^k$ if and only if the maps $a_1, a_2, \dots a_k\colon \G \to \R$ are group homomorphisms. This motivates the following parametrized version of Def. \ref{projiso}.

\begin{definition}
\label{para_projiso}
Let $\s_s, \s_s'$ be parametrized multipliers on $\G$ constructed as above.

Let $a_1,a_2, \dots a_k\colon \Gamma \to \R$ be such that 
$$z_s\colon \G \to U(1),~ z_s=\exp(i(s_1a_1+s_2a_2+\dots s_ka_k))$$ fulfills $\s_s'=\s_s \partial z_s$ near $0$.

Then we call the induced family $$b_{z_s}\colon C^*(\G,\s'_s)\rightarrow C^*(\G,\s_s), ~\delta_\g\mapsto z_s(\g)\delta_\g$$ a \emph{parametrized projective isomorphism}.

If $\s'_s=\s_s$ near $0$, then we call the induced family $b_{z_s}$ a \emph{parametrized projective automorphism}.
\end{definition}

The following is the parametrized version of Lemma \ref{choices-conn}.

\begin{lem}
\label{choices-conn-s}
Let $\N^{\s_s},\N^{\s_s'}$ be two bundles arising from different choices in the constructions above.

For any  $\phi_1, \phi_2, \dots \phi_k \in \Omega^1(M,\R)$ with $d\phi_i=\omega_i'-\omega_i$ there are $z_s\colon \G \to U(1)$  and $f_s\colon \widetilde M \to U(1)$ as in the previous lemma such that the  unitary isomorphism $$U_s \colon \N^{\s_s'}\ten_{b_{z_s}} \cgs \to \N^{\s_s},~ [u\otimes \delta_\g \ten \delta_\mu] \mapsto [f_s u\otimes z_s(\g)\delta_\g*\delta_\mu]$$ fulfills  $$U_s\nabla^{\s_s'}U_s^{-1}=\nabla^{\s_s}+ i \phi_s\ .$$
Here $\phi_s=s_1\phi_1+s_2\phi_2 + \dots s_k\phi_k$.  
\end{lem}

One reason for considering the parametrized version is that it allows us to make the curvature as small as necessary in a controlled way.

\begin{lem}\label{s-picc} 
For any $\ep >0$ there exists a neighborhood $U \subset \Q^k$ of $0$ such that for all $s \in U$ the connection $\nabla^{\s_s}$ has curvature smaller than $\ep$.
\end{lem}

\subsection{The index of twisted Dirac operators}
\label{twistdir}

The parametrized version allows one to construct a well-defined $K$-theoretic index. 

Let $M$ be an even-dimensional closed manifold and $D$ a Dirac operator on $M$. Consider the twisted Dirac operator $D_{\mathcal V^{\s_s}}$. We can interpret $K_0(C^*(\G,\s^s))$ as a bundle of groups over a neighbourhood $U$ of $0 \in \Q^k$. Note that parametrized projective automorphisms act trivially on the fibers of this bundle since they are connected to the identity through a path of projective automorphisms. Namely, for $t\in [0,1]$, we can consider the path $z_s^t=\exp(it(s_1a_1+s_2a_2+\dots s_ka_k))$ in Definition \ref{para_projiso}.

We define $\ind(D^+_{\mathcal V^{\s_s}})$ as a section of this bundle and call this the index of $D_{\mathcal V^{\s_s}}$. It is independent of the choices. We will be interested only in the germ at $0$, for which we use the same notation.

Clearly, by evaluating the index at some $s_0$ one obtains an index in the group $K_0(C^*(\G,\s^{s_0}))$. However, its definition does not depend on $\s^{s_0}$ alone but depends also on the family $\s^s$ (which, since it is analytical in $s$, is determined by the germ at $s_0$).

%\section{The construction revisited}\label{rev-constr}

%When $\widetilde{M}$ is simply connected, then there exists $\psi_\g$ such that $d\psi_\g=\g^*\eta-\eta$, and say we fix it by asking $\psi_\g(x_0)=0$, for a fixed $x_0\in \widetilde{M}$.
%These data give the Mathai construction of the preceding section.
%Suppose more generally $W$ is another manifold admitting a regular $\G$-covering so that we have a projection $q\colon\pi_1(W)\rightarrow \G$. Then there is a natural $(\pi_1(W),q^*\bar{\s})$-projective action on $C^*(\G,\bar{\s})$, and we define a $C^*(\G,\bar{\s})$ bundle on $W$ as 
%\begin{equation}\label{VW}
%\mathcal V_{W}=\frac{(\widetilde{W}\times \C)\otimes C^*(\G,\bar{\s}) }{A_g\otimes T_{q(g)}}\stackrel{\pi}{\longrightarrow} W
%\end{equation}
%where $A_g$ is a $(\pi_1(W), q^*\s)$-action on $\widetilde W\times \C$, and $T_q(g)$ is a $(\pi_1(W), q^*\bar{\s})$-action on $C^*(\G,\bar{\s})$.

\section{Von Neumann versus $C^*$-algebraic Atiyah--Patodi--Singer index theorem}
\label{neumann_Cstar}

Let $\cA$ be a unital $C^*$-algebra endowed with a positive finite (not necessarily faithful) trace $\tau$ and let $\cA_{\tau}$ be the associated Hilbert space. We denote by $\tralg\colon M_n(\cA) \to \cA/\overline{[\cA,\cA]}$ the standard trace. Thus we get a trace $\tau \tralg\colon M_n(\cA) \to \C$. The induced trace for integral operators on a closed manifold with smooth integral kernel is denoted by $\Tr_{\tau}$.

In this section we derive a general von Neumann algebraic Atiyah--Patodi--Singer index theorem from the $C^*$-algebraic one, for which we refer to \cite{ps1}. A special case is Ramachandran's $L^2$-Atiyah--Patodi--Singer index theorem for coverings \cite{r}. The relation between von Neumann algebraic and $C^*$-algebraic approaches to $L^2$-index theory for coverings is well-understood for manifolds without boundary, see for example \cite[Appendix E]{ps1} \cite[\S 7]{wj}. Many arguments carry over from there so that we do not give full details.

Let $M$ be an even-dimensional oriented Riemannian manifold with boundary $\ra M$. For a collar neighbourhood $U$ of the boundary we fix an isometry $U \cong (-\ep,0] \times \ra M$ and denote the projection onto the first variable by $x$ and the projection onto $\ra M$ by $p_{\ra}$. By assuming that the isometry is orientation preserving we get an induced orientation on $\ra M$.

Let $E$ be a $\Z/2$-graded Dirac bundle and let $\cF$ be an $\cA$-vector bundle on $M$ endowed with a $\cA$-valued scalar product and a unitary connection $\nabla^{\cF}$. We assume all structures to be of product type on $U$. We denote the grading
operator on $E$ by $\grad$.

On $U$ we can identify $E$ with $p_{\ra}^*E_{\ra M} \ten (\C^+ \oplus \C^-)$, where $E_{\ra M}:=E^+|_{\ra M}$ and $\C^{\pm}$ denotes the vector space $\C$ with a grading consisting only of positive and negative elements, respectively. There is an induced Dirac bundle structure on $E_{\ra M}$. 

The induced Dirac operators $D$ acting on $\Cinf(M,E \ten\cF)$ and $D_{\ra}$ acting on $\Cinf(\ra M,E_{\ra M} \ten \cF)$ are related on $U$ via the equation
$$D=c(dx)(\ra_x -\grad D_{\ra}) \ .$$
(Here we identified $E^+$ and $E^-$ on $U$ using $ic(dx)$.)

We assume that there is a smoothing operator $A$ on $L^2(\ra M,E_{\ra M} \ten \cF)$ such that $D_{\ra}+A$ has a bounded inverse. (This is a restriction, however by the results in \cite{lpdir} it holds if the fibers of $\cF$ are isomorphic to $\cA$ and furthermore it can always be enforced by stabilizing with a trivial bundle.)

We write $P(A):=1_{\{x\ge 0\}}(D_{\ra}+A)$
and define $D(A)^+$ as the closed operator for which
$$\{f\in \Cinf(M,E^+\ten \cF)~|~P(A) (f|_{\ra M})=0\}$$ is a core. The operator $D(A)^-$ is defined as the adjoint. As usual, we get an odd operator $D(A)=\left(\begin{array}{cc} 0 & D(A)^- \\ D(A)^+ & 0 \end{array}\right) \ .$ 

There is a well-defined index $\ind(D(A)^+)\in K_0(\cA)$. If $A=0$, then we write $\ind(D^+)$. For simplicity we will usually write $\ind(D(A))$ for $\ind(D(A)^+)$.

By the $C^*$-algebraic Atiyah--Patodi--Singer index theorem 
$$\tau \tralg (\ind D(A))=\tau (\int \hat A(M)\ch(E/S)\ch(\cF)) -\eta_{\tau}(D_{\ra}+A)\ ,$$
where 
\begin{equation}
\label{eta-tauC*}
\eta_{\tau}(D_{\ra}+A):=\frac{1}{2\sqrt{\pi}}\int_0^\infty \Tr_{\tau} \left((D_{\ra}+A)e^{-t (D_{\ra}+A)^2}\right)\frac{dt}{\sqrt t}\;\;\in \C \ ,
\end{equation}
and $$\ch(\cF)=\tralg e^{-(\nabla^{\cF})^2} \ .$$

Note that in some references (for example in \cite{ps1}) the $\eta$-invariant  does not have the factor $2$ in the denominator.

When we consider metrics of positive scalar curvature we will use the fact that this definition does not require the trace $\tau$ to be positive.
\medskip

Now we consider the von Neumann algebraic situation. Here it is essential that $\tau$ is positive. 

We interpret  (while keeping the notation) $D(A)$ and $D_{\ra}$ as a closed selfadjoint operators on the Hilbert spaces $H:=L^2(M,E \ten \cF)\ten_{\cA}\cA_{\tau}$ and $H_{\ra}:=L^2(\ra M,E_{\ra M} \ten \cF)\ten_{\cA}\cA_{\tau}$, respectively. We define the von Neumann algebras $\cN$ and $\cN_{\ra}$ as the weak closures of the images of the $C^*$-algebras of bounded operators $\cB(L^2(M,E \ten \cF))$ and $\cB(L^2(\ra M,E_{\ra M} \ten \cF))$ in $\cB(H)$ and $\cB(H_{\ra})$, respectively. Note that the $C^*$-algebras act on the first factor of $H$ and $H_{\ra}$, respectively. 

These von Neumann algebras are endowed with a semifinite trace which we denoted by $\Tr_{\tau}$ since it extends the trace for integral operators introduced above. This is best seen by embedding $\cF$ isometrically into a trivial bundle $M \times \cA^n$. The von Neumann algebra $\cN$ embeds into the von Neumann algebraic tensor product $\cB(L^2(M,E))\ten M_n(\cB(\cA_{\tau}))$, which is endowed with a semifinite trace induced from the semifinite traces on each factor. (If the trace $\tau$ is faithful, then $\cB(\cA_{\tau})=\cA''$.)

We denote by $\cK(L^2(M,E \ten\cF))$ the ideal of compact operators acting on the Hilbert $\cA$-module $L^2(M,E \ten\cF)$ and by $\cK(\cN)$ the closed ideal generated by positive operators of finite trace in $\cN$.

There are induced homomorphisms 
\begin{align*}
\cB(L^2(M,E \ten\cF)) &\to \cN\\
\cK(L^2(M,E \ten\cF)) &\to \cK(\cN) \ .
\end{align*} 
Furthermore there is an induced map mapping regular operators on $L^2(M,E \ten\cF)$ to operators affiliated to $\cN$. Note that the image of a Fredholm operator is a Breuer--Fredholm operator.

\medskip
Thus, the operators $D(A)$ and $D_{\ra}$ are affiliated to $\cN$ and $\cN_{\ra}$, respectively.
Furthermore the operator $D(A)$ is Breuer--Fredholm, since it is Fredholm as an operator on the Hilbert $\cA$-module $L^2(M,E\ten \cF)$.
 
\medskip
 
Now we consider the unperturbed operator $D$ for which we want to prove a von Neumann algebraic Atiyah--Patodi--Singer index theorem. 
 
\medskip
Note that the space $\Cinf(M,\cA_{\tau}^n)$ is well-defined. By embedding $E \ten\cF$ isometrically into a trivial $\cA$-vector bundle $M \times \cA^n$ one may define the space $\Cinf(M,E\ten \cF\ten_{\cA}\cA_{\tau})$ and similar spaces used in the following.
 
\medskip  
Let $P_{\ge}=1_{\{x\ge 0\}}(D_{\ra})$. First we check that $P_{\ge}$ acts continuously on $\Cinf(\ra M, E_{\ra M} \ten \cF\ten_{\cA}\cA_{\tau})$. Note that since $D_{\ra}$ is a selfadjoint operator on the Sobolev spaces $H^k(\ra M,E_{\ra M} \ten \cF)\ten_{\cA}\cA^n_{\tau}$, the operator $P_{\ge}$ acts continuously on them. Now we consider the projective limit $k \to \infty$. By \cite[Theorem 48.3]{tr} there is a canonical bounded map from the Hilbert space tensor product $H^k(\ra M) \ten \cA_{\tau}^n$ to the $\ve$-tensor product $H^k(\ra M) \hat\ten_{\ve} \cA_{\tau}^n$. On the other hand, by the universal property of the projective tensor product, there is a canonical continuous map $H^{\infty}(\ra M) \hat \ten_{\pi} \cA_{\tau}^n \to H^k(\ra M) \ten \cA_{\tau}^n$. Since $H^{\infty}(\ra M)\cong \Cinf(\ra M)$ is nuclear, the composition of these maps induces an isomorphism in the projective limit. Here we used that the $\ve$-tensor commutes with taking projective limits \cite[p. 282]{ko}. It follows that 
$$H^{\infty}(\ra M) \ten \cA_{\tau}^n \cong H^{\infty}(\ra M) \hat\ten_{\ve} \cA_{\tau}^n \cong \Cinf(\ra M)\hat\ten_{\ve} \cA_{\tau}^n\cong \Cinf(\ra M,\cA_{\tau}^n) \ .$$ 
Thus also $H^{\infty}(\ra M,E_{\ra M} \ten \cF)\ten_{\cA}\cA_{\tau}\cong \Cinf(\ra M,E_{\ra M}\ten \cF\ten_{\cA}\cA_{\tau})$.

\medskip
In particular, this implies that every $f \in \Cinf(\ra M,E_{\ra M}\ten \cF\ten_{\cA}\cA_{\tau})$ can be written as $P_{\ge} f + (1-P_{\ge})f$ with $P_{\ge}f, (1-P_{\ge})f \in  \Cinf(\ra M,E_{\ra M}\ten \cF\ten_{\cA}\cA_{\tau})$.

\medskip
The Atiyah--Patodi--Singer boundary conditions for the operator $D$ are defined by requiring that 
$$\{f\in \Cinf(M,E^+\ten \cF\ten_{\cA}\cA_{\tau}) ~|~P_{\ge}(f|_{\ra M})=0\}$$ 
is a core of $D^+$
and 
$$\{f\in \Cinf(M,E^-\ten \cF\ten_{\cA}\cA_{\tau})~|~(1-P_{\ge})(f|_{\ra M})=0\}$$ 
is a core of $D^-$.

\begin{prop}
\label{selfad}
The operator $D$ with these boundary conditions is selfadjoint and affiliated to $\cN$.
\end{prop}

\begin{proof}
The proof is similar to the proof of \cite[Prop. 2.10]{aw}, where more details can be found.

In the following $\cN$ denotes various von Neumann algebras (for example defined using a cylinder instead of $M$) whose definitions are analogous to the one above and should be clear from the context.

Let $\chi\colon (-\ep,0] \to [0,1]$ be a smooth function with $\chi(x)=1$ for $x \in (-\frac 38\ep,0]$ and $\chi(x)=0$ for $x \in (-\ep,-\frac 58 \ep)$. In an obvious way $\chi$ defines a function on $M$ supported on $U$.

We consider the operator 
$$D':=D - \chi c(dx)\grad (1_{[0,\frac 12]}(D_{\ra})-1_{[-\frac 12,0)}(D_{\ra}))$$ 
and endow it with the same boundary conditions as $D$. Note that the operator 
$$D'_{\ra}:=D_{\ra} + 1_{[0,\frac 12]}(D_{\ra})-1_{[-\frac 12,0)}(D_{\ra})$$ is invertible and that the Atiyah--Patodi--Singer boundary conditions for $D'$ are the same as those for $D$ since $1_{\{x \ge 0\}}(D'_{\ra})=1_{\{x \ge 0\}}(D_{\ra})$.

We show that for $\lambda \in \R^+$ large the operator $D- i\lambda$ is invertible with inverse in $\cN$. Analogous arguments work for $D+ i\lambda$. This will imply the assertion.

Let $U_0=(-\frac 12\ep,0]\times \ra M,~ U_1=(-\frac 34 \ep,-\frac 14 \ep) \times \ra M,~ U_2=M\setminus \ov{U_0}$ and let $(\phi_j)_{j=0,1,2}$ be a partition of unity subordinate to the covering $(U_0,U_1,U_2)$ of $M$. We assume that $\phi_j$ only depends on the normal variable $x$. 

We consider $U_0$ as a subset of the half-cylinder $Z_-=(-\infty,0] \times \ra M$. The operator $D^{Z_-}:=c(dx)(\ra_x - \grad D'_{\ra})$ acting on sections over $Z_-$ with Atiyah--Patodi--Singer boundary conditions at $0$ is invertible with symmetric inverse in $\cN$ (now defined with respect to the half-cylinder instead of $M$). In particular it is selfadjoint and affiliated to $\cN$. Define $Q_0(\lambda):=(D^{Z_-}-i\lambda)^{-1} \in \cN$. 

Similarly, we consider $U_1$ as a subset of the cylinder $Z=\R \times \ra M$ and define the operator $D^{Z}:=c(dx)(\ra_x - \grad D'_{\ra})$ on $Z$. It has a symmetric inverse and thus is selfadjoint and affiliated to $\cN$. It follows that $Q_1(\lambda):=(D^Z-i\lambda)^{-1} \in \cN$.

In order to construct $Q_2(\lambda)$ we embed $M$ into its double and extend all structures to the double in the standard way. We denote the induced Dirac operator on the double by $D^{db}$ and set $Q_2(\lambda):=(D^{db}-i\lambda)^{-1} \in \cN$. 

Define $$Q(\lambda)=\sum_{j=0}^2\phi_jQ_j(\lambda)\phi_j$$
and set
$$K(\lambda) :=(D' -i\lambda)Q(\lambda)-1=\sum_{j=0}^{2}c(d\phi_j)Q(\lambda)_j\phi_j \in \cN \ .$$
Since $\|Q_j(\lambda)\| \le \lambda^{-1}$, for $\lambda$ large enough $\|K(\lambda)\| \le \frac 12$ and the operator $1+K(\lambda)$ is invertible. Then $Q(\lambda)(1+K(\lambda))^{-1} \in \cN$ is a right inverse of $D'-i\lambda$. One gets a left inverse in an analogous way. It follows that $(D'-i\lambda)^{-1}=Q(\lambda)(1+K(\lambda))^{-1}\in \cN$. Thus $D'$ is selfadjoint and affiliated to $\cN$. Since $D$ is a pertubation of $D'$ by a symmetric element in $\cN$, it is also selfadjoint and affiliated to $\cN$. 
\end{proof}

\begin{prop}
The operator $D$ with Atiyah--Patodi--Singer boundary conditions is Breuer--Fredholm.
\end{prop}

\begin{proof}
First we show that the inverse $(D'-i\lambda)^{-1}=Q(\lambda)(1+K(\lambda))^{-1}$ constructed in the previous proof is an element of $\cK(\cN)$. This is well known for the operator  $\phi_0Q_0(\lambda)\phi_0$. For the operators $\phi_1Q_1(\lambda)\phi_1$ and $\phi_2Q_2(\lambda)\phi_2$ it follows, for example, from the results in \cite[\S 2]{aw}. Thus $Q(\lambda) \in \cK(\cN)$ and therefore also $(D'-i\lambda)^{-1}$.

Since $(D-D'+i\lambda)$ is bounded,
$$D(D'-i\lambda)^{-1}-1=(D-D'+i\lambda)(D'-i\lambda)^{-1} \in \cK(\cN) \ .$$
It follows that $(D'-i\lambda)^{-1}$ is a left parametrix of $D$.

In order to show that it is a right parametrix one checks similarly that
$$(D+i)(D'-i\lambda)^{-1}D(D+i)^{-1}-1 \in \cK(\cN) \ .$$
\end{proof}

We denote by $\ind_{\tau}$ the index of Breuer--Fredholm operators affiliated to $\cN$. 
For the theory of Fredholm operators affiliated to semi finite von Neumann algebras, we refer the reader to \cite{CPRS2}.
Note that $\ind_{\tau}(D(A)^+)=\tau \tralg \ind(D(A)^+)$. As above, we often write $\ind_{\tau}(D(A))$ for this expression.

\medskip

Now we use the $C^*$-algebraic Atiyah--Patodi--Singer index theorem for $D(A)$ and a von Neumann algebraic relative index theorem to derive the von Neumann algebraic index theorem for $D$.

\medskip

In the following proposition we define $D^I:=\ra_x - D_{\ra}$ as an operator on $L^2([0,1]\times \ra M,\cF)\ten_{\cA} \cA_{\tau}$. The boundary conditions for $D^I$ are such that the space
$$\{f \in \Cinf([0,1]\times \ra M,\cF\ten_{\cA} \cA_{\tau})~|~ (1-P(A))f(0)=0,P_{\ge}f(1)=0\}$$
is a core.

Here $f(j)=f|_{\{j\} \times \ra M}$.

\begin{prop}
It holds that $$\ind_{\tau}(D^+)=\ind_{\tau}(D(A)^+) + \ind_{\tau}(D^I) \ .$$
\end{prop}

\begin{proof}
In the $C^*$-algebraic context similar results were proven in \cite[\S 4]{lpdir} based on the  strategy of the proof of the $K$-theoretic relative index theorem \cite{bu}. 
This strategy was adapted to the von Neumann algebraic context in the proof of \cite[Lemma 2.5]{aw}. We leave the details to the reader.
\end{proof}

We define the $\eta$-invariant
\begin{equation}\label{etavN}
\eta_{\tau}(D_{\ra})=\frac{1}{2\sqrt{\pi}}\int_0^\infty \Tr_{\tau} \left(D_{\ra}e^{-t D_{\ra}^2}\right)\frac{dt}{\sqrt t} \ .
\end{equation}
The integral converges at $t=0$ since $D_{\ra}$ is a Dirac operator. For $t \to \infty$ a proof of the convergence can be found in \cite[\S 8]{cpJLO}, which generalizes the classical \cite{CG}. Here we use the fact that $D_{\ra}$ is $\theta$-summable. Clearly, the definition is compatible with the one in \eqref{eta-tauC*}.

Now we can derive the Atiyah--Patodi--Singer index theorem. Let $P_{\Ker D_{\ra}} \in \cN_{\ra}$ be the projection onto the kernel of $D_{\ra}$. 

\begin{prop}
With $n=\dim M$ it holds that
$$\ind_{\tau}(D^+)=(2\pi i)^{n/2}\tau (\int \hat A(M)\ch(E/S)\ch(\cF)) -\eta_{\tau}(D_{\ra})-\frac 12 \tau(P_{\Ker D_{\ra}}) \ .$$
\end{prop}

\begin{proof}
Let $\chi\colon [0,1] \to [0,1]$ be a smooth function with $\chi(x)=j$ for $x$ near $j=0,1$.

By the previous proposition and the $C^*$-algebraic Atiyah--Patodi--Singer index theorem
\begin{align*}
\lefteqn{\ind_{\tau}(D^+)}\\ 
&=(2 \pi i)^{n/2}\t(\int \hat A(M)\ch(E/S)\ch(\cF)) -\eta_{\tau}(D_{\ra}+A)+ \ind_{\tau}(D^I)\\
&=(2\pi i)^{n/2}\t(\int \hat A(M)\ch(E/S)\ch(\cF)) - \eta_{\tau}(D_{\ra}+A)-\spfl((D_{\ra}+\chi(t)A)_{t\in [0,1]})\ .
\end{align*}
Here the last equation follows from \cite[Theorem 2.11]{aw} by using that $\ind_{\tau}(D^I)=\ind_{\tau}(D^I -\chi A)$.
By taking the limit $\ep \to 0$ in \cite[Corollary 8.11]{cpJLO} we get that
$$\spfl((D_{\ra}+\chi(t)A)_{t\in [0,1]})=\eta_{\tau}(D_{\ra})-\eta_{\tau}(D_{\ra}+A)+\frac 12\tau(P_{\Ker D_{\ra}}) \ .$$
\end{proof}

\section{Twisted $\eta$- and $\rho$-invariants}
\label{etarho}

%\medskip
%We want to define it for a regular $\Delta$-covering, not necessarely for the universal covering. (maybe not necessary, enough for $\Delta=\G$ here)
%We proceed as follows.

%Let $[c]\in H^2(\Delta, U(1))$, and $r \colon M\rightarrow B\Delta$ and $\hat{M}$ the covering given by $r$. $\hat{M}=\cup_{\alpha\in A} \hat{M}_\alpha$. Each $\hat{M}_\alpha\rightarrow M$ is a regular $\Delta_\alpha=\Aut(p_\alpha)$ covering. 

%\medskip

%\emph{first case:} 
%\emph{suppose that $\hat M $ is connected.}

%Let $\Mt$ be the universal covering of $\widehat{M}$, then we have a map
%$$
%M\rightarrow B\Gamma\rightarrow B\Delta
%$$

%The projection $p\colon \Gamma\rightarrow \Delta$ induces  a map $p_*\colon C^*(\Gamma, p^*\s)\rightarrow C^*(\Delta, \sigma)$. We build the Mathai bundle as before $
%\mathcal V_{\Gamma, p^*\s}=\displaystyle\frac{L\otimes C^*(\Gamma, p^*\s) }{\sim}\rightarrow M$
%and define\begin{equation}
%\label{}
%\mathcal V_{\Delta, \s}:= \mathcal V_{\Gamma, p^*\s}\otimes_{p_*} C^*(\Delta, \s)
%\end{equation}
%which equivalently can be described as $\displaystyle \frac{L\otimes C^*(\Delta , \s)}{A_\g\otimes T_{p(\g)}}$.

%We perform the construction of the connection on $\mathcal V_{\Delta, \s}$ as before.

Let now $\G$ be a discrete group, and ${\mathbf c}=(c_1, c_2 \dots c_k) \in H^2(B\G, \Q)^k$ be cohomology classes of degree $2$. 
We consider a closed connected odd-dimensional manifold $M$ with a map $f\colon M \to B\G$ classifying the universal covering, and  a Dirac operator $D$ associated to a Dirac bundle $E\to M$. As in \S \ref{twistdir}, we denote by
  $D_{\mathcal V^{\s^s}}$ the induced twisted Dirac operator acting on the sections of $E\otimes \mathcal V^{\s^s}$.
  
Let $\t_s$ be a family of positive (finite) traces on $C^*(\G, \s^s)$ which are defined for $s\in \Q^k$ near $0$ and are invariant under parametrized projective automorphisms.

\begin{definition}
\label{deftwistedeta}
Using \eqref{etavN} we define the family of von Neumann twisted $\eta$-invariants by
\begin{equation*}
\etast(D_{\mathcal V^{ \s^s}})=\frac{1}{2\sqrt{\pi}}\int_0^\infty \Tr_{\tau_s} \left(D_{\mathcal V^{ \s^s}}e^{-t (D{}_{\mathcal V^{\s^s}})^2}\right)\frac{dt}{\sqrt t}\;\;\in \C \ .
\end{equation*}

Assume that $D_{\mathcal V^{ \s^s}}$ is invertible. Then the definition works also if $\tau_s$ is a family of delocalized traces. Recall that a trace $\tau_s$ on $C^*(\G,\s^s)$ is called delocalized if $\tau_s(\delta_e)=0$. 

If $\tau_s$ is delocalized, we call the associated twisted $\eta$-invariant a \emph{twisted $\rho$-invariant} and write $\rho_{\tau_s}(D_{\N^{\s^s}})$.  

If $D_{\mathcal V^{ \s^s}}$ is not invertible, twisted $\rho$-invariants can be defined by considering the difference between two positive normalized traces.

Both, the $\eta$- and the $\rho$-invariant, are defined for manifolds which are not necessarily connected but whose components have the same fundamental groups: in fact, one can apply the constructions to each component and add up the so gained contributions.
\end{definition}

In general, the twisted $\eta$- and $\rho$-invariants depend on the choices made in \S \ref{para-constr}. We shall now analyze the particular case of the spin Dirac operator in presence of positive scalar curvature.

%\subsection{Twisted $\eta$- and $\rho$-invariants of the spin Dirac operator
% for metrics of positive scalar curvature}
\subsection{Case of the spin Dirac operator
 for metrics of positive scalar curvature}

In the following we assume that $(M,g)$ is a closed connected spin manifold with positive scalar curvature $\scal (g)>0$. We denote by $\Di$ the spin Dirac operator. The twisted spin Dirac operator $\Di_{\mathcal V^{\s^s}}$ satisfies the following property, which is a direct consequence of the Lichnerowicz formula.

\begin{lem} \textrm{\cite[p. 19]{Ma1}} 
\label{lem-invert} 
There exists $\ep_0>0$ such that for all $s\in \Q^k$ with $|s| <\ep_0$ the operator $\Di_{\mathcal V^{\s^s}}$ is invertible.
\end{lem}

For each $s\in B_{\ep_0}=\{s\in \Q^k,  |s| < \ep_0 \}$, let now $\tau_s$ be a (not necessarily positive) trace on $C^*(\G,\s^s)$ which is invariant under parametrized projective automorphisms.

Note that by the previous lemma the $C^*$-algebraic definition of the $\eta$-invariant \eqref{eta-tauC*} applied to $\Di_{\mathcal V^{\s^s}}$ is well-defined. 

In order to avoid the choice of $\ep_0$ (which depends on the geometric data), we define the twisted $\eta$-invariant as a germ.

\begin{definition}
\label{etadef}

Consider the direct system $\{\eta^U\}_{U\subseteq B_{\ep_0}}$ of (arbitrary set theoretic) maps $\eta^U\colon U\to \C$ defined by $\eta^U(s)=\etast(\Di_{\mathcal V^{ \s^s}})$, where $U\subseteq B_{\ep_0}$ varies in the directed set of neighborhoods of $s=0$.

 We define the \emph{twisted $\eta$-invariant} of the operator $\Di_{\mathcal V^{\s^s}}$ as the direct limit of $\eta^U$, \emph{i.e.} the germ at $s=0$ of $\etast(\Di_{\mathcal V^{ \s^s}})$. 
 Thus
$$\etast(\Di_{\mathcal V^{ \s^s}}) \in \varinjlim_{0 \in U \subset B_{\ep_0}}\G(U \to \C) \ .$$
Here $\G(U \to \C)$ is the set of arbitrary set theoretic maps from $U$ to $\C$.

For a delocalized family of traces $\tau_s$ we write $\rhoct(M,g,f)$ for the associated $\rho$-invariant, which we also considered as a germ.

\end{definition}

\subsubsection{Dependence on the choices}
\label{depchoic}
We assume that $M$ is connected. By Lemma \ref{choices-conn-s} the $\eta$-invariant $\etast(\Di_{\mathcal V_{ \s^s}})$ depends on the choice of $(\o_1,\o_2 \dots, \o_k)$ but not on the choices of $(\eta_1,\eta_2 \dots \eta_k)$ and $(\psi^1_\g,\psi^2_{\g} \dots \psi^k_{\g})$. We compute the variation for two choices of $\o_j, \o_j',~j=1,\dots k$. Let $\Di_{\mathcal V^{\s^s}}$, $\Di_{\mathcal V^{\s^{s'}}}$ be the two corresponding twisted spin Dirac operators. For each $j$ choose $\phi_j \in \O^1(M,\R)$ with $\o_j'-\o_j=d\phi_j$ and define $\phi_s=s_1 \phi_1 +s_2 \phi_2 + \dots s_k \phi_k$. 

Let $\chi\colon [0,1] \to [0,1]$ be a smooth cut-off function with $\chi(t)=0$ for $t\in [0, \frac{1}{4}]$, and  $\chi(t)=1$ for $t\in [\frac{3}{4}, 1]$. 

\begin{prop}
We have the following variation formula.
$$
\etast(\Di_{\mathcal V^{\s^{s'}}})-\etast(\Di_{\mathcal V^{ \s^s}})=(2\pi i)^{\frac{\dim M+1}{2}}\int_0^1 \hat A(M)\tau_s \exp(-i(\omega_s+\chi d\phi_s+\chi dt\wedge \phi_s)) \ .
$$ 
In particular, the twisted $\rho$-invariants are independent of the choice of the $\o_j$.
\end{prop}

\begin{proof}
As in the proof of Lemma \ref{choices-conn}, we set $\eta'_s=\eta_s+p^*\phi_s$ and $(\psi^s_{\g})'=\psi^s_{\g}$. Then $\nabla^{\sigma^{s'}}=\nabla^{\sigma^{s}}+i\phi_s$ and on the covering $\nabla'=\nabla+ip^* \phi_s=d+ i(\eta_s+p^*\phi_s)$.
We pull the bundle $\N^{\sigma^{s}}$ back to the cylinder $Z=[0,1] \times M$ and consider on it the connection $\nabla^Z=dt\frac{\partial }{\partial t}+\nabla^{\sigma^{s}}+i \chi(t)\phi_s$. 
We compute the transgressed Chern character (also called Chern--Simons form) 
$$\cs_{\tau_s}(\nabla^{\sigma^{s}},\nabla^{\sigma^{s'}})=\int_0^1\ch_{\tau_s}(\nabla^Z)=\int_0^1\tau_s (\tralg e^{-(\nabla^Z)^2})\ .$$ 
The curvature of $\nabla^Z$ can be computed using the corresponding connection $\tilde \nabla^Z$ on the covering. We get
\begin{eqnarray*}
(\widetilde \nabla^Z)^2&=&(dt\frac{\partial }{\partial t} +d_{\widetilde M}+i\eta_s+i\chi p^*\phi_s)^2\\
&=&i p^*\omega_s+i\chi(d_{\widetilde M} p^*\phi_s)+i\chi' dt \wedge p^*\phi_s \ .
\end{eqnarray*}
Thus $(\nabla^Z)^2=i(\omega_s+\chi d_M\phi_s+\chi'dt \wedge \phi_s) \in \O^2(Z)$ and 
$$\cs(\nabla^{\sigma^{s}},\nabla^{\sigma^{s'}})=\int_0^1 \tau_s \exp(-i(\omega_s+\chi d_M\phi_s +\chi'dt \wedge \phi_s)) \in \O^{odd}(M) \ .$$
 
Since $Z$ has positive scalar curvature, the Dirac operator $\Di_Z$ associated to $\nabla^Z$ is invertible for $s$ small enough. The $C^*$-algebraic Atiyah--Patodi--Singer index theorem applied to $\Di_Z$ implies that (as a germ at $s=0$)
$$\etast(\Di_{\mathcal V^{\s^{s'}}})-\etast(\Di_{\mathcal V^{ \s^s}})=(2\pi i)^{\dim Z/2}\int_Z \hat A(M)\ch_{\tau_s}(\nabla^Z)\ .$$
Now the assertion follows.
\end{proof}

\begin{rem} The condition of positive scalar curvature is crucial for this formula. Otherwise there might be an additional contribution from the index on the cylinder.
\end{rem}

\subsection{Examples of invariant traces}

For notational simplicity we deal with a fixed multiplier $\s$ in the following and construct examples of traces on $C^*(\G,\s)$ which are invariant under projective automorphisms. Analogously, one obtains invariant parametrized versions.

\subsubsection{Perfect groups}

If $\G$ is perfect, then any group homomorphism $\chi\colon \G \to U(1)$ is trivial and thus any trace on $C^*(\G,\s)$ is invariant. 

If one considers the parametrized version as constructed in \S \ref{para-constr}, then it is enough to assume that $\G$ has finite abelization since then any group homomorphism $a\colon \G \to \R$ is trivial. Thus, when one considers parametrized versions of the following examples, one may relax the condition of perfectness by assuming only finite abelization.

\subsubsection{Left regular trace}\label{lrt}
We denote by $\tr_{(2)}^{\s}$ the positive faithful trace on $C^*(\G,\s)$ defined by $\tr_{(2)}^{\s}(\sum_\g a_\g\de_\g)=a_e$.
If $b_z \colon C^*(\G,\s') \to \cgs$ is a projective isomorphism, then $\tr_{(2)}^{\s'}(a)=\tr_{(2)}^{\s}(b_z(a))$, because $z(e)=1$.

In this case we write $\eta^{\s}_{(2)}$ for the twisted $\eta$-invariants and, similarly, $\ind^{\s}_{(2)}$ for the twisted index.

\subsubsection{Traces on products}\label{tr-prod}

Let $\G$ be a perfect group.
Consider a product $\G\times G$ and denote by $\pi_2\colon \G\times G\rightarrow G$ the projection onto the second factor. Let $\s$ be a multiplier on $G$ and 
$\pi_2^*\s$ the induced multiplier on the product. 
There is a canonical map $C^*(\G\times G,\pi_2^*\s ) \rightarrow C^*(\G) {\otimes}C^*(G,\s)$.
(We always use the spatial tensor product.)

Given a trace $\tau$ on $C^*(\G)$, we get a trace  $\tau\otimes \tr^{\s}_{(2)}$ on  $C^*(\G)\ten C^*(G,\pi_2^*\s)$ and by pull back on $C^*(\G\times G,\pi_2^*\s)$. It is explicitely given by 
$$(\tau\otimes \tr^{\s}_{(2)})( \sum_{(\g,g)}a_{\g,g}\delta_{(\g,g)})=\sum_{\g \in \G} a_{\g,e} \tau(\de_\g) \ .$$

Let  $\chi \colon  \G\times G\rightarrow U(1)$ be a homomorphism. Then $\G\rightarrow U(1),~\g\mapsto \chi(\g, e)$ is a homomorphism. Since $\G$ is perfect, it is trivial. %Then $\theta(\g, g)=\theta(\g, e)\theta(e,g)=\theta(e,g)$.
The trace $\tau\otimes \tr^{\s}_{(2)}$ is well defined because

$$(\tau\otimes \tr^{\s}_{(2)})( \sum_{(\g,g)}a_{\g,g}\chi(\g,g)\delta_{(\g,g)})=\sum_{\g \in \G}a_{\g,e}\chi(\g,e) \tau(\de_\g)=\sum_{\g \in \G} a_{\g,e} \tau(\de_\g)
 \ .$$

Now we are going to apply this construction to prove a product formula for twisted $\eta$-invariants. Compare with the product formula for higher $\eta$-invariants \cite[\S 2]{lpetapos}. 

Let $M=L\times N$ be the product of two closed connected spin manifolds with $\dim N=2n$ , $\dim L=2l+1$, with fundamental groups $\pi_1(N)=G$ and  $\pi_1(L)=\G$ and maps $f_N\colon N \to BG$, $f_L \colon L \to B\G$ inducing universal coverings. We assume that $L$ has positive scalar curvature and the metric on $N$ is such that the induced metric on $M$ has positive scalar curvature as well. Set $f_M=f_L \times f_N$. 

Let $(c_1,c_2 \dots c_k)\in H^2(G, \Q)^k$ and choose $\psi_g^s, \eta_s, \omega_s$ for $N$ as in \S \ref{para-constr}. The induced $C^*(G,\s^s)$-bundle is denoted by $\N^{\s^s}$ and the induced twisted spin Dirac operator is denoted by $\Di^N_{\N^{\s^s}}$. Furthermore we write $\Di^L_{\N}$ for the spin Dirac operator on $L$ twisted by the flat Mishenko--Fomenko $C^*\G$-bundle $\N_L$.

By pulling back $\eta_s, \omega_s$ to $M$ via the projection $p_2 \colon M \to N$ and by setting $\psi_{(\g,g)}^s:=\psi_\g^s\circ p_2 \in \Cinf(M)$ we define the twisted spin Dirac operator 
$\Di^M_{\N^{\pi_2^*\s^s}}$ on $M$. 

\begin{prop}[Product formula for twisted eta invariants]\
\label{prod}
Assume that $\G$ has finite abelization.

Let $\tau$ be a trace on $C^*(\G)$ and $\tau_s=\t\otimes \tr_{(2)}^{\s^s}$ the induced trace on $C^*(\G\times G, \pi_2^{*}\s^s)$.

Then, with the definitions as above,
$$
\eta_{\t_s}(\Di^M_{\N^{\pi_2^*\s^s}})=\eta_{\tau}(\Di^L_{\N})\cdot \ind^{\s^s}_{(2)}(\Di^N_{\N^{\s^s}}) \ .
$$
\end{prop}

\begin{proof}
Consider the bundle 
$$\mathcal U=\mathcal V_M^{\pi_2^*{\s^s}}\otimes_{C^*(\G\times G, \pi_2^*{\s^s})}\left(C^*(\G)\otimes C^*(G,\s^s)\right)$$
and write $\Di_{\mathcal U}$ for the spin Dirac operator twisted by $\mathcal U$. 

Observe that $\eta_{\tau_s}(\Di^M_{\mathcal V^{\pi_2^*{\s^s}}})=\eta_{\tau_s}(\Di_{\mathcal U})$. 
The spinor bundle on $L\times N$ twisted by $\mathcal U$ satisfies
$$
S_{L\times N}\otimes \mathcal U=\left(S_L\otimes \mathcal V_L\right)\boxtimes \left(S_N\otimes \mathcal V^{\s_s}_N\right) \ .
$$ 

If $\grad_N$ denotes the chirality grading on $S_N$,  then $\Di_{\mathcal U}=\grad_N \Di^L_{\N}+ \Di^N_{\N^{\s^s}}$.

Thus 
\begin{eqnarray*}
\eta_{\t_s}(\Di^M_{\N^{\pi_2^*\s^s}})&=&\eta_{\t_s}(\Di_\mathcal U)= \frac{1}{\sqrt{\pi}}\int_0^\infty \tau_s \tralg \left(\Di_{\mathcal U}e^{-t (\Di{}_{\mathcal U})^2}\right)
\frac{dt}{\sqrt t}\\
&=&\frac{1}{\sqrt{\pi}}\int_0^\infty \tau \tralg \left(\Di^L_{\N} e^{-t (\Di^L_{\N})^2}\right) \tr^{\s^s}_{(2)} \tralg \grad_N e^{-t (\Di^N_{\N^{\s^s}})^2}\frac{dt}{\sqrt t}\ .
\end{eqnarray*}

Now the assertion follows because the term $\tr^{\s^s}_{(2)} \tralg \grad_N e^{-t (\Di^N_{\N^{\s^s}})^2}$ is independent of $t$ and equals $\ind_{(2)}^{\s^s}(\Di^N_{\N^{\s^s}})$, by the McKean--Singer formula.
\end{proof}

One may obtain similar product formulas by applying the twist only to the group $\G$ or by using twists for both groups. We leave their consideration to the interested reader.

\subsubsection{Traces associated to exact sequences}
This construction is a generalization of the previous one. 

Let $0 \rightarrow  G \rightarrow  \G \stackrel{\pi}{\rightarrow} H \to 0$ be an exact sequence. We assume that any group homomorphism $\chi\colon \G \to U(1)$ restricts to the trivial homomorphism on $G$. This is the case, for example, if $G$ or $\G$ are perfect. 
Let $\s$ be a multiplier on $H$ and 
$\pi^*\s$ the induced multiplier on $\G$. 
We can pull back the trace $\tau_{(2)}^{\s}$ on $C^*(H,\s)$ to a trace $\tau^{\pi^*\s}_H$ on  $C^*(\G,\pi^*\s )$.

Let $\chi\colon  \G \rightarrow U(1)$ be a homomorphism.  
The trace $\tau^{\pi^*\s}_H$ is well defined because

$$\tau^{\pi^*\s}_H( \sum_{\g}a_\g\chi(\g)\delta_{\g})=\sum_{h \in H}\tau_{(2)}^{\s}(\delta_h)\sum_{(\g:\pi(\g)=h)} a_{\g}\chi(\g)=\sum_{\g\in G} a_{\g}\chi(\g)=\sum_{\g\in G} a_{\g} \ .$$

\subsubsection{Traces associated to unitary representations}
\label{tr-unitary}
We further generalize the previous example. The notation is as before.
Let $u\colon \G \to U(n)$ be a representation. There is an induced homomorphism
 $$u\ten \pi\colon C^*(\G,\pi^*\s) \to M_n(\C) \ten C^*(H,\s),\;~\delta_\g \mapsto u(\g) \ten \delta_{\pi(\g)}\ .$$
Using the standard trace $\tr$ on $M_n(\C)$ we get a trace
$$\tau^\s_{u,H}:=(\tr \ten \tau^\s_{(2)})\circ (u \ten \pi)$$ on $C^*(\G,\pi^*\s)$. Then for a  homomorphism  $\chi\colon  \G \rightarrow U(1)$ we have
$$\tau^\s_{u,H}( \sum_{\g\in \G}a_\g\chi(\g)\delta_{\g})=\sum_{h \in H}\tau_{(2)}^{\s}(\delta_h)\sum_{(\g:\pi(\g)=h)} a_{\g}\chi(\g)\tr(u(\g)) 
=\sum_{\g\in G} a_{\g}\tr(u(\g)) \ .$$
We write $\tr_1$ for this trace if $H$ is trivial and $u\colon \G \to U(1)$ is the trivial representation.
\medskip

We can use this construction to define $\rho$-invariants as follows.

Let $u_1, u_2\colon \G \to U(n)$ be two unitary representations.

Fix $\mathbf c=(c_1, c_2 \dots c_k) \in H^2(BH,\Q)^k$ and let $M$ be closed connected odd-dimensional manifold and $f_M\colon M \to B\G$ a map classifying the universal covering. We can define twisted $\rho$-invariants as in Def. \ref{deftwistedeta} using as delocalized trace $\tau_s:=\tau^{\s^s}_{u_1, H}-\tau^{\s^s}_{u_2, H}$. The induced $\rho$-invariant generalizes the Atiyah--Patodi--Singer $\rho$-invariant.

If $H$ is the trivial group, then there is no twist and this is the ordinary Atiyah--Patodi--Singer $\rho$-invariant. 

If $H=\G$, then one checks that $\tau^{\s^s}_{u_1,H}=\tau^{\s^s}_{u_2,H}$. Thus the induced $\rho$-invariant vanishes. This reproves and generalizes \cite[Lemma 6.2]{hu}.

\section{Bordism invariance of $\rho^{\mathbf c}_\t$}
\label{bordism_inv}

We introduce a new equivalence relation for positive scalar curvature metrics on closed manifolds. 

\begin{definition}
Let $(M,g_M, f_M)$ be a triple consisting of a closed spin manifold $M$ with $\pi_1(M)=\G$ endowed with a metric $g_M$ of positive scalar curvature and a reference map $f_M \colon M \to B\G$ inducing the universal covering. Two such triples $(M,g_M,f_M),~ (N,g_N,f_N)$ are called \emph{strongly $\G$-bordant} if there exists a triple $(W, g, f_W)$ where $W$ is a compact spin manifold with positive scalar curvature metric $g$ such that $(\partial W,g|_{\partial W})=(M,g_M)\cup (-N, g_N)$, and $f_W\colon W \to B\G$ is a reference map classifying the universal covering and extending the reference maps $f_M$, $f_N$. 
\end{definition}

Recall that the definition of $\G$-bordism is analogous, without requiring that the reference maps induce the universal covering.

Strong $\G$-bordism might not be an equivalence relation (in contrast to $\G$-bordism). 

For a fixed closed connected spin manifold $M$ with fixed classifying map $f_M\colon M\to B\G$, we say that two positive scalar curvature metrics $g, g'$ on $M$ are strongly $\G$-bordant if the triples $(M,g,f_M)$ and $(M,g',f_M)$ are \emph{strongly $\G$-bordant}. 
This is indeed an equivalence relation. It is weaker and therefore more flexible than concordance. Furthermore it is stronger than $\G$-bordism. 

In high dimensions strong $\G$-bordism and $\G$-bordism of positive scalar curvature metrics agree:

\begin{lem}
\label{strong_bordism}
Assume that $(M,g_M,f_M)$ and $(N,g_N,f_N)$ with $\dim M=\dim N \ge 4$ are $\G$-bordant and $f_M, f_N$ induce the universal covering. Then $(M,g_M,f_M)$ and $(N,g_N,f_N)$ are strongly $\G$-bordant.
\end{lem}

\begin{proof}
Let $(W,g_W,f_W)$ be a $\G$-bordism between $(M,g_M,f_M)$ and $(N,g_N,f_N)$. Since $\dim W \ge 5$ we may obtain by spin surgery of codimension $\ge 3$ in the interior of $W$ that $f_W$ induces the universal covering. By the surgery technics of Schoen--Yau and Gromov--Lawson there is a metric of positive scalar curvature on the new bordism which restricts to the original metric on the boundary (see \cite[\S 2.2]{walsh} for a detailed treatment).
\end{proof}

\begin{prop}
\label{bord-inv}
If $(M,g_M,f_M)$ and $(N,g_N,f_N)$ are strongly $\G$-bordant, then $\rhoct(M,g_M,f_M) =\rhoct(N,g_N,f_N)$.
%\begin{com}
%Should we change the statement of Prop. \ref{bord-inv} to a statement about homomorphism property of $\rho_\s$, along the lines of \cite[prop.2.7]{ps}?
%\end{com}
\end{prop}

\begin{proof} Let $(W, g_W, f_W)$ be a triple realizing the strong--$\G$-bordism.

Without loss of generality we may assume that $W$ is connected.

Now for $s$ sufficiently small the operator $\Di^W_{\mathcal V^{\s^s}}$ constructed in
 section \ref{bcgs-bundle} is invertible by Lichnerowicz formula. 
 
Hence by the $C^*$-algebraic Atiyah--Patodi--Singer index theorem we have that
\begin{equation}
\label{APS}
0=(2\pi i)^{\dim W} \tau_s\int_W\hat{A}(W) e^{-i\omega_s}-\etast(\Di^M_{\mathcal V^{\s^s}})+\etast(\Di^N_{\mathcal V^{\s^s}})
\end{equation}

Since $\tau_s$ is delocalized and therefore $\tau_s\left(\displaystyle\int_W\hat{A}(W)e^{-i\omega_s}\right)=0$,  we obtain the result.

\end{proof}

\section{Applications}
\label{app}

Crucial for our applications is the following well-known result which, in addition to the cited sources, is based on ideas of Gromov--Lawson:

\begin{lem}[Bordism theorem]\label{bordlemma}
Assume that $M$ is a closed spin manifold with $\dim M \ge 5$ endowed with a map $f_M\colon M \to B\G$ classifying the universal covering. Let $N$ be a closed spin manifold with a map $f_N \colon N \to B\G$ such that $(M,f_M)$ and $(N,f_N)$ are bordant in $\O^{spin}(B\G)$.

Then there is a spin bordism $(W, f_W\colon W \to B\G)$ between $(M,f_M)$ and $(N,f_N)$ such that $f_W$ classifies the universal covering of $W$.
 
Furthermore, if $M$ is connected and $N$ is endowed with a metric of positive scalar curvature, then there is a metric of positive scalar curvature on $W$ (which is of product structure near the boundary) extending the metric on $N$.
\end{lem}

\begin{proof}
We follow the arguments from the proof of the \emph{Bordism Theorem} in \cite{rs}. Let $(W,f_W)$ be a  spin-bordism between $(M,f_M)$ and $(N,f_N)$. By changing $W$ by spin surgeries in the interior one may obtain that $f_W$ is a 3-equivalence and thus the inclusion $M \to W$ is a 2-equivalence. (The arguments have been detailed in \cite{miy}.)

Thus one can find a Morse function $h \colon W \to [0,1]$ with $h(N)=0, h(M)=1$ whose critical points have index $\ge 3$. It follows by \cite[Theorem 3.13]{mi} that $W$ is the trace of a surgery on $N$ of codimension $\ge 3$. 

If $N$ is endowed with a metric of positive scalar curvature, then, by \cite{ga}, there is a metric of positive scalar curvature on $W$ extending the metric on $N$. (A detailed proof of this fact can be found in \cite[Theorem 2.2]{walsh}.)
\end{proof}

For the following applications we adapt the methods from \cite{lpetapos} where similar results were proven using higher $\rho$-invariants. 

\begin{prop}
\label{infbordism}
Let $M$ be a closed connected spin manifold with a metric $g_M$ of positive scalar curvature and with fundamental group $\Gamma=\Gamma_1 \times \Gamma_2$ such that $\Gamma_1$ has torsion and its abelization is finite. 

Let $f_M\colon M \to B\Gamma$ be a map classifying the universal covering.

Assume that $\dim M=4k-1+4m-2$, for some $k,m\in \bn$ with $k \ge 2$ and that there is a closed connected manifold $N$ of dimension $4m-2$ and a map $f_N\colon N \to B\Gamma_2$ classifying the universal covering of $N$ with $\int_N \hat A(N)\wedge f_N^*c \neq 0$ for some $c$ in the subring of $H^*(B\Gamma_2, \Q)$ generated by $H^{2}(B\Gamma_2, \Q)$.

Then there are infinitely many $\G$-bordism classes of metrics of positive scalar curvature on $M$.
\end{prop}

\begin{proof} 
First we note that we can find a $(4k-1)$-dimensional closed connected manifold $L$ and map $f_L\colon L \to B\G_1$ with the following properties:
\begin{itemize}
\item $f_L$ induces the universal covering on $L$,
\item
 $[(L,f_L)]=0 \in \O^{spin}(B\G_1)$,
\item $L$ is endowed with a metric $g_L$ of positive scalar curvature,
\item the $L^2$-$\rho$-invariant $\rho_{(2)}(L,g_L)$ is nontrivial. 
\end{itemize}

Such a manifold can be constructed as follows. Take a $(4k-3)$-dimensional closed connected spin manifold $K$ with $\pi_1(K)=\G_1$ and a map $f_K\colon K \to B\G_1$ inducing the universal covering on $K$. Let $L=K\times S^2$ and $f_L=f_K \circ p_1$ where $p_1:K\times S^2 \to K$ is the projection. Since $S^2$ has a metric of positive scalar curvature, $L$ has one as well. By \cite[Theorem 1.3]{pstorsion} the manifold $L$ has infinitely many metrics of positive scalar curvature with mutually different $\rho_{(2)}$, in particular one with nonvanishing $L^2$-$\rho$-invariant.

The product $[(L \times N,f_L \times f_N)]$ vanishes in $\O^{spin}(B(\G_1\times \Gamma_2))$. Here we identified $B(\G_1\times \G_2)$ with $B(\G_1)\times B(\G_2)$. 

For $j \in\bn$ we consider the manifold $Q_j:=M\cup  \bigcup_{i=1}^j (L\times N)$. Using the metrics $g_M, g_L$ and an appropriate metric $g_N$ on $N$, one may define a metric of positive scalar curvature $h_j$ on $Q_j$.  The maps $f_N, f_M, f_L$ induce a map $f_k\colon Q_k \to  B(\G_1\times \G_2)$.

Then, by Lemma \ref{bordlemma}, for any $j\in\bn$ there is spin bordism $(W,f_W\colon W \to B\G)$ between $(M,f_M)$ and $(Q_j, f_j)$ such that $f_W$ induces the universal covering on $W$.  
 
Furthermore $W$ is endowed with a metric of positive scalar curvature extending the metric on $Q_j$ and inducing a (new) metric of positive scalar curvature $g^j_M$ on $M$.
 
Without loss of generality we may assume that there is $(c_1, c_2,\dots, c_l) \in H^2(B\G_2,\Q)^l$ such that $c=c_1 \cup c_2 \dots \cup c_l$. Let $\s^s$ be as arising from the construction in \S \ref{para-constr} when applied to $(N,f_N)$.  We take the delocalized trace $\tau=\tr_{(2)}-\tr_1$ on $C^*\G_1$, and set $\tau_s=\tau\otimes \tr_{(2)}^{\s_s}$ as in \S \ref{tr-prod}. We denote by ${\mathbf c}$ the pullback of $(c_1, c_2,\dots, c_l)$ to $H^2(B\G,\Q)^l$. By Prop. \ref{bord-inv} we have that $$\rhoct(M, g_M^j,f_M)=\rhoct(Q_j, h_j,f_j)=\rhoct(M,g_M,f_M)+j\rhoct(L\times N,g_L+g_N,f_L \times f_N) \ .$$ 

The product formula in Theorem \ref{prod} implies that $\rhoct(L\times N)$ is the germ of a polynomial in $s_1, s_2 \dots s_l$. Furthermore, the polynomial is not identically zero since a nontrivial multiple of $\rho_{(2)}(L,g_L) \int_N \hat A(N)\wedge f_N^*c$ is one of its coefficients. Thus $\rhoct(L\times N)\neq 0$. The germs $\rhoct(M, g_M^j, f_M)$ are then all distinct and the metrics $g_M^j$ are mutually not strongly $\G$-bordant. Now the result follows from Lemma \ref{strong_bordism}.
\end{proof}

Examples are easy to construct: Let $\Gamma_2$ be the fundamental group of a closed Riemannian surface $N \neq S^2$. Then $N$ is its classifying space. Since $H^2(N,\Q)=\Q$, the condition on $N$ is fulfilled with $m=1$.

Furthermore, for higher dimensions we get the following result from \cite{lpetapos}.

\begin{lem}
\label{ahat}
Assume that $\G$ is a finitely presented group and that there is $c\neq 0$ in $H^k(B\G,\Q)$. Then for $m \ge 5$ with $m-k=0 \mod 4$ there is an $m$-dimensional closed connected spin manifold $N$ and a map $f_N\colon N \to B\G$ classifying the universal covering such that $\int_N\hat A(N)\cup f^*c \neq 0$.
\end{lem}

\begin{proof}
Without the condition that $f_N$ classifies the universal covering the result has been proven as the crucial step in the proof of \cite[Lemma 5.2]{lpetapos}. Now the assertion follows by applying spin surgery to $(N,f_N)$. This works because spin surgery does not change the class in $\O_*^{spin}(B\G)$ and since the map $\O_*^{spin}(B\G) \to \C,~ [(N,f_N)] \mapsto \int_N\hat A(N)\cup f^*c$ is well-defined.
\end{proof} 

%\begin{proof} 
%In Prop. 2.4 in \cite{ps} if $x$ is in $\ker(Pos_{4k-1}^{spin}(\Gamma_1) \to \Omega_m^{spin}(B\Gamma_1))$, then $x \times N$ is in $\ker(Pos_{4k-1}^{spin}(\Gamma) \to \Omega_m^{spin}(B\Gamma))$. Furthermore it holds for the twisted $\rho$-invariant by the product formula that $$\rho_c(x \times N +[g_M])=\rho_{(2)}(x)\int_N \hat A(N)\wedge f^*c + \rho_c([g_M]) \ .$$

%Now the proof is as the proof of Theorem 2.25 in \cite{ps}, but instead of $k \in K$ one takes $k \times N$ and uses the previous formula.
%\end{proof}

A similar result as Prop. \ref{infbordism} may be obtained for even-dimensional $M$ as we show now. Note that this is a crucial difference to the Atiyah--Patodi--Singer $\rho$-invariant and the Cheeger--Gromov $\rho$-invariant, which yield no information for even-dimensional manifolds. Thus, our twisted $\rho$-invariants behave more like higher $\rho$-invariants. So far, the precise relation is unclear, as we remarked in Question \ref{conhigh}. In some respect, higher $\rho$-invariants are better behaved (namely $\G$-bordant) and more general (they do not require finite abelization or similar conditions) than our twisted ones. However, at the moment they can only be defined for groups fulfilling certain growth conditions.

\begin{prop}
Let $M$ be a closed connected spin manifold with a metric $g_M$ of positive scalar curvature and with fundamental group $\Gamma=\Gamma_1 \times \Gamma_2$ such that $\Gamma_1$ has torsion and its abelization is finite. 

Let $f_M \colon M \to B\Gamma$ be a map classifying the universal covering.

Assume that $\dim M=4k-1+2m-1$, for some $k,m\in \bn$ with $k \ge 2$ and that there is a closed connected spin manifold $N$ of dimension $2m-1$ and a map $f_N\colon N \to B\Gamma_2$ classifying the universal covering with $\int_N \hat A(N)\cup f_N^*c \neq 0$ for some $c$ in the subring of $H^*(B\Gamma_2, \Q)$ generated by $H^{2}(B\Gamma_2, \Q)$ and $H^1(B\Gamma_2,\Q)$.

Then there are infinitely many $\G$-bordism classes of metrics of positive scalar curvature on $M$.
\end{prop}

\begin{proof} 
Without loss of generality we may assume that there are $c_1 \in H^1(B\G_1,\Q)$ and $c_2,\dots, c_l\in H^2(B\G_2,\Q)$ such that $c=c_1 \cup c_2 \dots \cup c_l$. 

Let $L$ be the $(4k-1)$-dimensional manifold from the previous proof. Then the product $[L\times N,f_L \times f_N]$ vanishes in $\O_*^{spin}(B\G)$. 

As in the previous proof we get a metric of positive scalar curvature $h_j$ on $Q_j:=M\cup  \bigcup_{i=1}^j (L\times N)$ and a strong $\G$-bordism $(W,g_W,f_W)$ between $(Q_j,h_j,f_j)$ and $(M,g_M^j,f_M)$.

Let $S^1$ be endowed with the standard metric $g_{S^1}$. Then the metric $g^j_M + g_{S^1}$ is of positive scalar curvature on $M \times S^1$ and we also get a metric of positive scalar curvature $h_j+ g_{S^1}$ on $Q_j \times S^1$.  We use the pullback ${\mathbf c}\in H^2(B(\G \times \Z),\Q)^l=H^2(B\G\times S^1,\Q)^l$ of $(c_1 \cup \vol_{S^1}, c_2, \dots, c_l) \in H^2(B\G_2\times S^1,\Q)^l$ and the delocalized trace $\tau_s=(\tr_{(2)}-\tr_1)\otimes \tr_{(2)}^{\s^s}$, where $\tr_{(2)}^{\s^s}$ acts on $C^*(\G_2 \times \Z,\s^s)$,
for the definition of the twisted $\rho$-invariants. 

We have a strong $\G \times \Z$-bordism $(W \times S^1, g_W + g_{S^1},f_W \times \id \colon W \times S^1 \to B(\G \times \Z))$ between $(Q_j \times S^1,h_j+ g_{S^1},f_j \times \id)$ and  $(M \times S^1, g_M^j+ g_{S^1},f_M\times \id)$. Thus 
\begin{eqnarray*}
\lefteqn{\rhoct(M\times S^1,g^j_M + g_{S^1}, f_M \times \id) }\\
&=& \rhoct(Q_j \times S^1,h_j + g_{S^1},f_j \times \id) \\
&=&\rhoct(M\times S^1,g_M + g_{S^1},f_M \times \id)\\ 
&& +j\rhoct(L\times N\times S^1,g_L +g_M+g_{S^1},f_L \times f_M \times \id)  \ .
\end{eqnarray*}

Now we argue as above using that $\int_{N \times S^1} \hat A(N \times S^1)\cup f_N^*(\vol_{S^1} \cup c)\neq 0$ if $\int_N \hat A(N) \cup f_N^* c\neq 0$. Thus $\rhoct(L\times N\times S^1)\neq 0$. This implies that the triples  $(M \times S^1, g_M^j+ g_{S^1},f_M\times \id)$ are mutually not strongly $\G \times \Z$-bordant. Clearly, then the triples $(M, g_M^j,f_M)$ define mutually different strong $\G$-bordism classes.
\end{proof}

Low dimensional examples for such manifolds $N$ come from the 3-dimensional solvmanifolds studied in \cite{marc} in the context of twisted index theory. The existence of examples in higher dimensions is guaranteed by Lemma \ref{ahat}.

\medskip
The following two propositions show that, as in the classical case, the twisted $\rho$-invariants are less interesting for torsion-free groups. 

\begin{prop}
Let $M, N$ be odd-dimensional closed spin manifolds with metrics $g_M, g_N$ of positive scalar curvature and with fundamental group $\G$. Let $f_M \colon M \to B\Gamma$, $f_N \colon N \to B\Gamma$ be maps inducing the universal covering. 

Let ${\mathbf c}=(c_1, c_2 \dots c_k) \in H^2(B\G, \Q)^k$ and for each $s \in \Q^k$ in neighbourhood of $0$ let $\tau_s$ be a delocalized trace on $C^*(\G,\s^s)$ which is invariant under parametrized projective automorphisms. We assume that these traces are uniformly bounded.

Furthermore we assume that one of the following two conditions holds:

\begin{enumerate}
\item
The induced map $\tau_s \colon K_0(C^*(\G,\s^s)) \to \C$ vanishes for $s$ near $0$.

\item
The maximal Baum--Connes assembly map $K_0(B\G)\ten \Q \to K_0(C^*\G)\ten \Q$ is surjective. 
\end{enumerate}

If $(M,f_M)$ and $(N,f_N)$ are bordant in $\Omega_*^{spin}(B\G)$, then $$\rhoct(M,g_M,f_M)=\rhoct(N,g_N,f_N) \ .$$
\end{prop}

\begin{proof}
First we assume that $\dim M=\dim N\ge 5$.

We denote the spin bordism between $(M,f_M)$ and $(N,f_N)$ by $(W,f_W)$. By Lemma \ref{bordlemma} we may assume that $f_W\colon W \to B\G$ classifies the universal covering. 

Taking condition (1) for granted, the assertion follows now directly from the twisted Atiyah--Patodi--Singer index theorem.

We go on to prove the assertion under condition (2).

Our proof relies on a modification of the method used in \cite{HaS}. The strategy is to adapt the approach of \cite{ps1} as far as possible.

We assume that the assertion is wrong. Thus for each $n \in \bn$ there is $s_n \in \{s \in \Q~|~|s| < \frac 1n\}$ such that
$$\eta_{\tau_{s_n}}(\Di^M_{\mathcal V^{ \s^{s_n}}}) \neq \eta_{\tau_{s_n}}(\Di^N_{\mathcal V^{ \s^{s_n}}}) \ .$$ 

The following constructions are carried out on $W$.

For $n_0 \in \bn$ large we consider the $C^*$-algebra $$\cA^{\pi}:=\prod_{n>n_0} C^*(\G,\s^{s_n})$$ and the ideal in $\cA^{\pi}$
$$\cA^{\oplus}:=\overline{\bigoplus_{n>n_0} C^*(\G,\s^{s_n})} \ .$$ 
Since they are uniformly bounded, the traces $\tau_{s_n}$ assemble to a continuous trace 
$$
\tau^{\pi}\colon \cA^{\pi} \to \prod_{n>n_0} \C \ .
$$

The $\G$-actions $A^{s_n} \ten T^{s_n}$ on the trivial bundle $L \ten C^*(\G,\s^{s_n})$, which is defined as in \S \ref{bcgs-bundle} on the universal covering $\widetilde W$, assemble to a $\G$-action $A^{\pi} \ten T^{\pi}$ on $L\ten \cA^{\pi}$. Furthermore the 1-forms $\eta_{s_n}$ assemble to an $\cA^{\pi}$-valued 1-form $\eta^{\pi}$. Thus on the $\G$-invariant trivial bundle $L \ten \cA^{\pi}$ we have a $\G$-invariant connection $d+\eta^{\pi}$. Therefore, on the quotient bundle $\N^{\pi}$ we get an induced connection $\nabla^{\pi}$. We denote by $\Di^{\pi}$ the spin Dirac operator twisted by $\N^{\pi}$. 

If $n_0$ is large enough, the induced Dirac operator $\Di_\ra^{\pi}$ on $\ra W$ is invertible. As explained in \S \ref{twistdir} the Atiyah--Patodi--Singer index $\ind(\Di^{\pi}) \in K_0(\cA^{\pi})$ is independent of the choices since parametrized automorphisms on the family of $C^*$-algebras $C^*(\G,\s^s)$ induce a trivial action on $K_0(\cA^{\pi})$.

Using the above trace $\tau^{\pi}$ we can define an $\eta$-invariant $\eta^{\pi}(\Di_\ra^{\pi})\in \prod_{n>n_0}\C$, which equals $(\eta_{\tau_{s_n}}(\Di^N_{\mathcal V^{\s^{s_n}}})-\eta_{\tau_{s_n}}(\Di^M_{\mathcal V^{\s^{s_n}}}))_{n>n_0}$. Since the traces are delocalized, the $C^*$-algebraic Atiyah--Patodi--Singer index theorem implies that
$$\tau^{\pi}(\ind(\Di^{\pi})) =-\eta^{\pi}(\Di_\ra^{\pi})) \ .$$

In order to get a contradiction, it remains to show that $\tau^{\pi}(\ind(\Di^{\pi}))=0$.

We have a quotient map $q\colon \cA^{\pi} \to \cA^{\pi}/\cA^{\oplus}$ and a canonical $C^*$-homomorphism $h\colon C^*\G \to \cA^{\pi}/\cA^{\oplus}$. The bundle $\N^{\pi} \ten_q \cA^{\pi}/\cA^{\oplus}$ is flat and isomorphic to $\N \ten_h  \cA^{\pi}/\cA^{\oplus}$, where $\N$ is the Mishenko--Fomenko bundle on $W$ with the standard flat connection. Furthermore the isomorphism intertwines the connections. In particular, if $\Di^W_{\N}$ denotes, as usual, the spin Dirac operator on $W$ twisted by $\N$, then 
$$h_*\ind (\Di^W_{\N})=q_* \ind(\Di^{\pi}) \in K_0(\cA^{\pi}/\cA^{\oplus}) \ .$$

As explained in \cite[p. 124]{HaS}, the trace $\tau^{\pi}$ decends to a homomorphism
$$
\overline{\tau}^{\pi}\colon \Im (q_*:K_0(\cA^{\pi}) \to K_0(\cA^{\pi}/\cA^{\oplus})) \to \left(\prod_{n>n_0} \C\right)/\left(\bigoplus_{n>n_0} \C\right) \ .$$
Note that here the direct sum $\bigoplus \C$ is taken in the algebraic sense.

By the Baum--Douglas description of $K$-homology (see \cite[p. 121]{HaS} for details as needed here) our assumption implies that there is a closed spin manifold $L$ with a reference map $f_L \colon L \to B\G$ inducing the universal covering and a Dirac operator $D$ on $L$ such that the index $\ind(D_{\N_L}) \in K_0(C^*\G)\ten \Q$ of $D$ twisted by the Mishenko--Fomenko bundle $\N_L$ equals $x\ind(\Di^W_{\N})$ for some $x \in \Q$. By applying the above construction to $(L,f_L)$ we get a Dirac operator $\Di_L^{\pi}$ whose index $\ind(\Di_L^{\pi}) \in K_0(\cA^{\pi})$ is independent of the choices in the sense of \S \ref{twistdir}. In $K_0(\cA^{\pi}/\cA^{\oplus})\ten \Q$ it holds that
$$q_*\ind(\Di_L^{\pi})=h_*\ind(D_{\N_L}) =xh_*\ind(\Di^W_{\N})=x q_*(\ind(\Di^{\pi})) \ .$$
Since the traces $\tau_s$ are delocalized, the twisted Atiyah--Singer index theorem implies that $\overline{\tau}^{\pi}q_*\ind(\Di_L^{\pi})=0$.
 
It follows that $\overline{\tau}^{\pi}q_*(\ind(\Di^{\pi}))$ vanishes as well. Thus for $n_0$ large enough $\tau^{\pi}\ind(\Di^{\pi})$ vanishes and we obtain the desired contradiction.

If $\dim M=\dim N<5$, we choose an 8-dimensional manifold $C$ with $\int_C\hat A(C) \neq 0$ and $\pi_1(C)=1$, and consider the twisted $\rho$-invariants of $M\times C$ and $N \times C$, where the metric on $C$ is chosen such that the products have positive scalar curvature. Then the assertion follows from an analogue of the product formula, Prop. \ref{prod}, since the index of the spin Dirac operator on $C$ is nontrivial.
\end{proof}

The following example illustrates the first condition. It shows that the assumption in Prop. \ref{infbordism} that $\G_1$ has torsion is indeed necessary for the proof to work. 

\medskip

Assume that $\G=\G_1\times \G_2$ with $\G_1$ a torsion-free and $\G_2$ a finite group. Let ${\mathbf c} \in H^2(B\G_2,\Q)^k$ and let $\s^s$ be an associated family of multipliers on $\G_2$. We can take any delocalized trace $\tau_1$ on $C^*\G_1$ and consider $\tau_s:=\tau_1 \ten \tr^{\s^s}_{(2)}$ on $C^*(\G,\pi_2^*\s^s)$ as in \S \ref{tr-prod}. 

The twisted $C^*$-algebra $C^*(\G_2,\s^s)$ is finite-dimensional. In particular, by the K\"unneth formula, $K_0(C^*(\G,\pi_2^*\s^s))\ten \Q=K_0(C^*\G_1)\ten K_0(C^*(\G_2,\s^s))\ten \Q \oplus K_1(C^*\G_1)\ten K_1(C^*(\G_2,\s^s))\ten \Q $. We also assume that the maximal Baum--Connes assembly map $K_0(B\G_1) \ten \Q \to K_0(C^*\G_1)\ten \Q$ is surjective. It follows that $\tau_1$ vanishes on $K_0(C^*\G_1)$ (see the argument in the proof of \cite[Lemma 4.5]{ps1}). Thus $\tau_s$ vanishes on $K_0(C^*(\G,\pi_2^*\s^s))\ten \Q$.

\begin{prop}
Let $M$ be an odd-dimensional closed connected spin manifold with metrics $g_M$ of positive scalar curvature and with fundamental group $\G$. Let $f_M \colon M \to B\Gamma$ be a map inducing the universal covering. 

Let ${\mathbf c}=(c_1, c_2 \dots c_k) \in H^2(B\G, \Q)^k$ and for all $s \in \Q^k$ in a neighbourhood of $0$ let $\tau_s$ be a delocalized trace on $C^*(\G,\s^s)$ which is invariant under parametrised projective automorphisms.

We assume that one of the two conditions of the previous proposition holds and that the maximal Baum--Connes assembly map $K_0(B\G)\ten \Q \to K_0(C^*\G)\ten \Q$ is injective. Then $\rhoct(M,g_M,f_M)=0$.
\end{prop}

\begin{proof}
We again adapt the strategy of \cite{ps1}. By \cite[Prop. 5.3]{ps1} the injectivity of the Baum--Connes assembly map implies that there is $d \in \bn$ such that the union of $d$ copies of $(M,f_M)$ is bordant in $\O_*^{spin}(B\G)$ to a finite union $\cup_j (A_j \times B_j, f_{A_j} \circ p_1)$ with $\dim B_j=0 \mod 4$, $\pi_1(B_j)=1$, $\int_{B_j} \hat A(B_j) =0$. Here and in the following $p_1$ denotes the projection onto the first factor. 

We may choose an 8-dimensional manifold $N$ with $\pi_1(N)=1$ and $\int_N \hat A(N) \neq 0$, see Lemma \ref{ahat}. By applying spin surgery we alter $f_{A_j} \circ p_1 \colon A_j \times N \to B\G$ such that the resulting map $f_{C_j} \colon C_j \to B\G$ classifies the universal covering. Here we use that $\dim (A_j\times N_j) \ge 5$.
Then $(M \times N,f_M\circ p_1)$ is bordant in $\O_*^{spin}(B\G)$ to $\cup_j (C_j \times B_j, f_{C_j} \circ p_1)$.
 
Now we follow \cite[Theorem 6.5]{ps1}. On $B_j$ there is a metric of positive scalar curvature by \cite{St1}. We may assume that the product metric on $C_j\times B_j$ has positive scalar curvature as well. Let $\Di_{B_j}$ denote the spin Dirac operator on $B_j$ and $\grad_{B_j}$ the grading operator. By the arguments from the proof of Prop. \ref{prod} the twisted $\rho$-invariant $\rhoct(C_j\times B_j,g_{C_j\times B_j},f_{C_j} \circ p_1)$ vanishes since $\tralg \grad_{B_j} e^{-t \Di_{B_j}^2}$ vanishes by the McKean--Singer formula. 

By the previous proposition it follows that $\rhoct(M \times N, g_{M\times N}, f_M \circ p_1)=0$. Since $\int_N\hat A(N)\neq 0$, a product formula proven as Prop. \ref{prod} implies that $\rhoct(M,g_M,f_M)=0$.
\end{proof}

\section{$2$-cocycle twists and the signature operator}
\label{sign}

Now we turn our attention to the signature operator.

Fix as usual $(c_1, c_2, \dots, c_k)\in H^2(B\G,\Q)^k$, let $M$ be a closed manifold with fundamental group $\G$ and a reference map $f_M\colon M \to B\G$ inducing the universal covering.  

Let $\grad$ be the chirality grading induced by the Hodge star operator on $\Lambda^*T^*M$, and let $c$ be the Clifford action. On the Clifford module $\Lambda^*T^*M \ten \N^{\s^s}$ we have a connection $\nabla^{sign}$ obtained as the tensor product of the Levi-Civita connection with the connection on $\N^{\s^s}$. For $\dim M$ even, the signature operator  is defined as $D^{\s^s}=c \circ \nabla^{sign}$, while for $\dim M$ odd $D^{\s^s}=\grad (c \circ \nabla^{sign})$.

Let $M$ be even-dimensional. For all $s$ near $0 \in \Q^k$ let $\t_s$ be a trace on $C^*(\G,\s^s)$ invariant under projective automorphisms. We write $$\sign^{\s_s}_{\t_s} (M)=\t_s \ind(D^{\s^s}) \ .$$

If $\t_s=\tr_{(2)}^{\s^s}$ is the trace given in \S \ref{lrt}, then we write $\sign_{(2)}^{\s^s}(M)$.

In the following we assume that $\t_s$ is positive for each $s$.

\begin{definition}\label{twBetti}  We define the \emph{twisted Betti numbers} to be 
$$
b_{\tau_s}^{\s^s, even}(M):= \tau_s(P^{even})\;\;;\;\; b_{\tau_s}^{\s^s, odd}(M):=\tau_s(P^{odd})
$$
where $P^{even}$ and $P^{odd}$ are the projections onto the kernel of $D^{ \s^s}$ restricted to differential forms of even and odd degree, respectively. As usual, these expressions are understood as a germ at $s=0$.
\end{definition}

The difference $b_{\tau_s}^{\s^s, even}(M)-b_{\tau_s}^{\s^s, odd}(M)$ equals the twisted Euler characteristic, which is defined as the index of the operator $D^{\s^s}$ with respect to the even/odd grading.
\medskip

If $\tau_s=\tr_{(2)}^{\s_s}$, then we write $b_{(2)}^{\s^s, even}(M)$.

From the results in \S \ref{para-constr} it is clear that these Betti numbers do not depend on the choices of $\eta_i, \psi_\g^i$ for fixed $\o_i$, but may depend on the choice of $\o_i$.

Furthermore it is natural to pose the following question.

\begin{que} 
\begin{enumerate}
\item Does $b_{\tau_s}^{\s^s,*}(M)$ depend on $\o_i$?
\item Does $b_{\tau_s}^{\s^s,*}(M)$ depend on the metric on $M$?
\item Is $b_{\tau^s}^{\s^s,*}(M)$ a homotopy invariant of $M$?
\end{enumerate}
\end{que}

\begin{rem} In \cite[Question 57.2]{Mis} Sauer and Schick constructed Betti numbers  $b^k_E(M)$ for the signature operator twisted by a vector bundle $E$ with given connection $\nabla^E$: they are defined as the dimension of the kernel of the signature operator twisted by $E$ in degree $k$. Since the twisted Laplacian does not preserve the degree of differential forms, the alternating sum of these Betti numbers may not be equal to the twisted Euler characteristic, \emph{i.\,e.} the index of the twisted de Rham operator with respect to the even/odd grading. 

One may also define Betti numbers using the operator $(\nabla^{sign})^*\nabla^{sign}+ \nabla^{sign}(\nabla^{sign})^*$, which preserves the degree of differential forms and agrees with the Laplacian in the flat case. However, in the general case the topological significance of this operator remains unclear.
\end{rem}

Now let $\dim M$ be odd. Def. \ref{deftwistedeta} gives us von Neumann twisted $\eta$- and $\rho$-invariants associated to the signature operator. The twisted $\eta$-invariant depends on the metric and on the choice of the $\o_i$.

\begin{que}
\label{questions}
\begin{enumerate}
\item Do the twisted von Neumann $\rho$-invariants for the signature operator depend on the $\o_i$ and the metric?
\item The classical von Neumann $\rho$-invariants of the signature operator are homotopy invariant when $\G$ is torsion free and fulfills the Baum--Connes conjecture for the maximal group $C^*$-algebra \cite{kes}; are there analogous results for the twisted von Neumann $\rho$-invariants?
\end{enumerate}
\end{que}

In view of the results in \cite{marc} it would be interesting to calculate the twisted $\eta$-invariant for 3-dimensional solvmanifolds and $\t_s=\tr_{(2)}^{\s^s}$. For the ordinary $\eta$-invariant this has been done in \cite{ads}.

\subsection{Spectral flow and mapping torus}

In \cite{aw} it was proven that the $L^2$-spectral flow of a path of signature operators with varying metrics vanishes. This was used to reprove the metric independence of the $L^2$-$\rho$-invariants and of $L^2$-signatures for manifolds with boundary. Here we show by an example that this method does not generalize easily to our twisted signatures.
For that aim we first prove a formula relating the twisted signatures of a mapping torus to von Neumann spectral flow.
\medskip

Let $(F, h,r, H)$ be a quadruple consisting of an oriented manifold $F$, an orientation preserving diffeomorphism $h\colon F\to F$, a reference map $r\colon F\to B\G$ inducing the universal covering, and a homotopy 
$H\colon [0,1] \times F\to B\G$ such that $F(0,\cdot)=r$ and $F(1,\cdot)=r\circ h$. Note that this implies that $h$ induces the identity on the fundamental group. Choose a path of metrics $(g_t)_{t\in [0,1]}$ on $F$ which is constant near $0,1$ and with $g_1=h^*g_0$. We  endow $[0,1]\times F$ with the metric $dt^2+g_t$. 

We denote by $p_F:\widetilde F \to F$ the projection.

Let $(c_1, c_2, \dots, c_k)\in H^2(B\G,\Q)^k$. 

For each $i=1,\dots,k$ the forms $\psi_\g^i, \eta_i, \o_i$ on $\widetilde F$ and $F$, respectively, are constructed as usual. 

For $i=1,\dots, k$ let $\phi_i \in \O^1(F,\R)$ be such that $d\phi_i=h^*\o_i-\o_i$. 

Our aim the following is to study a path connecting the twisted signature operator associated to the choices $\psi_\g^i, \eta_i, \o_i$ and the metric $g_0$ with the operator associated to the choices $\psi_\g^i, \eta_i+\phi, h^*\o_i$ and the metric $h^*g_0$.

\medskip
Let $T_h$ be the mapping torus. The metric $dt^2 + g_t$ is well-defined on $T_h$.
Let $r_G\colon T_h \to BG$ a map inducing the universal covering. Furthermore let $r_{T_h}\colon T_h \to B\G$ be the reference map induced by $H$. Since the covering space induced by $r_{T_h}$ is connected, there is a map $\pi\colon G\to \G$ and (up to homotopy) $r_{T_h}$ factorizes as $T_h \stackrel{r_G}{\to} BG \stackrel{Bp}{\to} B\G$.

%It holds that $[h^*\o_i]=r^*c_i=(r \circ h)^*c_i=[\o_i]$. Thus there is
%$\phi_i \in \O^1(F,\R)$ such that $d\phi=h^*\o_i - \o_i$. 
%With $\chi$ a cut-off function as in \S \ref{depchoic}, on $\widetilde F\times \{t\}$ with $t\in [0,\frac 13]$ we set
%$\eta^t_i=\eta_i + \chi(3t)\pi^*\phi_i$ and $\psi_\g^{t,i}=\psi_\g^{i}$. We have that $h^*\eta_i-(\eta_i+\pi^*\phi_i)=d\alpha_i$ for $\alpha_i\colon \tilde F \to \R$ with $\alpha(x_0)=0$. For $t \in (\frac 13,23]$ we set $\eta_i^t=\eta_i+\pi^*\phi_i + \chi(3t-1)d \alpha$ and $\psi_\g^{t,i}=\psi_\g^i- \chi(3t-1)(\g^*\alpha-\alpha)$. Finally, for $t \in (\frac 23,1]$ we set $\psi_{\g}^{t,i}=(1-\chi(3t-2))(\psi_{\g}-\g^*\alpha-\alpha)+\chi(3t-2)h^*\psi_{\g}$. For each $t\in [0,1]$ we get an associated twisted signature operator $D^{\s^s}_{F\times\{t\}}$. For $t \in [0,\frac 13]$ these operators act on the same $C^*(\G,\s)$-Hilbert module. For $t\in [frac 13,1]$ even the $C^*$-algebra is not fixed. However there is an isometry $U_{(s,t)}\colon \N^{\s_t^s}\ten_{b_{z_{s,t}}} C^*(\G,\s_0^s) \to \N^{\s_0^s}$ defined as in Lemma \ref{choices-conn-s} with $f_{s,t}=\exp(i\chi(3t-1)(s_1\alpha_1 + s_2\alpha_2 + \dots +s_k\slpha_k)$ and $z_{s,t}: \G\to \C$. 

Let $\chi\colon [0,1] \to [0,1]$ be a smooth cut-off function with $\chi(t)=0$ for $t\in [0, \frac{1}{4}]$, and  $\chi(t)=1$ for $t\in [\frac{3}{4}, 1]$.

On $\{t\} \times \widetilde F$ with $t\in [0,\frac 12]$ we set 
$$\eta^t_i=\eta_i + \chi(2t)p_F^*\phi_i \ ,\quad \psi_\g^{t,i}=\psi_\g^{i} \ .$$ Then $\o^t_i=\o_i+\chi(2t)d\phi_i$. There is $\alpha_i\colon \widetilde F \to \R$ such that $h^*\eta_i-(\eta_i+p_F^*\phi_i)=d\alpha_i$. 

For $t \in (\frac 12,1]$ we set 
$$\eta_i^t=\eta_i+p_F^*\phi_i + \chi(2t-1)d \alpha_i\ , \quad \psi_\g^{t,i}=\psi_\g^i+ \chi(2t-1)(\g^*\alpha_i-\alpha_i) \ .$$ For each $t\in [0,1]$ we get an associated twisted signature operator $D^{\s^s}_{\{t\} \times F}$. Note that $\s^s$ does not depend on $t$. For $t \in [0,\frac 12]$ these operators act on the same $C^*(\G,\s^s)$-Hilbert module. For $t> \frac 12$ there is a unitary $U_{(s,t)}\colon \N_t^{\s^s} \to \N_{\frac 12}^{\s^s}$ defined as in Lemma \ref{choices-conn-s} with $f_{s,t}=\exp(i\chi(2t-1)(s_1\alpha_1 + s_2\alpha_2 + \dots +s_k\alpha_k))$ and $z_s=1$. By conjugation with this unitary we get a path of operators acting on the sections of a common bundle $\N^{\s^s}$ on $F$. 

This defines the right hand side of the equation in the following proposition.

To $(\pi^*c_1, \pi^*c_2, \dots, \pi^*c_k)\in H^2(BG,\Q)^k$ we associate the left hand side of the proposition as in the beginning of \S \ref{sign}. Here $\s_G^s$ is a multiplier on $G$.

\begin{prop}
\label{mappling_torus}
\begin{equation*}
\sign^{\s_G^s}_{(2)} (T_h)= \spfl_{\tr_{(2)}^{\s^s}} (D^{\s^s}_{\{t\}\times F})_{t\in [0,1]} \ .
\end{equation*}
\end{prop}

\begin{proof}
Let $c_s=s_1c_1 +s_2c_2 +\dots s_kc_k \in H^2(B\G,\Q)$.
  
By the twisted index theorem of Mathai \cite[p.\,14]{Ma2} 
$$
\sign^{\s^s_G}_{(2)}(T_h)=(2 \pi i)^{\frac{\dim T_h}{2}}(L(T_h)\cup e^{2 \pi i (r_{T_h}^*c_s)}, [T_h]) \ .
$$  
(Note that we use a different normalisation of the Chern character as Mathai.)

An explicit representative of the cohomology class $e^{2 \pi i (r_{T_h}^*c_s)}$ can be constructed as follows. For notational simplicity we assume that $c_i \in H^2(B\G,\Z)$. Thus there is a line bundle $\cL_i \to B\G$ whose first Chern class equals $c_i$. The line bundle $r_{T_h}^*\cL_i$ is related to the bundle $H^*\cL_i$ on $Z:=[0,1]\times F$ by identifying the fibers at $\{0\} \times F$ and $\{1\}\times F$ using $h^*$. For each $i$ we get the connection $\nabla^{Z,i}:=dt + \nabla^{\eta^t_i}$ on $H^*\cL_i$ with $\eta^t_i$ as above. Since $\eta^1_i=h^*\eta^0_i$, it holds that $\nabla^{\eta^1_i}=h^*\nabla^{\eta^0_i}$. Thus $\nabla^{Z,i}$ defines a connection on the line bundle $r_{T_h}^*\cL$. We denote its curvature by $i\omega_i^Z=i d_Z \eta^t_i$. As usual, we set $\omega_s^Z=s_1\o_1^Z+ s_2 \o_2^Z + \dots + s_k \o_k^Z$. Since $[-\frac{\o_s^Z}{2\pi}]=r_{T_h}^*c_s \in H^2(T_h,\Q)$, it holds that $\sign^{\s^s_G}_{(2)}(T_h)=\int_{T_h} L(T_h) e^{-i\omega_s^Z}$. 

On the other hand,
\begin{multline*}
\spfl_{\tr_{(2)}^{\s^s}} (D^{\s^s}_{\{t\}\times F})_{t\in [0,1]}\\
\begin{aligned}
 &= (2\pi i)^{\frac{\dim T_h}{2}}\int_0^1 \int_F L([0,1] \times F)e^{-i\o_s} -\eta(D^{\s^s}_{\{1\}\times F})+\eta(D^{\s^s}_{\{0\}\times F})\\
 &-\frac 12 \left(\beta_{(2)}^{\s^s,even}(\{1\}\times F) + \beta_{(2)}^{\s^s,odd}(\{1\}\times F)
- \beta_{(2)}^{\s^s,even}(\{0\}\times F) - \beta_{(2)}^{\s^s,odd}(\{0\}\times F) \right) \ .
\end{aligned}
\end{multline*}

Since $D^{\s^s}_{\{0\}\times F}$ and $D^{\s^s}_{\{1\}\times F}$ are unitarily equivalent via $h$, the $\eta$-invariants and Betti numbers cancel out. The assertion follows now by identifying $L(T_h)$ and $L([0,1]\times F)$.
\end{proof}

Based on \cite[Example 1.10]{llk} we construct a mapping torus with nonvanishing twisted signature. 

Let $r\colon \C P^2\times S^1\to S^1=B\Z$ be the projection onto the second factor and consider the element $\theta=[\C P^2\times S^1, r]$ in the oriented bordism group $\O_5(B\Z)$. 

Let $b_1$ be the generator of $H^1(S^1)$. The higher signature $\sign (S^1, b_1)=\int_{S^1} L(S^1)b_1$ does not vanish 
and therefore also $\sign(\C P^2\times S^1, b_1)\neq 0$.
By the Appendix in \cite{llk} there exists a quadruple $(F_1, h_1,r_1, H_1)$ as above with $\dim F_1=4$ such that $[T_{h_1}, r_{T_{h_1}}]=\theta$. The cobordism invariance of higher signatures implies that $\sign (T_{h_1}, b_1)\neq 0$. 

Construct an odd dimensional quadruple $(F, h, r, H)$ by multiplying by $S^1$, \emph{i.\,e.} $F=F_1\times S^1$, $h=h_1\times \id$, $r=r_1 \times \id \colon F\to S^1\times S^1$, $H=H_1\times \id$.

Let $c=b_1\cup b_2 \in H^2(S^1 \times S^1)$ where $b_2$ is the generator of $H^1(S^1)$ for the second factor.  

By the multiplicativity of higher signatures we have that $\sign (T_h, c)\neq 0$. It follows that $\sign^{\s_s}_{\t_s}(T_h)$ (considered, as usual, as a germ at $s=0$) does not vanish. 

Thus, by the previous proposition, we get a path of twisted signature operators with nonvanishing spectral flow. Note that in the example we vary the metric $g$ and the forms $\o_i$. It is not clear what happens if the forms are kept constant and only the metric is varied or vice versa.

\begin{rem}
In particular one could not prove the metric independence of the twisted $\rho$-invariants for the signature operator by means of the Cheeger--Gromov argument (\cite[pp. 24-27]{CG2}, see also \cite[\S 4]{aw} where this is recast using spectral flow).

Furthermore the example shows that a naive extension of the definition of the twisted signatures to manifolds with boundaries (using the Atiyah--Patodi--Singer index) does not have the invariance properties one would like. Namely, $\spfl_{\tr_{(2)}^{\s^s}} (D^{\s^s}_{\{t\}\times F})_{t\in [0,1]}$ would then be equal to the twisted signature on the cylinder, which is zero if all choices are taken of product type but nonzero in the above example. 
\end{rem}

\section{The higher twisted Atiyah--Patodi--Singer index theorem}
\label{high_twist_APS}

Higher $\rho$-invariants, introduced by Lott in \cite{lohigheta}, were used by Leichtnam and Piazza to distinguish $\Gamma$-bordism classes of metrics of positive scalar curvature in \cite{lpetapos}. 

In the following we apply the results of \cite{waAPS} in order to study higher twisted $\eta$- and $\rho$-invariants. Note that for their definition, the methods of \cite{lodiffeo} are sufficient.  

We restrict ourselves to the odd-dimensional case, the even-dimensional case can be treated analogously.

In order to define a higher $\eta$-invariant of the Dirac operator $\Di_{\mathcal V^{\s^s}}$ we need a ``smooth structure'' on $C^*(\G, \s^s)$.

Therefore we assume that $(\cA_i,\iota_{i+1,i}\colon \cA_{i+1} \to \cA_i)_{i \in \bn_0}$ is a projective system of
involutive Banach algebras with unit satisfying the following conditions:
\begin{itemize} 
\item $\cA_0=C^*(\G, \s^s)$.
\item For any $i \in \bn_0$ the map $\iota_{i+1,i}\colon \cA_{i+1} \to \cA_i$ is injective.
\item For any $i \in \bn_0$ the map $\iota_i\colon \Ai:= \varprojlim\limits_{j}\cA_j \to \cA_i$ has dense
range.
\item For any $i \in \bn_0$ the algebra $\cA_i$ is stable with respect to the holomorphic functional calculus in
$\cA$.  
\item The algebra $\C(\Gamma, \s^s)$ is densely contained in $\Ai$ (which we can identify with a subalgebra of $\cA_0$). 
\end{itemize}

We call $\Ai$ a smooth subalgebra of $C^*(\G,\s^s)$.

Examples of such projective systems, which are relevant for our applications, are constructed in the Appendix \S \ref{Appendix}.

Recall that the universal differential algebra associated to $\Ai$ is defined as $\Oi\Ai= \prod_{k=0}^{\infty} \Ok \Ai$ with $\Ok \Ai := \Ai (\hat\ten_{\pi} \Ai/\C)^k, ~k \in \bn_0$. The algebra is $\Z$-graded in a natural way; in the following all relevant constructions (tensor products, commutators etc.) are taken in a graded sense. 

The differential $\di$ is given by 
$\di(b_0 \ten b_1 \dots \ten b_k)=1 \ten b_0 \ten b_1 \dots \ten b_k$ and linear extension. The product fulfills the graded Leibniz rule. Furthermore the involution extends to $\Oi\Ai$ such that $(\alpha\beta)^*=\beta^*\alpha^*$. 

In general, we will not consider the universal algebra itself but quotients thereof. If $\cI \subset \Oi\Ai$ is a closed homogeneous involutive ideal which is closed under $\di$, then one may define the differential algebra $\Oi^{\cI}\Ai:=\Oi\Ai/\cI$. 

The de Rham homology $H^{dR}_{\cI}(\Ai)$ is defined as the homology of the complex $(\Oi^{\cI}\Ai/\overline{[\Oi^{\cI}\Ai,\Oi^{\cI}\Ai]},\di)$. Here (as always in this context) we divide out the closure of the image of $\di$ in order to get a Hausdorff space. 

For simplicity, we will omit the ideal $\cI$ from the notation in the following.

For a projection $P\in \Cinf(M,M_m(\Ai))$ there is a Chern character form $$\ch(P)=\sum_{j=0}^{\infty} \frac{(-1)^j}{j!}\tralg (P(d_M+\di)P(d_M+\di)P)^j \in \Omega^*(M) \hat\ten_{\pi} \Oi\Ai/\overline{[\Oi\Ai,\Oi\Ai]} \ ,$$ 
which induces a homomorphism $K_0(C(M,C^*(\G,\s^s))) \to H^*(M)\ten H^{dR}_*(\Ai)$. Note that $P(d_M+\di)P$ plays the role of a connection here with a component $Pd_MP$ in the direction of $M$ and a component $P\di P$ in the direction of the algebra. Up to exact forms, the Chern character is independent of it. If we want to emphasize the connection, we write $\ch(Pd_MP+P\di P)$. As in the classical case, we may (and will) use more general connections, see \cite[\S 1.2]{waAPS}. 

In order to define this type of connection for $\mathcal V^{\s^s}$ we will associate an explicit projection $P \in \Cinf(M,M_m(\C(\Gamma, \sigma^s)))$ to the bundle $\mathcal V^{\s^s}$ such that $P(M\times \C(\Gamma, \sigma^s)^m)\cong\mathcal V^{\s^s}$.

Let $(U_{i})_{i\in I}$ be a cover of $M$ consisting of trivializing open sets for
the universal covering $p\colon\Mt\rightarrow M$, and 
let $\phi_{i}\colon p^{-1}(U_{i})\stackrel{\sim}{\rightarrow} U_{i}\times
\Gamma$. Denote $m:=|I|$.

For $x \in U_i$ we write $\tilde{x}_{i}:=\phi_{i}^{-1}(x,e)$. We define the locally constant function $g_{ij}\colon U_i\cap U_j \to \G$ by $\tilde x_i=g_{ij} \tilde x_j$.
Choose a smooth partition of unity $(\chi_i^2)_{i \in I}$ subordinate to $(U_{i})_{i\in I}$.

Define $P \in \Cinf(M,M_m(\C(\Gamma, \sigma^s))$ by setting $P_{ij}=e^{-i\psi^s_{g_{ij}}(\tilde x_j)}\chi_i \chi_j \delta_{g_{ij}}$.
We have  
\begin{align*}
\sum_{j\in  I} P_{ij}P_{jk}&=\sum_{j\in  I} \chi_i e^{-i\psi^s_{g_{ij}}(\tilde x_j)}\chi_j^2 e^{-i\psi^s_{g_{jk}}(\tilde x_k)}\chi_k \delta_{g_{ij}}*\delta_{g_{jk}}\\
&=\sum_{j\in  I} \chi_i e^{-i\psi^s_{g_{ij}}(g_{jk}\tilde x_k)}\chi_j^2 e^{-i\psi^s_{g_{jk}}(\tilde x_k)}\chi_k \sigma^s(g_{ij},g_{jk}) \delta_{g_{ik}}\\
&=\sum_{j\in  I} \chi_i e^{-i\psi^s_{g_{ik}}(\tilde x_k)}\chi_j^2 \ov \sigma^s(g_{ij},g_{jk})\sigma^s(g_{ij},g_{jk}) \chi_k \delta_{g_{ik}}\\
&=\chi_i e^{-i\psi^s_{g_{ik}}(\tilde x_k)}\chi_k \delta_{g_{ik}}=P_{ik} \ .
\end{align*}

Thus $P$ is a projection. Using that $\sigma^s(\g,\g^{-1})= e^{i(\psi^s_{\g^{-1}}+(\g^{-1})^*\psi^s_\g)}$, one checks that $P$ is selfadjoint: 
\begin{multline*}
(P^*)_{ij}=e^{i\psi^s_{g_{ij}}(\tilde x_i)}\chi_j\chi_i \delta_{g_{ji}}^*=e^{i\psi^s_{g_{ij}}(\tilde x_i)}\chi_i\chi_j \bar{\s}^s(g_{ji}, g_{ji}^{-1})\de_{g_{ij}}\\
=e^{i\psi^s_{g_{ij}}(\tilde x_i)}\chi_i\chi_je^{-i\psi^s_{g_{ij}}(\tilde x_j)-i\psi^s_{g_{ji}}(g_{ij}\tilde x_j)}\de_{g_{ij}}= e^{-i\psi^s_{g_{ij}}(\tilde x_j)}\chi_i \chi_j \delta_{g_{ij}}\ .
\end{multline*}

\begin{lem}
The map
$$U_P\colon \Cinf(\Mt,L \ten C^*(\Gamma,\sigma^s))^{\Gamma} \to \Cinf(M, C^*(\Gamma,\sigma^s)^m)$$
$$s\mapsto (\chi_1 s(\tilde x_1), \chi_2 s(\tilde x_2), \dots, \chi_m s(\tilde x_m)) \ .$$
defines an isometric embedding $\N^{\s^s} \to M \times C^*(\Gamma,\sigma^s)^m$.

The image of the embedding equals $P(\Cinf(M,C^*(\Gamma,\sigma^s))^m)$. Moreover, $U_P$ intertwines $\nabla^{\s^s}$ and 
$$\nabla^P:=Pd_MP+i\sum_{j=1}^m \chi_j^2  \eta(\tilde x_j) \ .$$
\end{lem}

\begin{proof}
For $s=0$ the proof can be found in \cite[Prop. 3.2]{wu}. We leave some of the details to the reader. We only show that we get the right image and the last assertion.

From \eqref{AtensT} we can write $s(\tilde x_i)=(A \ten T)_{g_{ij}}s(\tilde x_j)=e^{-i\psi_{g_{ij}}(\tilde x_j)}(\delta_{g_{ij}}*s(\tilde x_j))$, so one checks that the image of the above map is in the range of $P$. For surjectivity we check that the column vectors of $P$ are in the image. Thus consider the vector $$\chi_i(\chi_1 e^{-i\psi^s_{g_{1i}}(\tilde x_i)} \delta_{g_{1i}},\chi_2 e^{-i\psi^s_{g_{2i}}(\tilde x_i)} \delta_{g_{2i}}, \dots ,\chi_m e^{-i\psi^s_{g_{mi}}(\tilde x_i)}\delta_{g_{mi}}) \ .$$ 

For $x \in p^{-1}(U_i)$ there is $\g \in \G$ such that $x=\g \tilde x_i$.  We set $s(x)=e^{-i\psi_{\g}(\tilde x_i)}\delta_{\g} (\chi_i \circ p)(x)$. (Here we write $\tilde x_i$ for $(\tilde x_i \circ p)(x)$). Furthermore let $s(x)=0$ for $x \notin p^{-1}(U_i)$. Then $s$ is the preimage of the above vector.

Now we consider the action of the above map on the connections.
Using that $\sum_{j=1}^m \chi_j d_M(\chi_j)=0$, we calculate
\begin{eqnarray*}
(Pd_MU_Ps)_i&=& \sum_{j=1}^m P_{ij}d_M \chi_j s(\tilde x_j)\\
&=&\sum_{j=1}^m e^{-i\psi^s_{g_{ij}}(\tilde x_j)}\chi_i \chi_j \delta_{g_{ij}}*((d_M\chi_j) s(\tilde x_j)+\chi_j d_M s(\tilde x_j))\\
&=&\sum_{j=1}^m \chi_i \chi_j(d_M\chi_j) s(\tilde x_i) + \chi_i \chi_j^2 d_M\bigl(e^{-i\psi^s_{g_{ij}}(\tilde x_j)} \delta_{g_{ij}}* s(\tilde x_j)\bigr)\\
&& \quad + \chi_i \chi_j^2 (i\eta(\tilde x_i) -i\eta(\tilde x_j))e^{-i\psi^s_{g_{ij}}(\tilde x_j)}\delta_{g_{ij}}* s(\tilde x_j)\\
&=&\sum_{j=1}^m \chi_i \chi_j^2 d_M s(\tilde x_i) + (i\eta(\tilde x_i) -i\eta(\tilde x_j))\chi_i \chi_j^2 s(\tilde x_i)\\
&=& \chi_i \bigl(d_M  +i \eta(\tilde x_i)-\sum_{j=1}^m i \eta(\tilde x_j) \chi_j^2 \bigr)s(\tilde x_i) \ .
\end{eqnarray*}

Thus the connection $Pd_MP+i\sum_{j=1}^m \chi_j^2  \eta(\tilde x_j)$ intertwines with the connection $\nabla=d_M  +i \eta$ on the covering.
\end{proof}

Now we assume that $M$ is an even-dimensional spin manifold with boundary $N$ such that $N$ is endowed with a metric of positive scalar curvature. We assume all structures to be of product type near the boundary. In particular the projection $P$ is assumed to be independent on the orthogonal variable. Then we can apply the noncommutative Atiyah--Patodi--Singer index theorem proven in \cite{waAPS} to the operator $\Di_{\mathcal V^{\s^s}}^M$ for $|s|$ small enough such that the induced operator on the boundary $\Di^N_{\mathcal V^{\s^s}}$ is invertible.

We first recall the definition of the $\eta$-form. 

Let $C_1$ be the Clifford algebra of $\R$ with odd generator $\theta$ with $\theta^2=1$. We define the superconnection $$\spc=P\di P + \theta \Di^N_{\N^{\s^s}} \ .$$ 
Then the rescaled superconnection is given by $$\spc_t=P \di P + \sqrt t \theta \Di^N_{\N^{\s^s}} \ .$$ 

We set $\tralg_{\theta}(a+\theta b):= \tralg(a)$ and define analogously $\Tr_{\theta}$.

Then, in $\Oi\Ai/\ov{[\Oi\Ai,\Oi\Ai]+\di \Oi\Ai}$,
$$\eta(\Di^N_{\N^{\s^s}}):= \frac{1}{2\sqrt{\pi}} \int_0^{\infty}\Tr_{\theta} \Di^N_{\N^{\s^s}} e^{-\spc_t^2}~\frac{dt}{\sqrt t}  \ .$$

Thus, for the $K$-theoretic index $\ind(\Di^M_{\mathcal V^{\s^s}}) \in K_*(C^*(\Gamma,\sigma^s))$ it holds in $H^{dR}_*(\Ai)$ 

$$\ch(\ind \Di^M_{\mathcal V^{\s^s}})=(2\pi i)^{\dim M/2} \int_M \hat A(M) \ch(\nabla^P+ P\di P) - \eta(\Di^N_{\mathcal V^{\s^s}}) \ .$$

In order to win numerical information from this formula, one pairs both sides with continuous reduced cyclic cocycles on $\Ai$.

\subsection{Pairing with cyclic cocycles localized at the identity}
First, we will consider the pairing with reduced cyclic cocycles localized at the identity in order to show that the above Atiyah--Patodi--Singer index theorem yields a generalization of the higher Atiyah--Patodi--Singer index theorem of Leichtnam and Piazza \cite{lphighAPS} on the one hand, and of the higher twisted Atiyah--Singer index theorem of Marcolli and Mathai \cite{mm} on the other (to be precise, of a restricted version: for manifolds instead of orbifolds and using $K$-theory of $C^*$-algebras instead of algebraic $K$-theory).
For the definition of $\rho$-invariants it is the pairing with delocalized cocycles which is interesting: this will be studied in \S \ref{deloc-sec}.

\medskip 

Recall that group cocycles on $\Gamma$ define reduced cyclic cocycles on $\C(\Gamma,\sigma^s)$. (See the Appendix \ref{Appendix}, to which we also refer for terminology and notation.)

Let $\fc \in C^n_{al}(\Gamma), ~n \ge 1,$ with $d_{\Gamma}\fc=0$. 

We abbreviate
$$\fc(i_0,i_1,\dots, i_n):=\fc(e,g_{i_1i_2},g_{i_1i_2} g_{i_2i_3}, \dots, g_{i_1i_2} \dots  g_{i_{n-1}i_n}g_{i_ni_0})$$
and define, as in \cite[Prop. 12]{losupcon}\cite[Lemma 3.4]{wu}, the differential form \begin{eqnarray}\label{lottform} \olott_{\fc}=\sum_{i_0, \dots, i_n \in I}\chi_{i_0}^2d_M(\chi_{i_1}^2)d_M(\chi_{i_2}^2) \dots d_M(\chi_{i_{n}}^2)\fc(i_0,i_1,\dots, i_n)\ .\end{eqnarray}

The form is closed \cite[Lemma 3]{losupcon}.
The cocycle $\fc$ defines a class $[\fc]\in H^{n}(B\Gamma)$ via the canonical isomorphism $H^{n}(\Gamma)\cong H^{n}(B\Gamma)$.
By \cite[Prop. 14]{losupcon}\cite[Lemma 3.6]{wu} it holds that $[\olott_{\fc}]=f^*[\fc]$, where $f\colon M \to B\Gamma$ is the classifying map from \S \ref{mathai}.

Furthermore $\fc$ induces a reduced cyclic $n$-cocycle $\tau^{\sigma^s}_{\fc}$ on $\C(\Gamma,\sigma^s)$, which is invariant under projective automorphisms.

\begin{prop}
\label{chloc}
Let $\fc \in C^n_{al}(\Gamma)$ with $d_{\Gamma}\fc=0$. Then on the level of differential forms $$\tau^{\sigma^s}_\fc(\ch(\nabla^P+P\di P))= (- 1)^{\frac{(n+1)n}{2}}~ e^{-i\omega_s} \olott_\fc \ .$$
Hence, if $M$ is closed, in $H^n(M)$
$$[\tau^{\sigma^s}_\fc(\ch(\nabla^P+P\di P))]=(- 1)^{\frac{(n+1)n}{2}}~ [e^{-i\omega_s}] \cup f^*[\fc] \ .$$
\end{prop}

\begin{proof} We define $d_{tot}=\di+d_M+\beta$, with $\beta:=i\sum_{j=1}^m \chi_j^2  \eta(\tilde x_j) \in \Omega^1(M,\C)$. We have to calculate $\tau_\fc(\tr e^{-P(d_{tot}P)(d_{tot}P)})$. We have
\begin{multline*}
P(d_{tot}P)(d_{tot}P)=(\nabla^P)^2 + P(\di P)(\di P) + P(d_MP)(\di P)+\\+P(\di P)(d_M P)+\beta P(\di P) + P(\di P)\beta \ .
\end{multline*}
Since $(\nabla^P)^2 =i\omega_s$ and $\beta$ anticommutes with $P(\di P)$, it follows that
$$e^{-P(d_{tot}P)(d_{tot}P)}=e^{-i\omega_s}\exp(-P(\di P)(\di P) - P(d_MP)(\di P)-P(\di P)(d_MP)) \ .
$$

We apply $\tau^{\sigma^s}_\fc\circ \tralg$ to the second factor of the right hand side, which we expand into a sum. We begin by considering summands of the form $\tau^{\sigma^s}_\fc(\tralg \alpha P(\di P) (\di P))$. Using Leibniz rule for $\di$ we deduce that $\tralg \alpha P(\di P) (\di P)$ can be written as 
\begin{multline*}
\sum_{i_0,i_1, \dots, i_n \in I}\chi_{i_0}e^{-i\psi^s_{g_{i_0i_1}}(\tilde x_{i_1})}\chi_{i_1}f_{i_1 \dots i_{n-1}} \chi_{i_{n-1}}e^{-i\psi^s_{g_{i_{n-1}i_n}}(\tilde x_{i_n})}\chi_{i_n}^2 \cdot \\
\cdot e^{-i\psi^s_{g_{i_ni_0}}(\tilde x_{i_0})}\chi_{i_0} g_{i_0i_1}\di g_{i_1i_2} \dots \di g_{i_{n-1}i_n} \di g_{i_ni_0}
\end{multline*} 
with $f_{i_1 \dots i_{n-1}} \in \Omega^*(M)$. 
It holds that
\begin{align*}
\lefteqn{e^{-i\psi^s_{g_{i_{n-1}i_n}}(\tilde x_{i_n})}e^{-i\psi^s_{g_{i_ni_0}}(\tilde x_{i_0})}e^{-i\psi^s_{g_{i_0i_1}}(\tilde x_{i_1})}}\\
&=e^{-i\psi^s_{g_{i_{n-1}i_n}}(\tilde x_{i_n})}e^{-i\psi^s_{g_{i_ni_1}}(\tilde x_{i_1})}\ov \sigma^s(g_{i_ni_0},g_{i_0i_1}) \\
&=e^{-i\psi^s_{g_{i_{n-1}i_1}}(\tilde x_{i_1})}\ov \sigma^s(g_{i_{n-1}i_n},g_{i_ni_1}) \ov \sigma^s(g_{i_ni_0},g_{i_0i_1}) \ .
\end{align*}
Furthermore
$$\tau^{\s^s}_\fc(g_{i_0i_1}\di g_{i_1i_2}, \dots, \di g_{i_{n-1}i_n} \di g_{i_ni_0})$$
$$=
\tr_{\langle e \rangle}(\delta_{g_{i_0i_1}}*\delta_{g_{i_1i_2}}* \dots \delta_{g_{i_ni_0}})\fc(i_0,i_1,\dots, i_n)$$
$$=\sigma^s(g_{i_{n-1}i_n},g_{i_ni_1})\sigma^s(g_{i_ni_0},g_{i_0i_1})\tr_{\langle e \rangle}(\delta_{g_{i_1i_2}}* \dots *\delta_{g_{i_{n-1}i_1}}) \fc(i_0,i_1,\dots, i_n) \ .$$

Thus $\tau^{\sigma^s}_\fc(\tralg \alpha P(\di P) (\di P))$ equals the sum over
$$e^{-i\psi^s_{g_{i_{n-1}i_1}}(\tilde x_{i_1})}\chi_{i_1} f_{i_1 \dots i_{n-1}} \chi_{i_{n-1}}\chi_{i_n}^2 \chi_{i_0}^2 \tr_{\langle e \rangle}(\delta_{g_{i_1i_2}} \dots \delta_{g_{i_{n-1}i_1}})\fc(i_0,\dots, i_n) \ .$$
Since the summand is antisymmetric in the pair $(i_0,i_n)$, the sum vanishes.

By the cyclicity of the trace any term with a factor $P(\di P)(\di P)$ vanishes.

We calculate
\begin{eqnarray*}
\lefteqn{\bigl(P(d_MP)(\di P) +P(\di P)(d_M P)\bigr)_{il}}\\
&=& \sum_{j,k \in I}\chi_i e^{-i\psi^s_{g_{ij}}(\tilde x_j)}\chi_j \delta_{g_{ij}} d_M(\chi_je^{-i\psi^s_{g_{jk}}(\tilde x_k)}\chi_k)\delta_{g_{jk}} \chi_k e^{-i\psi^s_{g_{kl}}(\tilde x_l)}\chi_l (\di \delta_{g_{kl}})  \\
&&\quad + \chi_i e^{-i\psi^s_{g_{ij}}(\tilde x_j)} \delta_{g_{ij}}\chi_j^2  e^{-i\psi^s_{g_{jk}}(\tilde x_k)}\chi_k (\di \delta_{g_{jk}}) d_M (\chi_k e^{-i\psi^s_{g_{kl}}(\tilde x_l)}\chi_l) \delta_{g_{kl}}  \\
&=&\sum_{j,k \in I}\sigma^s(g_{ij},g_{jk})\chi_i e^{-i\psi^s_{g_{ij}}(\tilde x_j)}\chi_j d_M(\chi_je^{-i\psi^s_{g_{jk}}(\tilde x_k)}\chi_k)\chi_k e^{-i\psi^s_{g_{kl}}(\tilde x_l)}\chi_l \delta_{g_{ik}}(\di \delta_{g_{kl}})  \\
&&\quad - \sigma^s(g_{jk},g_{kl})\chi_i e^{-i\psi^s_{g_{ij}}(\tilde x_j)} \chi_j^2  e^{-i\psi^s_{g_{jk}}(\tilde x_k)}\chi_k  d_M (\chi_k e^{-i\psi^s_{g_{kl}}(\tilde x_l)}\chi_l) \delta_{g_{ij}}(\di \delta_{g_{jl}})  \\
&& \quad +  \sigma^s(g_{ij},g_{jk})\chi_i e^{-i\psi^s_{g_{ij}}(\tilde x_j)} \chi_j^2  e^{-i\psi^s_{g_{jk}}(\tilde x_k)}\chi_k  d_M (\chi_k e^{-i\psi^s_{g_{kl}}(\tilde x_l)}\chi_l)\delta_{g_{ik}}(\di \delta_{g_{kl}}) \ .  
\end{eqnarray*}

Here we used that $\delta_{g_{ij}}(\di \delta_{g_{jk}})\delta_{g_{kl}}=\sigma^s(g_{jk},g_{kl})\delta_{g_{ij}}(\di \delta_{g_{jl}})-\sigma^s(g_{ij},g_{jk})\delta_{g_{ik}}(\di \delta_{g_{kl}})$.

By adding up the first and the last line and interchanging $k$ and $j$ in the second line we get
\begin{eqnarray*}
\lefteqn{\bigl(P(d_MP)(\di P)+P(\di P)(d_M P)\bigr)_{il}}\\
&=&\sum_{j,k \in I}\sigma^s(g_{ij},g_{jk})\bar{\s}^s(g_{jk},g_{kl})\chi_i e^{-i\psi^s_{g_{ij}}(\tilde x_j)}\chi_j d_M(\chi_je^{-i\psi^s_{g_{jl}}(\tilde x_l)}\chi_k^2 \chi_l) \delta_{g_{ik}}(\di \delta_{g_{kl}}) \\
&&\quad - \sigma^s(g_{kj},g_{jl})\bar{\s}^s(g_{ik},g_{kj})\chi_i e^{-i\psi^s_{g_{ij}}(\tilde x_j)} \chi_k^2 \chi_j  d_M (\chi_j e^{-i\psi^s_{g_{jl}}(\tilde x_l)}\chi_l) \delta_{g_{ik}}(\di \delta_{g_{kl}}) \ . 
\end{eqnarray*} 
Since
$$\sigma^s(g_{ij},g_{jk})\bar{\s}^s(g_{jk},g_{kl})=\sigma^s(g_{ij},g_{jl})\bar{\s}^s(g_{ik},g_{kl})=\sigma^s(g_{kj},g_{jl})\bar{\s}^s(g_{ik},g_{kj})$$
we get 
\begin{eqnarray*}
\lefteqn{\bigl(P(d_MP)(\di P) +P(\di P)(d_M P)\bigr)_{il}} \\
&=&\sum_{j,k \in I}\sigma^s(g_{ij},g_{jl})\bar{\s}^s(g_{ik},g_{kl})\chi_i e^{-i\psi^s_{g_{ij}}(\tilde x_j)}\chi_j^2 e^{-i\psi^s_{g_{jl}}(\tilde x_l)}d_M(\chi_k^2) \chi_l \delta_{g_{ik}}(\di \delta_{g_{kl}}) \\
&=&\chi_i e^{-i\psi^s_{g_{il}}(\tilde x_l)}\chi_l \sum_{k \in I}\bar{\s}^s(g_{ik},g_{kl})d_M(\chi_k^2)\delta_{g_{ik}}(\di \delta_{g_{kl}}) \ .
\end{eqnarray*}
Now consider $\bigl((P(d_MP)(\di P)+ P(\di P)(d_M P))^2\bigr)_{ip}$. For that aim we calculate
\begin{eqnarray*}
\lefteqn{\delta_{g_{ik}}(\di \delta_{g_{kl}})\delta_{g_{lm}}(\di \delta_{g_{mp}})}\\
&=&\delta_{g_{ik}}(\di \delta_{g_{kl}}\delta_{g_{lm}})(\di \delta_{g_{mp}})-\delta_{g_{ik}}\delta_{g_{kl}}(\di\delta_{g_{lm}})(\di \delta_{g_{mp}})\\
&=&\sigma^s(g_{kl},g_{lm})\delta_{g_{ik}}(\di \delta_{g_{km}})(\di \delta_{g_{mp}})-\sigma^s(g_{ik},g_{kl})\delta_{g_{il}}(\di\delta_{g_{lm}})(\di \delta_{g_{mp}}) \ .
\end{eqnarray*}
Thus one of the two contributing terms is
$$\sum_{k,l,m}\chi_i e^{-i\psi^s_{g_{il}}(\tilde x_l)}d_M(\chi_k^2)\chi_l^2 e^{-i\psi^s_{g_{lp}}(\tilde x_p)}d_M(\chi_m^2)\chi_n \bar{\s}^s(g_{lm},g_{mp})\delta_{g_{il}}(\di\delta_{g_{lm}})(\di \delta_{g_{mp}}) \ .$$
This term vanishes since $\sum_k d_M(\chi_k^2)=0$.
In order to evaluate the second term we calculate
\begin{eqnarray*}
\bar{\s}^s(g_{ik},g_{kl})\sigma^s(g_{kl},g_{lm})\bar{\s}^s(g_{lm},g_{mp})
&=&\bar{\s}^s(g_{ik},g_{kl})\bar{\s}^s(g_{km},g_{mp})\sigma^s(g_{kl},g_{lp})\\
&=&\bar{\s}^s(g_{km},g_{mp}) \sigma^s(g_{il},g_{lp})\bar{\s}^s(g_{ik},g_{kp}) \ .
\end{eqnarray*}
Thus
\begin{eqnarray*}
\lefteqn{\bigl((P(d_MP)(\di P)+ P(\di P)(d_M P))^2\bigr)_{ip}}\\
&=&-\sum_{k,m}e^{-i\psi^s_{g_{ip}}(\tilde x_p)}\bar{\s}^s(g_{ik},g_{kp})\bar{\s}^s(g_{km},g_{mp})\chi_i d_M(\chi_k^2) d_M(\chi_m^2)\chi_p \delta_{g_{ik}}(\di \delta_{g_{km}})(\di \delta_{g_{mp}}) \ .
\end{eqnarray*}
By induction we get
\begin{eqnarray*}
\lefteqn{\bigl((P(d_MP)(\di P)+ P(\di P)(d_M P))^n \bigr)_{i_0i_{n+1}}}\\
&=&(-1)^{\frac{(n+1)(n+2)}{2}}\sum_{i_1, \dots, i_n \in I}e^{-i\psi^s_{g_{i_0i_{n+1}}}(\tilde x_{i_{n+1}})}\bar{\s}^s(g_{i_0i_1},g_{i_1i_{n+1}})\bar{\s}^s(g_{i_1i_2},g_{i_2i_{n+1}})\dots \\
& \dots& \bar{\s}^s(g_{i_{n-1}i_{n}},g_{i_{n}i_{n+1}})  \chi_{i_0}d_M(\chi_{i_1}^2)d_M(\chi_{i_2}^2) \dots d_M(\chi_{i_{n}}^2)\chi_{i_{n+1}} \delta_{g_{i_0i_1}}(\di \delta_{g_{i_1i_2}}) \dots (\di \delta_{g_{i_ni_{n+1}}})
\end{eqnarray*}

By using that
\begin{multline*}
\tr_{\langle e \rangle}(\delta_{g_{i_0i_1}}*\delta_{g_{i_1i_2}}* \dots *\delta_{g_{i_{n-1}i_n}}*\delta_{g_{i_ni_{n+1}}})\\
=\s^s(g_{i_0i_1},g_{i_1i_{n+1}})\s^s(g_{i_1i_2},g_{i_2i_{n+1}})\dots \s^s(g_{i_{n-1}i_{n}},g_{i_{n}i_{n+1}})\tr_{\langle e\rangle}(\delta_{g_{i_0i_{n+1}}}) 
\end{multline*}
it follows that
\begin{eqnarray*}
\lefteqn{\tau^{\sigma^s}_\fc \circ \tralg \bigl((P(d_MP)(\di P)+ P(\di P)(d_M P))^n \bigr)}\\
&=&(-1)^{\frac{(n+1)(n+2)}{2}}\sum_{i_0, \dots, i_n \in I}\chi_{i_0}^2d_M(\chi_{i_1}^2)d_M(\chi_{i_2}^2) \dots d_M(\chi_{i_{n}}^2)\fc(i_0,i_1,\dots, i_n)\\
&=&(- 1)^{\frac{(n+1)(n+2)}{2}}\olott_{\fc} \ .
\end{eqnarray*}

\end{proof}

If $\tau_\fc^{\s^s}$ extends to a continuous cyclic cocycle on $\Ai$, then the noncommutative Atiyah--Patodi--Singer index theorem implies that
$$\tau_\fc^{\s^s}(\ch(\ind \Di^M_{\mathcal V^{\s^s}}))=C\int_M \hat A(M) e^{-i\omega_s}\olott_\fc - \tau_\fc^{\s^s}(\eta(\Di^N_{\mathcal V^{\s^s}})) \ .$$

Here $C$ only depends on $\dim M$ and $n$.

If $\G$  is Gromov hyperbolic, then each class $[\fc] \in H^{n}(\Gamma)$ has a representative $\fc$ of polynomial growth \cite[p. 384 ff.]{cm}. Thus $\tau_\fc$ is extendable by Prop. \ref{rd_ext}.

\subsection{Delocalized cocycles and twisted higher $\rho$-invariants}
\label{deloc-sec}

Now we assume that $M$ is a closed odd-dimensional spin manifold with positive scalar curvature and $\Di_{\N^{\s^s}}$ is the twisted Dirac operator derived from the spin Dirac operator. 

With the notation as in the beginning of \S \ref{depchoic} we have the following variation formula.

\begin{prop}
For each $s$ near $0$ let $\tau_s$ be a reduced cyclic cocycle on $\C(\G,\s^s)$ which is invariant under parametrized projective automorphisms and which extends to a continuous cyclic cocycle on $\Ai$.

Then (as a germ near $s=0$)
\begin{align*}
\lefteqn{
\tau_s\eta(\Di_{\mathcal V^{\s^{s'}}})-\tau_s\eta(\Di_{\mathcal V^{\s^s}})}\\
&=(2\pi i)^{\frac{\dim M+1}{2}}\int_0^1 \hat A(M)\exp(-i(\omega_s+\chi d\phi_s+\chi dt\wedge \phi_s)) \tau_s \ch(\nabla^P+P\di P)  \ .
\end{align*}
\end{prop}

\begin{proof}
The strategy of proof is as in \S \ref{depchoic}. We skip some details.

We may set $\eta_s'=\eta_s+p^*\phi_s$ and $(\psi^s_{\g})'=\psi^s_{\g}$. Then the induced projections $P', P$ agree and $\nabla^{P'}=\nabla^{P}+ i\phi_s$.

On the bundle $P(M \times C^*(\Gamma,\sigma^s)^m)$ pulled back to the cylinder $Z=[0,1]\times M$ we have the connection $\nabla^Z=dt\frac{\partial }{\partial t}+\nabla^P+i \chi(t)\phi_s$. 
The associated Chern--Simons form is now defined as 
$$\cs(\nabla^P,\nabla^{P'})=\int_0^1\ch(P\di P+ \nabla^Z)=\int_0^1 \exp(-P\di P \di P - [P\di P,\nabla^Z]-(\nabla^Z)^2)\ .$$
 
Recall that $(\nabla^Z)^2=i(\omega_s+\chi d_M\phi_s+\chi'dt \wedge \phi_s)$. Furthermore $[P\di P, \nabla^Z]=[P\di P,\nabla^P]$. Thus we get
 
$$\cs(\nabla^P,\nabla^{P'})=\int_0^1 \tralg \exp(-i(\omega_s+\chi d_M\phi_s +\chi'dt \wedge \phi_s))\ch(\nabla^P+ P\di P)  \ .$$
 
This implies the assertion.
\end{proof}

One easily checks that the pairing of $\ch(\nabla^P+ P\di P)$ with delocalized cyclic cocycles vanishes.

Thus, if $\tau_s$ is delocalized for $s$ near $0$, then $\tau_s\eta(\Di_{\mathcal V^{ \s^s}})$ does not depend on the choices of $\psi_{\g}^s,~ \eta_s, ~\omega_s$ involved in its construction. 

\begin{lem}
If $\tau_s$ is delocalized for $s$ near $0$, then $\tau_s\eta(\Di_{\mathcal V^{\s^s}})$ is independent of the choice of the open cover and the partition of unity.
\end{lem}

\begin{proof}
The argument is as in \cite[Lemma 9.4]{waAPS}. We only point out the changes.
Let $(U_j')_{j \in J}$ be a second open cover. Assume furthermore given a subordinate partition of unity $((\chi_j')^2)_{j \in J}$ and trivialisations $\phi_j'\colon p^{-1}(U_j')\stackrel{\sim}{\rightarrow} U_j'\times
\Gamma$. Let $\tilde x_j'$ and $g_{ij}'$ be as above. We get a projection $P_{ij}'=\exp(-i\psi^s_{g_{ij}'}(\tilde x'_j))\chi'_i \chi'_j \delta_{g_{ij}'}$. For $i\in I, j \in J$ let $h_{ji}\colon U_j' \cap U_i \to \Gamma$ such that $h_{ji}\tilde x_i=\tilde x_j'$. Then 
$$U_{ji}=\chi'_j \exp(-i\psi^s_{h_{ji}}(\tilde x_i))\delta_{h_{ji}} \chi_i$$
is a partial isometry with $UU^*=P'$ and $U^*U=P$. We only show the last equation. 

We calculate
\begin{eqnarray*}
\lefteqn{(U^*U)_{kj}=\sum_{i\in J} (U_{ik})^*U_{ij}}\\
&=& \sum_{i\in J} \chi_k\exp(i\psi^s_{h_{ik}}(\tilde x_k)) (\chi'_i)^2 \exp(-i\psi^s_{h_{ij}}(\tilde x_j)) \chi_j \bar \s^s(h_{ik},h_{ik}^{-1})\delta_{h_{ik}^{-1}}*\delta_{h_{ij}}\\
&=&\sum_{i\in J} \chi_k\exp(i\psi^s_{h_{ik}}(\tilde x_k)) (\chi'_i)^2 \exp(-i\psi^s_{h_{ij}}(\tilde x_j))\chi_j \bar\s^s(h_{ik},h_{ik}^{-1})\sigma^s(h_{ik}^{-1},h_{ij})\delta_{g_{kj}} \ .
\end{eqnarray*}

With $\sigma^s(\g,\g^{-1})= e^{i(\psi^s_{\g^{-1}}+(\g^{-1})^*\psi^s_\g)} $ and $h_{ik}\tilde x_k=\tilde x_i'$,
 \begin{eqnarray*}
(U^*U)_{kj}
&=&\sum_{i\in J} \chi_k \exp(-i\psi^s_{h_{ik}^{-1}}(\tilde x_i')) (\chi'_i)^2 \exp(-i\psi^s_{h_{ij}}(\tilde x_j))\chi_j \sigma^s(h_{ik}^{-1},h_{ij})\delta_{g_{kj}} \\
&=&\chi_k \exp(-i\psi^s_{g_{kj}}(\tilde x_j))\chi_j\delta_{g_{kj}} \ .
\end{eqnarray*}

Here we also used that $\exp(-i(\psi^s_{h_{ik}^{-1}}(h_{ij}\tilde x_j)+\psi^s_{h_{ij}}(\tilde x_j)-\psi^s_{g_{kj}}(\tilde x_j)))=\bar\sigma^s(h_{ik}^{-1},h_{ij})$.

Thus $U^* P'\di P'U=P\di P+P U^*(\di U)P$ and $U^*P' d_M P'U=P d_M P + PU^*d_M(U)P$. In the following we show that $U$ intertwines the connections $\nabla^{P'}$ and $\nabla^P$.

Using the expression from the first line of the previous equation we calculate 
\begin{eqnarray*}
\lefteqn{(U^*d_M U)_{kj}}\\
&=&\sum_{i\in J} \chi_k \exp(-i\psi^s_{h_{ik}^{-1}}(\tilde x_i')) \chi'_i d_M\bigl(\chi_i' \exp(-i\psi^s_{h_{ij}}(\tilde x_j))\chi_j\bigr)
 \sigma^s(h_{ik}^{-1},h_{ij})\delta_{g_{kj}} \\
&=&\sum_{i\in J} \chi_k \exp(-i\psi^s_{h_{ik}^{-1}}(\tilde x_i')) \chi'_i d_M(\chi_i') \exp(-i\psi^s_{h_{ij}}(\tilde x_j))\chi_j
 \sigma^s(h_{ik}^{-1},h_{ij})\delta_{g_{kj}}\\
 &&+i\sum_{i\in J} \chi_k \exp(-i\psi^s_{h_{ik}^{-1}}(\tilde x_i')) \chi'{}_i^2 (\eta(\tilde x_j)-h_{ij}^*\eta(\tilde x_j))\exp(-i\psi^s_{h_{ij}}(\tilde x_j))\chi_j
 \sigma^s(h_{ik}^{-1},h_{ij})\delta_{g_{kj}}\\
&&+\sum_{i\in J} \chi_k \exp(-i\psi^s_{h_{ik}^{-1}}(\tilde x_i')) \chi'{}_i^2  \exp(-i\psi^s_{h_{ij}}(\tilde x_j))d_M\chi_j
 \sigma^s(h_{ik}^{-1},h_{ij})\delta_{g_{kj}}\\
&=&i\chi_k \exp(-i\psi^s_{g_{kj}}(\tilde x_j))\chi_j\delta_{g_{kj}} \sum_{i\in J}\chi'{}_i^2  (\eta(\tilde x_j)-\eta(\tilde x_i'))\\
 &&+\chi_k \exp(-i\psi^s_{g_{kj}}(\tilde x_j))d_M\chi_j\delta_{g_{kj}} \\
 &=&i P_{kj} \bigl(\eta(\tilde x_j) - \sum_{i\in J}\chi'{}_i^2\eta(\tilde x_i')\bigr)+\chi_k \exp(-i\psi^s_{g_{kj}}(\tilde x_j))d_M\chi_j\delta_{g_{kj}} \ .
\end{eqnarray*}

Now 
\begin{eqnarray*}
\lefteqn{(U^*(d_M U) P)_{kl}}\\
&=&- i P_{kl} \sum_{i\in J}\chi'{}_i^2\eta(\tilde x_i') +i\sum_{j \in I}P_{kj}\eta(\tilde x_j)P_{jl}\\
&& +\sum_{j \in I}\chi_k \exp(-i\psi^s_{g_{kj}}(\tilde x_j))(d_M\chi_j)\chi_j \exp(-i\psi^s_{g_{jl}}(\tilde x_l))\chi_l \delta_{g_{kj}}*\delta_{g_{jl}} \ .
\end{eqnarray*}

By arguments, which are by now standard, the last line vanishes and 
\begin{eqnarray*}
\lefteqn{\sum_{j \in I}P_{kj}\eta(\tilde x_j)P_{jl}}\\
&=&\sum_{j \in I}\chi_k \exp(-i\psi^s_{g_{kj}}(\tilde x_j))\chi_j^2\eta(\tilde x_j) \exp(-i\psi^s_{g_{jl}}(\tilde x_l))\chi_l \delta_{g_{kj}}*\delta_{g_{jl}}\\
&=&P_{kl}  \sum_{j \in I} \chi_j^2 \eta(\tilde x_j) \sum_{j \in I} \chi_j^2 \eta(\tilde x_j) \ .
\end{eqnarray*}

Thus
$$PU^*d_M(U)P=  i\sum_{i \in I} \chi_i^2 \eta(\tilde x_i)P -i\sum_{j\in J} (\chi_j')^2 \eta(\tilde x_j')P$$
and therefore $U^* \nabla^{P'} U=\nabla^P$.

Let $\chi\colon [0,1] \to [0,1]$ be a smooth cut-off function that vanishes for  $x < 1/4$ and equals $1$ for $x>3/4$. Set $Z=[0,1] \times N$. Define $\gamma(t,x)=\chi(t)PU^*\di (U)P(x)$ for $(t,x) \in Z$. On the bundle $P(Z\times C^*(\Gamma,\sigma^s)^m)$ on $Z$ we have the connection $P\di P + \gamma + dt \frac{\partial}{\partial t} + \nabla^P$. Note that this is a connection in the sense of \cite[\S 1.2]{waAPS}. By the noncommutative Atiyah--Patodi--Singer index theorem applied to the associated Dirac operator the difference of the $\eta$-invariants defined using $P$ and $P'$ equals 
$C\int_Z \hat A(Z)\tau_s(\ch(P\di P + \gamma + dt \frac{\partial}{\partial t} + \nabla^P))$. One checks that the pairing of $\ch(P\di P + \gamma + dt \frac{\partial}{\partial t} + \nabla^P)$ with delocalized cocycles vanishes.
\end{proof}

The previous lemmata motivate the following definition:

\begin{definition}
For $s$ near $0$ let $\tau_s$ be a delocalized cyclic cocycle on $\C(\G,\s^s)$ which is invariant under parametrized projective automorphisms. We also assume that the cocycle extends to a continuous cyclic cocycle on $\Ai$.

The twisted higher $\rho$-invariant is defined as
$$\rhoct(M,g,f_M)=\tau_s(\eta(\Di_{\mathcal V^{\s^s}})) \ .$$

As usual, we understand it as a germ at $s=0$.
\end{definition}

It follows as in Prop. \ref{bord-inv} that the twisted higher $\rho$-invariants are invariant under strong $\G$-bordism.

We conclude by studying product formulas which are helpful for calculations and applications. We only consider $\rho$-invariants here, but there are analogous formulas for $\eta$-invariants. 

The situation is similar to the one before Prop. \ref{prod}.

Thus, let $M=L\times N$ be the product of two closed connected spin manifolds with metrics $g_M=g_L+g_N$ and let  $f_L \colon L \to B\G$, $f_N\colon N \to BG$ be maps inducing universal coverings. We have $f_M=f_L \times f_N$. We assume that $L$ and $M$ are odd-dimensional and have positive scalar curvature. Furthermore assume that $\G$ has finite abelization. 

Let ${\mathbf c} \in H^2(\G, \Q)^k$. The above constructions applied to $L$ yield a twisted group $\C(\G,\s^s)$ for any $s$ near $0 \in \Q^k$. Let $\tau_s$ be a delocalized trace on it. We get a twisted $\rho$-invariant $\rho_{\tau_s}^{\mathbf c}(L,g_L,f_L)$.  

Let $\fc \in C^m_{al}(G)$ be a cocycle and $\tau_\fc$ the induced reduced cyclic cocycle on $\C G$. 

Then $\tau_s \# \tau_\fc$ is a reduced cyclic cocycle on $\C(\G \times G, \pi_1^*\s^s)$. Here $\pi_1\colon \G \times G \to \G$ is the projection. It is easy to check that the cocycle is invariant under projective automorphisms.

The product formula we want to obtain now looks as follows:

$$\rho_{\tau_s \# \tau_{\fc}}^{\pi_1^*\mathbf c}(M,g_M,f_M)=C\rho_{\tau_s}^{\mathbf c}(L,g_L,f_L)\int_N \hat A(N) \olott^N_\fc \ .$$

Here $\olott_\fc^N \in \Omega^*(N)$ is the form defined before Prop. \ref{chloc}. The constant $C$ only depends on $m$ and $\dim N$. (One may prove a still more general product formula by taking an additional multiplier on $G$ into account.)
 
The main difficulty is to make the $\rho$-invariants on both sides of the product formula well-defined, \emph{i.\,e.} to find appropriate smooth subalgebras which are well-behaved with respect to products and to which $\tau_s \# \tau_\fc$ extends in a continuous way. In general, also in the untwisted situation the interplay between products and the extension property for cyclic cocycles has not yet received much attention. 
   
Using the results from Appendix \ref{Appendix} we get the desired properties in some situations.

\begin{prop}
Assume that the cocycle $\fc\in C^m_{al}(G)$ has polynomial growth.

The above product formula is well-defined and holds in the following situations:

\begin{enumerate}

\item The group $G$ polynomial growth and the trace $\tau_s$ is defined on $C^*(\G, \s^s)$.

\item The group $G$ has polynomial growth and it holds that $\tau_s=\tau_{\langle g\rangle}$ for some conjugacy class $\langle g\rangle \subset \G$ of polynomial growth.

\item The group $\G\times G$ has rapid decay and it holds that $\G=\G_1 \times \G_2$ and $\G_1$ has finite abelization. The twist is induced by an element ${\mathbf c} \in H^2(B\G_2,\Q)^k$, and $\tau_s=(\tr^{\G_1}_1-\tr^{\G_1}_{(2)}) \ten \tr^{\s^s}_{(2)}$.
\end{enumerate}
\end{prop}

\begin{proof}
We choose word length functions on $G$ and $\G$.

Once appropriate smooth subalgebras are defined to which the cocycle $\tau_s \# \tau_\fc$ extend continuously, the formula follows directly from the product formula of noncommutative $\eta$-forms \cite[Prop. 10.1]{waAPS}, which generalizes the formula proven in \cite[\S 2]{lpetapos} in the higher case.

We use that smooth subalgebras behave well under intersection and pullback. (If a norm is pulled back, one gets a seminorm, which is then turned into a norm by adding the norm of the preimage.)

\medskip

(1) Set $\cB=C^*(\G,\s^s)$. By Prop. \ref{productext} (1) the cocycle extends to a smooth subalgebra $C_{\infty}(G,\cB) \subset C_r^*(\G,\s) \ten C_r^*(G)$.  

\medskip

(2) By the proof of Prop. \ref{conj_class} the trace $\tau^s\colon H_l^{\infty}(\G) \to \C$ is continuous. Prop. \ref{productext} (2) implies that $\tau_s \# \tau_c \colon C_{\infty}(\G \times G, \pi_1^*\s^s) \to \C$ is continuous. Now we pull $C_{\infty}(\G \times G, \pi_1^*\s^s) \subset C_r^*(\G \times G, \pi_1^*\s)\cong C_r^*(\G,\s^s) \ten C_r^*(G)$ back.

\medskip

(3) Let $\G \times G \stackrel{p_2}{\longrightarrow} \G_2 \times G \stackrel{p_{22}}{\longrightarrow} G$ and $\G_2 \times G \stackrel{p_{21}}{\longrightarrow} \G_2$ be the projections. With $p_{22}^*\fc \in C^{al}(\G_2 \times G)$ it holds that $\tr^{\s^s}_{(2)} \# \tau_\fc=\tau_{p_{22}^*\fc}$ as a reduced cyclic cocycle on $\C(\G_2 \times G, p_{21}^*\s)$. Since $\tau_{p_{22}^*\fc}$ is of polynomial growth and $\G_2 \times G$ is of rapid decay, the cocycle $\tau_{p_{22}^*\fc}$ extends to $C_{\infty}(\G_2 \times G, p_{21}^*\s^s)$. Since $(\tr^{\G_1}_1 \ten \tr^{\s^s}_{(2)})\# \tau_\fc$ is the pullback of $\tau_{p_{22}^*\fc}$ via $p_2$ to $\C(\G \times G,(p_{21}\circ p_2)^*\s^s)$, it also extends to the pullback of $C_{\infty}(\G_2 \times G, p_{21}^*\s^s)$.

Analogously, $(\tr^{\G_1}_{(2)} \ten \tr^{\s^s}_{(2)})\# \tau_\fc=\tau_{(p_{22} \circ p_2)^*\fc}$ extends to $C_{\infty}(\G \times G,(p_{21}\circ p_2)^*\s^s)$. 
Now we take the intersection of these two smooth subalgebras in the reduced twisted group $C^*$-algebra and the pullback to the maximal twisted group $C^*$-algebra.

\end{proof}

In general, conjugacy classes only define linear functionals on twisted group algebras. However, for example in the product situation, it is easy to find conjugacy classes whose associated functional is a trace.

The third rather special case was discussed here since it may directly be combined with the results in \S \ref{app} to get applications.

One easily obtains more results and applications along these lines.

\section{Relation between twisted $\eta$-invariants and higher $\eta$-invariants}
\label{sec10}
If $c\in H^2(B\G,\Q)$ then, by the index theorem of Gromov \cite[Thm. 2.3.B]{Gr1}\cite[Thm. 3.6]{Ma1}, the twisted $L^2$-index of a Dirac operator is equal to its higher index associated to group cocycle $\mathfrak{e}(s):=e^{-2 \pi i s c}$.

In the following we establish an identity between the twisted $L^2$-$\eta$-invariants and the higher $\eta$-invariants defined by Lott for a manifold $M$ with positive scalar curvature. Let $f\colon M \to B\Gamma$ be a classifying map for the universal covering. We start with a group cocycle $\fc \in C^2_{al}(\Gamma)$. Let  $s\in [0,\ep_0]$ and  assume that the cyclic cocycle $\tau_\mathfrak{e(s)}$ associated to the group cocycle $\mathfrak{e}(s)$ is extendable with respect to a suitable subalgebra $\Ai \subset C_r^*\Gamma$. 

\medskip
For the definition of the higher $\eta$-invariant we start with an open cover $(U_i)_{i\in I}$ of $M$ of trivializing sets for the universal covering and a subordinate partition of unity $(\chi_i^2)_{i \in I}$. As in the previous section (but now without twist) this gives a projection $P \in \Cinf(M,M_{|I|}(\C(\Gamma))$, which is used to define the higher $\eta$-invariant $\eta(\Di_{\mathcal V})$ as in the previous section. Here we are interested in the numerical invariant $\tau_{\mathfrak{e}(s)}(\eta(\Di_{\mathcal V}))$.

\medskip
For the twisted $\eta$-invariant we need to specify the choice of $\omega$. 

By (\ref{lottform}) the previous data induce a closed form $\olott_{\fc} \in \Omega^2(M)$. We take $\omega=-2 \pi \olott_{\fc}$ for the construction of the twisted $L^2$-$\eta$-invariant. Thus we also get a group multiplier $\sigma^s$. By \S \ref{depchoic} the form $\omega$ induces a well-defined twisted $L^2$-$\eta$-invariant $\eta_{(2)}^{\sigma_s}(\Di_{\mathcal V^{\s^s}})$.

\medskip

Recall that there is a reduced crossed product $C([0,\ep_0])\rtimes_{\sigma^s} \Gamma$ coming with an evaluation map $\ev_s$ to $C^*_r(\Gamma,\sigma^s)$. This map induces an isomorphism in $K$-theory if $\Gamma$ satisfies the Baum--Connes conjecture with coefficients \cite[Cor. 1.11]{elpw}.

\begin{theorem} 
\label{twist-vs-high}
Under the following assumptions it holds that 
\begin{equation}
\label{tw-h}
\eta_{(2)}^{\sigma_s}(\Di_{\mathcal V^{\s^s}})=\tau_{\mathfrak{e}(s)}(\eta(\Di_{\mathcal V})) \ .
\end{equation}

\begin{enumerate}
\item
We assume that $\Gamma$ satisfies the Baum--Connes conjecture with coefficients;
% and that the rational Baum--Connes map $KO_*(B\Gamma)\ten \Q \to KO_*(C_r^*\Gamma)\ten \Q$ is injective.
\item
We assume that the following diagram commutes:
\begin{equation}
\label{diag}\xymatrix{ K_0(C_r^* (\Gamma,\sigma^s)) \ar[rdd]_{\tau_{(2)}}\ar[rr]^{\ev_0 \circ \ev_s^{-1}} && K_0(C_r^*\Gamma) \ar[ddl]^{\tau_{\mathfrak{e}(s)}} \\
\\
& \C  &}
\end{equation}
\end{enumerate}
\end{theorem}

We conjecture that the equality \eqref{tw-h} holds in general.

\begin{proof}
We use methods of \cite[\S 4]{pstorsion}. We denote by $B$ a Bott manifold, \emph{i.e} an $8$-dimensional simply connected manifold such that $\int_B \hat A(B)=1$.

Since $\Gamma$ fulfills the Baum--Connes conjecture, it also fulfills the real Baum--Connes conjecture \cite{BK}. Therefore the rational Baum--Connes map $KO_*(B\Gamma)\ten \Q \to KO_*(C_r^*(\Gamma,\R))\ten \Q$ is injective. Here $C^*_r(\G,\R)$ is the real reduced group $C^*$-algebra. 

Since $M$ has positive scalar curvature, the image of $(M,f)$ vanishes in $KO_*(C_r^*(\Gamma,\R))$. Therefore there is $l$ such that $\sqcup_{j=1}^l (M,f)$ vanishes in $KO_*(B\Gamma)$.

This implies, by \cite[Lemma 4.2]{pstorsion}, that there exist spin manifolds $A_i$, $C_i$, $i=1,\dots, n$, with $\int_{C_i}\hat A (C_i)=0$ and maps $f_i \colon A_i \to B\Gamma$ such that there is a spin bordism $(W,f_W)$ between $\sqcup_{j=1}^l(M \times B^n, f \circ \pi_1)$ and $\sqcup_i (A_i \times C_i, \sqcup_i f_i \circ \pi_1)$ for $n$ high enough. Here $\pi_1$ denotes the projection onto the first factor. By taking the Cartesian product with $B$ if necessary and using spin surgery on $A_i$, $C_i$, $W$ we may assume that $C_i$ is simply connected and that $f_i$, $f_W$ induce universal coverings. Note that $C_i$ can be endowed by a metric of positive scalar curvature \cite{St1}, which induces a metric of positive scalar curvature on $A_i \times C_i$. 

We can define a cover of $W$ and a subordinate partition of unity which restricts to the given one on $M$ and is of product type on $A_i \times C_i$. This yields an extension of the projection $P$ to $W$ and a 2-form $\olott_{\fc}^W$ on $W$. We write $\Di_{\mathcal V}^W$ for the spin Dirac operator on $W$ twisted by this extended projection and apply the higher Atiyah--Patodi--Singer theorem to it \cite{lphighAPS}. By the product formula \cite{lpetapos} the higher $\eta$-invariants of $A_i \times C_i$ vanish. In the same way we apply the product formula to $M \times B^n$ and get
$$\tau_{\mathfrak{e}}(s)(\eta(\Di_{\mathcal V}))=\tau_{\mathfrak{e}(s)}(\ind (\Di_{\mathcal V}^W)) + C\int_W \hat A(W)e^{2\pi i \olott^W_\fc} \ .$$

Now we consider the twisted situation. Using the product formula and the twisted Atiyah--Patodi--Singer index theorem we get 
$$\eta_{(2)}^{\sigma_s}(\Di_{\mathcal V^{\s^s}})=\ind_{(2)}(\Di_{\mathcal V^{\sigma^s}}^W) + C\int_W \hat A(W) e^{2\pi i\olott^W_{\fc}} \ .$$

It remains to show that the contributions from the indices agree.

For this we note that the bundles $\mathcal V^{\sigma^s}$ assemble to a bundle with connection $\mathcal V^{\sigma}$ over $C([0,\ep_0])\rtimes_{\sigma^s} \Gamma$ and that the Dirac operator twisted by this bundle $\Di^W_{\mathcal V^{\sigma}}$ is a Fredholm operator. It holds that $$\ev_{s}(\ind \Di^W_{\mathcal V^{\sigma}})=\ind(\Di_{\mathcal V^{\sigma^s}}^W) \in K_0(C^*(\Gamma,\sigma^s)) \ .$$
This, in connection with the second assumption, implies the assertion.
\end{proof}

We finally give some examples in which the assumptions of Theorem \ref{twist-vs-high} are satisfied:

\begin{enumerate}
\item Let $\Gamma$ be a torsion free group that satisfies the Baum--Connes conjecture. Then any element of $K_*(C_r^*\Gamma)$ can be represented as the index class of a Dirac operator $D_{\mathcal V}$ over a closed manifold $M$ with $\pi_1(M)=\Gamma$, and where $\mathcal V$ is the Mishchenko bundle. 
 By the higher index theorem and Mathai's twisted index theorem, $ \tau_{\mathfrak{e}(s)} (\ind D_\mathcal V)=\tau_{(2)}(\ind D_{{\mathcal V}^\sigma})$ so that the diagram \eqref{diag} commutes.
\item Let $\Gamma$ be a cocompact Fuchsian group, \emph{i.e.} a 
 discrete, cocompact subgroup of ${\rm PSL}(2, \R)$. By \cite[Prop. 2.1]{mm1}, every class in $K_*(C^*_r\Gamma)$ is the index class of a Dirac operator on a good orbifold. The higher index formula of Mathai--Marcolli \cite[Thm 2.2]{mm} shows as above that the diagram \eqref{diag} commutes.
\end{enumerate}

\appendix
\section{Rapid decay, polynomial growth and smooth subalgebras}
\label{Appendix}

In the following we adapt results from \cite[p. 383ff.]{cm} to the present situation.

Let $\G$ be a finitely generated group and $l$ be a word length function on $\Gamma$. For $s \ge 0$ the Hilbert space $H^s_l(\Gamma)$ is defined as the completion of $\C(\Gamma)$ with respect to the norm induced by the scalar product
$$\langle\sum_{\g \in \Gamma} a_\g \delta_\g,\sum_{\g \in \Gamma} b_\g \delta_\g\rangle_s:=\sum_{\g\in \Gamma} a_\g^*b_\g(1+l(\g))^{2s} \ .$$
The projective limit of these spaces is denoted by $H^{\infty}_l(\Gamma)$. 

By definition, the group $\Gamma$ has rapid decay if $H^{\infty}_l(\Gamma) \subset C^*_r(\Gamma)$. For the properties of these groups and spaces see \cite{jo1,jo}. Note that in the Appendix by Chatterji to \cite{marange} the notion of twisted rapid decay has been defined and studied. However, we will not need this notion in the following. 

\medskip
Let $\sigma$ be a multiplier on $\G$.
In the following the norm on $C^*_r(\Gamma,\sigma)$ is denoted by $\| \cdot \|$ and the product, as usual, by $*$. 

We define an unbounded operator $D_l$ on $l^2(\Gamma)$ with domain $\C\Gamma$ by $D_l (a_\g \delta_\g)=l(\g)a_\g \delta_\g$. Note that $D_l$ is closable. It holds that 
$$D_l(\delta_\g *\delta_{\g'})-\delta_\g *D_l \delta_{\g'}=(l(\g\g')-l(\g')) \delta_{\g}*\delta_{\g'}\ .$$
Since $l(\g \g') \le l(\g)+l(\g')$, it follows that $l(\g \g')-l(\g') \le l(\g)$. Analogously, $l(\g')-l(\g\g
')\le l(\g^{-1})=l(\g)$. Hence $|l(\g\g')-l(\g')|\le l(\g)$ and therefore $(\g'\mapsto l(\g\g')-l(\g')) \in l^{\infty}(\Gamma)$.

\medskip

We denote by $B$ the smallest subalgebra of $\cB(l^2(\Gamma))$ containing $\C(\Gamma,\sigma)$ and multiplication by elements of $l^{\infty}(\Gamma)$. By the previous calculation $[\delta_g,D_l] \in B$. 

Let $\cB$ be the $C^*$-algebraic completion of $B$. Then 
$$\der\colon B\to \cB,~\der(T)=[D_l,T]$$
is a densely defined closable derivation whose image is in $B$. We set $\cB_0=\cB$ and define inductively for $j \in \bn$ the Banach algebra $\cB_j$ as the domain of the closure $\der_j$ of $\der\colon B\to \cB_{j-1}$ and endow it with the norm $\|b\|_j:=\|b\|_{j-1} + \|\der_jb\|_{j-1}$. Furthermore we set $C_j(\G,\sigma):=\cB_j\cap C^*_r(\Gamma,\sigma)$. Each $C_j(\G,\sigma)$ is closed under holomorphic functional calculus in $C^*_r(\Gamma,\sigma)$. We write $C_{\infty}(\G,\s)$ for the projective limit.

\begin{prop}
\label{coninj}
For $j\in \bn$ the identification of vector spaces $\C(\Gamma,\sigma) \to \C\G$ extends to a continuous Banach space homomorphism $C_j(\G,\s) \to H^j_l(\G)$. 
\end{prop}

\begin{proof}
For $x \in \C(\Gamma,\sigma)$ it holds that $\der(x)\delta_e=D_l(x)$. Thus, by induction, for any $j \in \bn$, 
$$\der^j(x)\delta_e=D_l(\der^{j-1}(x)\delta_e)=D_l^j(x) \in l^2(\Gamma) \ .$$ 
If $x =\sum_{\g \in \Gamma} x_\g \delta_{\g}$, then $D_l^j(x)=\sum_{\g \in \Gamma} l(\g)^jx_\g \delta_\g$. Taking the norm we get that $\|\der^j(x)\delta_e\|_{l^2}^2=\sum_{\g \in \Gamma} l(\g)^{2j}|x_\g|^2$. It follows that there is $C>0$ such that for all $x \in \C(\G,\sigma)$  
$$\|x\|_{H^j} \le C\sum_{i=0}^j \|\der^i(x)\delta_e\|_{l^2} \ .$$ 
Since $\|\der^i(x)\delta_e\| \le \|\der^i(x)\|$ and $\sum_{i=0}^j \|\der^i(x)\| \le \|x\|_j$, this implies the assertion.
\end{proof}

In the following we study the relation between group cohomology and reduced cyclic cohomology of the twisted group algebra by elaborating on arguments from \cite[\S 2.3]{mm}.

The definition of group cohomology we are dealing with is the following:

Let $C^n(\Gamma)$ be the vector space of $\G$-invariant maps from $\Gamma^{n+1}$ to $\C$. There is a differential 
$$d_{\Gamma}\colon C^n(\Gamma)\to C^{n+1}(\Gamma) \ ,$$
$$(d_{\Gamma}\tau)(\g_0,\g_1, \dots,  \g_{n+1})=\sum_{j=0}^{n+1} (-1)^j\tau(\g_0,\g_1, \dots  , \widehat\g_j, \dots, \g_{n+1})\ .$$ 
Denote by $C_{al}^n(\Gamma)\subset C^n(\Gamma)$ the subspace of alternating elements. The homology of $(C_{al}^*(\Gamma),d_{\Gamma})$ is the group cohomology $H^*(\Gamma)$.

Reduced cyclic cohomology of our twisted algebra can be defined as follows:

Let $\ov{C}_n^{\lambda}(\C(\Gamma,\sigma))$ be the quotient of $(\C(\Gamma,\sigma)/\C)^{\ten n+1}$ by the action of $\Z/(n+1)\Z$ given by permutation.
Here the boundary map is given by $$b\colon \ov{C}_n^{\lambda}(\C(\Gamma,\sigma)) \to \ov{C}_{n-1}^{\lambda}(\C(\Gamma,\sigma)) \ ,$$
\begin{eqnarray*}
b(a_0 \ten a_1 \ten \dots \ten a_n)&=&(-1)^na_n*a_0 \ten a_1 \ten \dots \ten a_{n-1}\\
&& + \sum_{i=0}^{n-1}(-1)^i
a_0 \ten \dots  \ten a_i*a_{i+1}\ten \dots \ten a_n \ .
\end{eqnarray*} 
The homology of the complex $(\ov{C}_*^{\lambda}(\C(\Gamma,\sigma)),b)$ is by definition the \emph{reduced cyclic homology} $\ov{HC}_*(\C(\Gamma,\sigma))$.

\medskip
The dual complex $(\ov{C}^*_{\lambda}(\C(\Gamma,\sigma)),b^t)$ yields the \emph{reduced cyclic cohomology} of $\C(\Gamma,\sigma)$.

\medskip

Set $$\ov{C}_{\lambda}^n(\C(\Gamma,\sigma))_{\langle e\rangle}=\{\tau \in \ov{C}_{\lambda}^n(\C(\Gamma,\sigma))~|~ \tau(\delta_{g_0},\delta_{\g_1}, \dots, \delta_{\g_n})=0 \mbox{ for } \g_0\g_1 \dots \g_n \neq e \} \ .$$

We call elements of $\ov{C}_{\lambda}^n(\C(\Gamma,\sigma))_{\langle e\rangle}$ \emph{localized at the identity}. In contrast, we call $\tau \in \ov{C}_{\lambda}^n(\C(\Gamma,\sigma))$ \emph{delocalized} if $\tau(\delta_{\g_0},\delta_{\g_1}, \dots, \delta_{\g_n})=0$ for $\g_0\g_1 \dots \g_n = e$. Note that reduced cyclic cocycles which are localized at the identity are invariant under the action of projective automorphisms on $\C(\Gamma,\sigma)$.

The homology of the complex $(\ov{C}_{\lambda}^n(\C(\Gamma,\sigma))_{\langle e\rangle},b^t)$ will be denoted by $\ov{HC}^*(\C(\Gamma,\sigma))_{\langle e\rangle}$.

\begin{prop}
There is an isomorphism of complexes 
$$(C_{al}^*(\Gamma),d_{\Gamma}) \to (\ov{C}^*_{\lambda}(\C(\Gamma,\sigma))_{\langle e\rangle},b^t)\;\;\;\; c \mapsto \tau_{c} \ ,$$ where $$\tau_c(\delta_{\g_0},\delta_{\g_1}, \dots , \delta_{\g_n})=\tr_{(2)}^\s(\delta_{\g_0}*\delta_{\g_1}* \dots *\delta_{\g_n})c(e,\g_1,\g_1\g_2, \dots, \g_1\g_2\dots \g_n) \ .$$
\end{prop}

\begin{proof}
The result is well-known for $\sigma=1$. Since we deal with the algebras $\C(\G,\s)$ and $\C\G$ in the following, we write $*_{\s}$ for the product on $\C(\G,\s)$.

We only need to prove that the map
$$T \colon\ov{C}^n_{\lambda}(\C\Gamma)_{\langle e\rangle } \to \ov{C}^n_{\lambda}(\C(\Gamma,\sigma))_{\langle e\rangle},~\tau \mapsto \tau^{\sigma}$$
with 
$$\tau^{\sigma}(\delta_{\g_0},\delta_{\g_1}, \dots, \delta_{\g_n}):=\tr_{(2)}^{\s}(\delta_{\g_0}*_{\s}\delta_{\g_1}*_{\s} \dots *_{\s}\delta_{\g_n})\tau(\delta_{\g_0},\delta_{\g_1}, \dots,\delta_{\g_n})$$
defines an isomorphism of complexes. Since for $\g_0\g_1 \dots \g_n = e$ we have that $\tr_{(2)}^\s(\delta_{\g_0}*_{\s}\delta_{\g_1}*_{\s} \dots *_{\s}\delta_{\g_n})\neq 0$, it is clear that $T$ has an inverse. It remains to show that $T$ is compatible with the differential.

It holds that 
\begin{align*}
(b^t\tau^{\sigma})(\delta_{\g_0},\delta_{\g_1}, \dots, \delta_{\g_{n+1}})&=(-1)^n \tau^{\sigma}(\delta_{\g_{n+1}}*_{\s}\delta_{\g_0},\delta_{\g_1}, \dots, \delta_{\g_n})\\
&\quad + \sum_{i=0}^{n-1}(-1)^i\tau^{\sigma}(\delta_{\g_0},\delta_{\g_1}, \dots, \delta_{\g_i}*_{\s}\delta_{\g_{i+1}}, \dots, \delta_{\g_n})\\
&=(-1)^n \sigma(\g_{n+1},\g_0)\tau^{\sigma}(\delta_{\g_{n+1}\g_0},\delta_{\g_1}, \dots, \delta_{\g_n})\\
&\quad + \sum_{i=0}^{n-1}(-1)^i\sigma(\g_i,\g_{i+1}) \tau^{\sigma}(\delta_{\g_0},\delta_{\g_1}, \dots, \delta_{\g_i\g_{i+1}}, \dots, \delta_{\g_n}) \ .
\end{align*}
Now one uses that, for example, 
\begin{align*}
\lefteqn{\sigma(\g_{n+1},\g_0)\tau^{\sigma}(\delta_{\g_{n+1}\g_0},\delta_{\g_1}, \dots, \delta_{\g_n})}\\
&=\sigma(\g_{n+1},\g_0)\tr^\s_{(2)}(\delta_{\g_{n+1}\g_0}*_{\s}\delta_{\g_1}*_{\s} \dots *_{\s}\delta_{\g_n})\tau(\delta_{\g_{n+1}\g_0},\delta_{\g_1}, \dots, \delta_{\g_n})\\
&=\tr_{(2)}^\s(\delta_{\g_0}*_{\s}\delta_{\g_1} \dots \delta_{\g_n}*_{\s}\delta_{\g_{n+1}})\tau(\delta_{\g_{n+1}\g_0},\delta_{\g_1}, \dots, \delta_{\g_n}) \ .
 \end{align*}
One concludes that $b^t\tau^{\sigma}=Tb^t \tau$.
\end{proof}

\begin{prop}
\label{rd_ext}
Assume that $\Gamma$ has rapid decay. Let $c \in C_{al}^n(\Gamma)$ be a group cocycle of polynomial growth. Then the associated reduced cyclic cocycle $\tau_c$ on $\C(\G,\s)$ extends to a continuous cyclic cocycle on $C_{\infty}(\G,\s)$.
\end{prop}

\begin{proof}
For $\sigma=1$ it has been proven in \cite[p. 383ff.]{cm} (based on ideas of Jolissaint) that $\tau_c$ is continuous as a multilinear map from $H_l^{\infty}(\Gamma)$ to $\C$. Since $C_{\infty}(\G,\s) \to H_l^{\infty}(\G)$ is continuous by Prop. \ref{coninj} and
$$C_{\infty}(\G,\s)^{n+1} \to \C, ~(a_0,a_1 \dots, a_n)\mapsto \tr_{(2)}^\s(a_0*a_1* \dots *a_n)$$
is continuous, the map $T$ from the previous proof maps reduced cyclic cocycles on $\C\G$ which extend to continuous maps from $H_l^{\infty}(\Gamma)^{n+1}$ to $\C$ to reduced cyclic cocycles on $\C(\G,\s)$ which extend to continuous maps from $C_{\infty}(\G,\s)^{n+1}$ to $\C$.
\end{proof}

For applications it is also interesting to have continuity for a linear functional $\tau_{\langle g \rangle}$ associated to a conjugacy class $\langle g \rangle$ of $\G$. Note that, contrary to the untwisted case, this functional is not necessarily a trace. The following is an adaption of \cite[Prop. 13.5]{ps1}.

\begin{prop}
\label{conj_class}
If $\langle g \rangle$ is of polynomial growth, then the linear functional\\ $\tau_{\langle g \rangle}\colon C_{\infty}(\G,\s) \to \C$ is continuous.
\end{prop}

\begin{proof}
As in the proof of \cite[Thm. 3.1.7]{jo} one shows that the map $H_l^{\infty}(\Gamma) \to l^1(\langle g\rangle)$ is well-defined and continuous. It follows that $\tau_{\langle g \rangle} \colon H_l^{\infty}(\Gamma) \to \C$ is continuous and thus also $\tau_{\langle g \rangle}\colon C_{\infty}(\G,\s) \to \C$.  
\end{proof}

Aiming at the product situation we generalize some of the above considerations to the case of coefficients in a unital $C^*$-algebra $\cA$. In applications this will be the (full or reduced) twisted $C^*$-algebra of a second group. 
\medskip

For $s>0$ we define the Hilbert $\cA$-module $H^s_l(\Gamma,\cA)$ as the completion of the algebraic tensor product $\C(\Gamma)\odot \cA$ with respect to the norm induced by the $\cA$-valued scalar product
$$\langle\sum_{\g \in \Gamma} a_\g \delta_\g,\sum_{\g \in \Gamma} b_\g \delta_\g\rangle_s:=\sum_{\g\in \Gamma} a_\g^*b_\g(1+l(\g))^{2s} \ .$$
The projective limit of these spaces is denoted by $H^{\infty}_l(\Gamma,\cA)$.

As Hilbert $\cA$-modules $H^s_l(\Gamma,\cA)\cong H^s_l(\Gamma)\ten \cA$.

\medskip

We have a closable operator $D_l$ on the Hilbert $\cA$-module $l^2(\Gamma,\cA)$ with domain $\C(\Gamma)\ten \cA$ given by $D_l (a_\g\g)=l(\g)a_\g \delta_\g$. We denote by $B$ the smallest subalgebra of $\cB(l^2(\Gamma,\cA))$ containing $\C(\Gamma,\sigma)$ and multiplication by elements of $l^{\infty}(\Gamma,\cA)$. Let $\cB$ be its $C^*$-algebraic completion. Then 
$$\der \colon B\to \cB,~\der(T)=[D_l,T]$$
is a densely defined closable derivation. We define $\cB_j$ as above and set $C_j(\G,\s,\cA)=\cB_j\cap C^*_r(\Gamma,\sigma)\ten \cA$. 

The following proposition generalizes Prop. \ref{coninj}.

\begin{prop}
For $j\in \bn$ the identification of $\cA$-modules $\C(\Gamma,\sigma) \odot \cA \to \C(\G) \odot \cA$ extends to a continuous map $C_j(\G,\s,\cA) \to H^j_l(\Gamma,\cA)$. 
\end{prop}

\begin{proof}
The proof is essentially as above.
\end{proof}

For results on the completed $\ve$- and $\pi$-tensor products $\hat \ten_{\ve, \pi}$ used in the following, see \cite[Ch. 43]{tr}.

\begin{lem}
The identity on $\C\G \odot \cA$ extends to a continuous map $H_l^j(\Gamma,\cA) \to H_l^j(\Gamma)\hat\ten_{\ve} \cA$.
\end{lem}

\begin{proof}
The arguments from \cite[p. 494]{tr} show that $H_l^j(\Gamma)\hat\ten_{\ve} \cA$ is isomorphic to the closure of $\C\G \odot \cA$ in the Banach space $\cB(H_l^j(\Gamma)',\cA)$ of bounded operators from $H_l^j(\Gamma)'$ to $\cA$. Here $H_l^j(\Gamma)'$ denotes the dual Hilbert space. On the other hand there is a bounded map 
$$H_l^j(\Gamma,\cA) \cong H_l^j(\Gamma) \ten \cA \to \cB(H_l^j(\Gamma)',\cA) \ ,$$ which on $\C \G \odot\cA$ agrees with that isomorphism.
\end{proof}

\begin{cor}
If $\Gamma$ is of polynomial growth, then $H_l^{\infty}(\Gamma,\cA) \cong H_l^{\infty}(\Gamma)\hat\ten_{\pi} \cA$.
\end{cor}

\begin{proof}
By the previous lemma and the fact that the $\ve$-tensor product commutes with projective limits (see \cite[p. 282]{ko}) there are canonical maps $$H_l^{\infty}(\Gamma)\hat\ten_{\pi} \cA \to H_l^{\infty}(\Gamma,\cA) \to H_l^{\infty}(\Gamma)\hat\ten_{\ve} \cA \ .$$	
Since $H^{\infty}(\Gamma)$ is nuclear by \cite[Theorem 3.1.7]{jo}, the left and right hand side are isomorphic. The assertion follows now from the fact that the composition equals the identity on $\C\G \odot \cA$.
\end{proof}

For the following we choose a set of generators of $\G$ and $G$, respectively. This induced a word length function on $\G$, $G$ and $\G\times G$.

\begin{prop}
Let $\Gamma, G$ be finitely generated groups with $\Gamma$ of polynomial growth. Then $H_l^{\infty}(\Gamma) \hat\ten_{\pi} H_l^{\infty}(G) \cong H_l^{\infty}(\Gamma \times G)$.
\end{prop}

\begin{proof}
By the universal property of the projective tensor product there is a continuous map $H_l^{\infty}(\Gamma) \hat\ten_{\pi} H_l^{\infty}(G) \to H_l^{\infty}(\Gamma \times G)$.

We show that there is a continuous map $H_l^{\infty}(\Gamma \times G) \to \cK(H_l^k(\Gamma)',H_l^k(G))=H_l^k(\Gamma)\hat \ten_{\ve} H_l^k(G)$. Here $\cK$ denotes the Banach space of compact operators. Then the assertion follows from taking the projective limit and from the fact that $H_l^{\infty}(\Gamma)$ is nuclear.

Let $v \in H_l^k(\Gamma)'$ be a unit vector and $a=\sum_{(\gamma_1,\gamma_2)} a_{(\gamma_1,\gamma_2)} \delta_{(\gamma_1,\gamma_2)} \in \C(\Gamma \times G)$. 
We write $a_{\gamma_2}:=\sum_{\gamma_1 \in \Gamma} a_{(\gamma_1,\gamma_2)} \delta_{\gamma_1} \in \C\Gamma$.

Then $av:=\sum_{\gamma_2} v(a_{\gamma_2}) \delta_{\gamma_2} \in \C G$ fulfills

\begin{align*}
\|av\|^2_{H_l^k(G)} &= \sum_{\gamma_2} |v(a_{\gamma_2})|^2 (1+l(\gamma_2))^{2k} \\
&\le \sum_{\gamma_2}  \|a_{\gamma_2}\|^2_{H_l^k(\Gamma)}(1+l(\gamma_2))^{2k} \\
&\le \sum_{(\gamma_1,\gamma_2)}  |a_{(\gamma_1,\gamma_2)}|^2(1+l(\gamma_1))^{2k}(1+l(\gamma_2))^{2k} \\
&\le \sum_{(\gamma_1,\gamma_2)} |a_{(\gamma_1,\gamma_2)}|^2(1+l(\gamma_1)+l(\gamma_2))^{4k}\\
&=\|a\|^2_{H_l^{2k}(\Gamma \times G)} \ .
\end{align*}
\end{proof}

\begin{prop}
\label{productext}
Assume that $\Gamma$ has polynomial growth and let $c\in C^n_{al}(\G)$ be a group cocycle of polynomial growth. Let $\s_1$ be a multiplier on $\G$.
\begin{enumerate}
\item Let $\tr$ be a trace on $\cA$.
Then $\tau_c\#\tr$ extends to a continuous cyclic cocycle on $C_{\infty}(\G,\s_1,\cA)$.
\item Let $G$ be a finitely generated group and $\s_2$ a multiplier on $G$. Let $\tr$ be a trace on $\C(G,\s_2)$ which extends to a continuous map $H_l^{\infty}(G)\to \C$. We denote by $\pi_1 \colon \Gamma \times G \to \G$ and $\pi_2 \colon \Gamma \times G \to G$ the projections. 

Then $\tau_c\#\tr$ extends from $\C(\Gamma \times G, \pi_1^*\s_1\pi_2^*\s_2)$ to a continuous cyclic cocycle on $C_{\infty}(\G \times G,\pi_1^*\s_1 \pi_2^*\s_2)$. 
\end{enumerate}
\end{prop}

\begin{proof}
(1) The cocycle $\tau_c \#\tr$ defines a multilinear map from $(\C(\Gamma,\s) \times \cA)^{n+1}$ to $\C$. By taking the projective tensor product throughout we get that $\tau_c \#\tr: H^{\infty}_l(\Gamma,\cA)^{\ten_{\pi}^{n+1}} \to \C$ is a continuous linear map. 

(2) By Prop. \ref{coninj} and the previous proposition there is a continuous injection $C_{\infty}(\G \times G,\pi_1^*\s_1 \pi_2^*\s_2) \to H^{\infty}_l(\G) \hat \ten_{\pi} H^{\infty}_l(G)$. Now we get the conclusion as before from the fact that $\tau_c \#\tr$ defines a multilinear map from $(H^{\infty}_l(\G) \times H^{\infty}_l(G))^{n+1}$ to $\C$.
\end{proof}

\end{document}